\documentclass[11pt]{amsart}
\usepackage[foot]{amsaddr}
\makeatletter
\def\BState{\State\hskip-\ALG@thistlm}
\renewcommand{\email}[2][]{%
  \ifx\emails\@empty\relax\else{\g@addto@macro\emails{,\space}}\fi%
  \@ifnotempty{#1}{\g@addto@macro\emails{\textrm{(#1)}\space}}%
  \g@addto@macro\emails{#2}%
}
\makeatother

\usepackage[margin=1.5in]{geometry}

\usepackage{enumerate}
\usepackage{graphicx}
\usepackage{enumitem}
\usepackage{color}
\usepackage[hyphens,spaces]{url}
\usepackage{amssymb}
\usepackage{epstopdf}
\usepackage{romanbar}
\usepackage{stmaryrd}
\usepackage{algorithm2e}
\usepackage{MnSymbol}
\usepackage{multirow}
\usepackage{verbatim}
\newtheorem{thm}{Theorem}

\newtheorem{lemma}[thm]{Lemma}
\newtheorem{cor}[thm]{Corollary}

\def\restriction#1#2{\mathchoice
	{\setbox1\hbox{${\displaystyle #1}_{\scriptstyle #2}$}
		\restrictionaux{#1}{#2}}
	{\setbox1\hbox{${\textstyle #1}_{\scriptstyle #2}$}
		\restrictionaux{#1}{#2}}
	{\setbox1\hbox{${\scriptstyle #1}_{\scriptscriptstyle #2}$}
		\restrictionaux{#1}{#2}}
	{\setbox1\hbox{${\scriptscriptstyle #1}_{\scriptscriptstyle #2}$}
		\restrictionaux{#1}{#2}}}
\def\restrictionaux#1#2{{#1\,\smash{\vrule height .8\ht1 depth .85\dp1}}_{\,#2}} 

\newcommand{\vc}{\mathbf{c}}

\newcommand{\vm}{\mathbf{m}}
\newcommand{\vx}{\mathbf{x}}
\newcommand{\vy}{\mathbf{y}}

\newcommand{\vv}{\mathbf{v}}
\newcommand{\vw}{\mathbf{w}}
\newcommand{\vf}{\mathbf{f}}
\newcommand{\vn}{\mathbf{n}}

\newcommand{\vp}{\mathbf{p}}
\newcommand{\vvarphi}{\boldsymbol{\varphi}}
\newcommand{\vpsi}{\boldsymbol{\psi}}
\newcommand{\vphi}{\boldsymbol{\phi}}
\newcommand{\vnu}{\boldsymbol{\nu}}

\newcommand{\bz}{\boldsymbol{0}}

\newcommand{\vmu}{\boldsymbol{\mu}}
\newcommand{\vla}{\boldsymbol{\lambda}}
\newcommand{\vA}{\mathbf{A}}
\newcommand{\vB}{\mathbf{B}}
\newcommand{\vS}{\mathbf{S}}

\newcommand{\jump}[1]{[ #1 ]} 
\newcommand{\avrg}[1]{\left\{ #1 \right\}} 
\newcommand{\Th}{\mathcal{T}_h} 
\newcommand{\Eh}{\mathcal{E}_h} 
\newcommand{\E}{\mathcal{E}} 
\newcommand{\K}{T} 
\newcommand{\V}{\mathbb{V}} 
\newcommand{\h}{ {\rm h}} 
\newcommand{\Gh}{\Gamma_h}
 
\newcommand{\I}{\normalfont{\Romanbar{1}}} 
\newcommand{\II}{\normalfont{\Romanbar{2}}}

\newcommand{\tr}{{\rm tr}} 
\newcommand{\di}{\mathop{\rm div}\nolimits} 
 
\newcommand{\A}{\mathbb{A}}

\newcommand{\Ccal}{\mathcal{C}}

\title[LDG approximation of large deformations of bilayer plates]{Gamma-convergent LDG method for large bending deformations of bilayer plates}
\author{Andrea Bonito}
\address[Andrea Bonito]{Department of Mathematics, Texas A\&M University, College Station, TX 77845, USA. AB was partially supported by the NSF Grant DMS-2110811.}
\email{bonito@tamu.edu}
\author{Ricardo H. Nochetto}
\address[Ricardo H. Nochetto]{Department of Mathematics and Institute for Physical Science
		and Technology \\ University of Maryland,
		College Park, Maryland 20742, USA. RHN and SY were partially supported by the NSF Grant DMS-1908267.}
\email{rhn@umd.edu}
\author{Shuo Yang}
\address[Shuo Yang]{Yanqi Lake Beijing Institute of Mathematical Sciences and Applications, Beijing 101408, China.}
\email{shuoyang@bimsa.cn}

\date{\today}
\begin{document}
\maketitle

\begin{abstract}

Bilayer plates are slender structures made of two thin layers of different materials. They react to environmental stimuli and undergo large bending deformations with relatively small actuation. The reduced model is a constrained minimization problem for the second fundamental form, with a given spontaneous curvature that encodes material properties, subject to an isometry constraint. We design a local discontinuous Galerkin (LDG) method which imposes a relaxed discrete isometry constraint and controls deformation gradients at barycenters of elements. We prove $\Gamma$-convergence of LDG, design a fully practical gradient flow, which gives rise to a linear scheme at every step, and show energy stability and control of the isometry defect. We extend the $\Gamma$-convergence analysis to piecewise quadratic creases. We also illustrate the performance of the LDG method with several insightful simulations of large deformations, one including a curved crease.

\end{abstract}

\section{Introduction}\label{sec:intro}
Bilayer plates are slender structures made of two thin layers of different materials glued together. These layers react differently to non-mechanical stimuli, such as thermal, electrical, and chemical actuation \cite{janbaz2016programming,sodhi2010modeling,kim2008stretchable}. Bilayer plates can undergo large bending deformations using a small amount of energy, which makes them appealing at small and large scales alike. Amongst the many and broad applications of bilayer materials in engineering and biomedical science, we list drug delivery vesicles \cite{guan2005self,stoychev2011self}, cell encapsulation devices \cite{stoychev2012shape}, sensors \cite{mano2008stimuli} and self-deployable sun sails \cite{love2007demonstration}.    

We model bilayer plates as thin 3d hyper-elastic bodies as depicted in Fig. \ref{fig:thindomain}. Exploiting their relatively small thickness, two dimensional plate models for the mid-plane deformation $\vy(\Omega)$, $\Omega \subset \mathbb R^2$, are derived and analyzed in \cite{schmidt2007minimal,schmidt2007plate}; we also refer to \cite{bartels2017bilayer} for a formal dimension reduction argument. The plates equilibria are characterized as solutions to a nonlinear minimization problem with a nonconvex constraint expressing the plates ability to bend without stretching or shearing. Therefore, distances within the midplane remain unchanged thereby resulting in isometric deformations.  

\begin{figure}[htbp]
	\begin{center}
		\includegraphics[width=8cm]{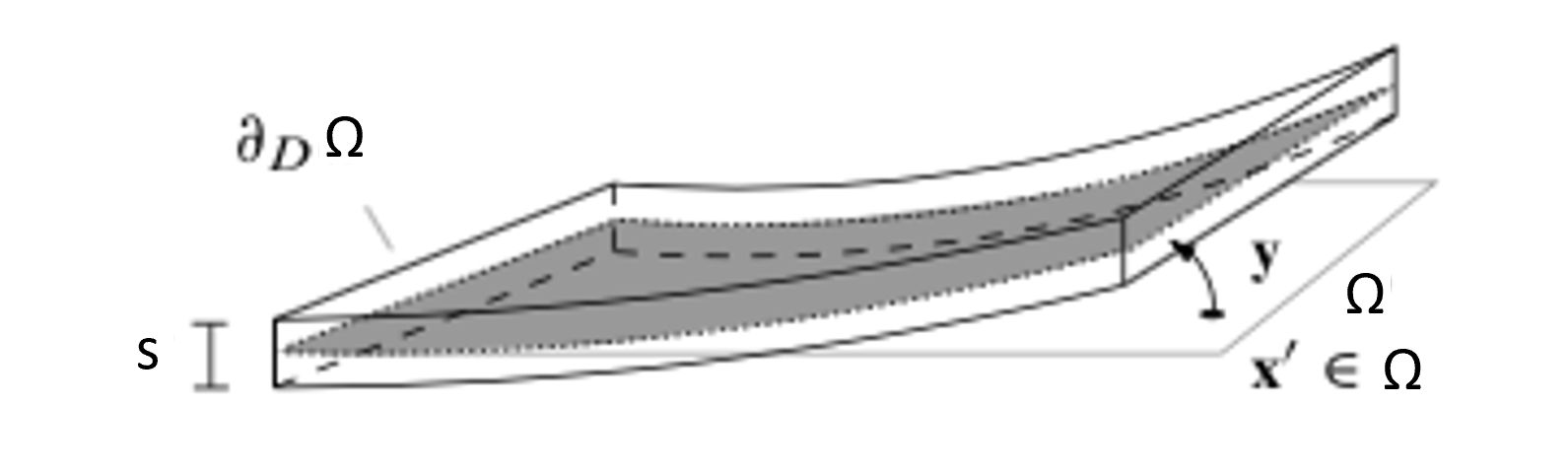}
		\caption{Bilayer plates: $\Omega\times(-s/2,s/2)$, $\Omega\subset\mathbb{R}^2$ is the mid-plane (bounded Lipschitz domain) and $s$ is the thickness parameter. The sets $\Omega\times(-s/2,0)$ and $\Omega\times(0,s/2)$ represent the two undeformed layers of different materials.}
		 \label{fig:thindomain}
	\end{center}
\end{figure}

\begin{subsection}{Problem statement}
  
The plate deformation $\vy:\Omega\to\mathbb{R}^3$ must belong to the following {\it admissible set} $\A$, which prevents shearing and stretching within the surface $\vy(\Omega)$ and imposes possible boundary conditions:
\begin{equation}\label{admissible}
\A:=\big\{\mathbf{y}\in[H^2(\Omega)]^3:\quad\I[\vy]=I_2\ \text{ in }\Omega,\quad \mathbf{y}=\vvarphi, \ \nabla\mathbf{y}=\Phi\text{ on }\Gamma^D\big\},
\end{equation}
where $I_2$ is the $2\times2$ identity matrix and
\begin{equation}\label{first-fund-form}
  \I[\vy]:=\nabla\mathbf{y}^T\nabla\mathbf{y}
\end{equation}
is the first fundamental form of $\vy(\Omega)$. We assume that $\Gamma^D\subset\partial\Omega$ is nonempty and open and $\vvarphi\in [H^2(\Omega)]^3$ and $\Phi \in [H^1(\Omega)]^{3\times2}$ are given and are compatible with the isometry constraint, namely $\Phi=\nabla\vvarphi$ and $\Phi^T\Phi=I_2$ on $\Gamma_D$; thus $\A$ is non-empty. Moreover, condition \eqref{first-fund-form} entails that $\{\partial_i\vy\}_{i=1}^2$ is an orthonormal basis of the tangent plane to $\vy(\Omega)$ and its unit normal $\vnu$ can be written as
\begin{equation}\label{normal}
\vnu := \frac{\partial_1\vy\times\partial_2\vy}{|\partial_1\vy\times\partial_2\vy|}=\partial_1\vy\times\partial_2\vy.
\end{equation}

Although we will present simulations in Section \ref{sec:numerical} for both Dirichlet boundary conditions (i.e.  $\Gamma^D\ne\emptyset$) and free boundary conditions (i.e. $\Gamma^D=\emptyset$), we focus our presentation on the former for convenience. We emphasize that the analysis of the latter follows from that in this paper. The modifications are in the spirit of \cite{bonito2020ldg-na}, where we analyze the LDG method for prestrained plates with free boundary conditions. Consequently, we do not include details to avoid repetitions.

Equilibrium configurations of bilayer plates are solutions $\vy\in\A$ of the following constrained minimization problem
\begin{equation}\label{minimization}
\min\limits_{\mathbf{y}\in\A}E\left[\mathbf{y}\right]:= \min\limits_{\mathbf{y}\in\A} \frac{1}{2}\int_{\Omega}\big|\II[\vy]-Z\big|^2,
\end{equation}
where $\II[\vy]$ is the second fundamental form of $\vy(\Omega)$
\begin{equation}\label{second-fund-form}
\II[\vy] := \big( \partial_{ij} \vy \cdot \vnu \big)_{ij=1}^2 = \big( \partial_{ij} \vy \cdot (\partial_1\vy\times\partial_2\vy) \big)_{ij=1}^2,
\end{equation}
and $Z\in[L^{\infty}(\Omega)]^{2\times2}$ is a \emph{spontaneous curvature} which encodes the material properties of the bilayer plates. In fact, $Z$ forces the plate $\vy(\Omega)$ to bend so that $\II[\vy]$ gets as close as possible to $Z$. If the material is homogenous and isotropic, then the spontaneous curvature is diagonal, i.e. $Z=\alpha I_2$ with a constant $\alpha$ depending on the materials parameters. In particular, when the two layers are identical, $Z=\bz$ and the model reduces to a \emph{single layer plate} \cite{bartels2013approximation,bonito2019dg}, which coincides with the classical (nonlinear) Kirchhoff plate theory. 

Thanks to the isometry constraint $\I[\vy] = I_2$, the energy functional $E\left[\mathbf{y}\right]$ can be further simplified. 
Recall that for isometries, there holds \cite{bartels2015numerical}
\begin{equation}\label{identity-isometry}
\big|\II[\vy]\big|^2=\big|D^2\vy\big|^2=\big|\Delta\vy\big|^2=\big(\tr \,\II[\vy]\big)^2,
\end{equation}
whence expanding the square in \eqref{minimization} and using \eqref{second-fund-form} and \eqref{identity-isometry}
yields
\begin{equation}\label{energy-2}
E\left[\mathbf{y}\right] = \frac{1}{2}\int_{\Omega}\big|D^2\vy\big|^2-\sum_{i,j=1}^2\int_{\Omega}\partial_{ij}\vy\cdot(\partial_1\vy\times\partial_2\vy)Z_{ij}+\frac12\int_{\Omega}\big|Z\big|^2.
\end{equation}
Furthermore, since $\frac12\int_{\Omega}\big|Z\big|^2$ does not depend on $\vy$, minimizing the energy in \eqref{energy-2} over $\A$ is equivalent to minimizing the reduced energy
\begin{equation}\label{energy-3}
E\left[\mathbf{y}\right]:=\frac{1}{2}\int_{\Omega}\big|D^2\vy\big|^2-\sum_{i,j=1}^2\int_{\Omega}\partial_{ij}\vy\cdot(\partial_1\vy\times\partial_2\vy)Z_{ij},
\end{equation}
over $\A$; we keep the same notation for the energies in \eqref{energy-2} and \eqref{energy-3} for simplicity. 
The effect of the layers mismatch appears in the cubic term leading to a nonlinear Euler-Lagrange equation for the equilibrium deformation $\vy$, 
namely
\begin{equation}\label{eq:energy-E-L}
\begin{aligned}
0=\delta E\left[\mathbf{y}\right](\vv)&:=\int_{\Omega}D^2\vy:D^2\vv-\sum_{i,j=1}^2\int_{\Omega}\partial_{ij}\vv\cdot(\partial_1\vy\times\partial_2\vy)Z_{ij} \\ 
&-\sum_{i,j=1}^2\int_{\Omega}\partial_{ij}\vy\cdot(\partial_1\vv\times\partial_2\vy)Z_{ij}-\sum_{i,j=1}^2\int_{\Omega}\partial_{ij}\vy\cdot(\partial_1\vy\times\partial_2\vv)Z_{ij},
\end{aligned}
\end{equation}
where $\vv\in[H^2_0(\Omega)]^3$ is an arbitrary test function. 
For later use,
we also introduce a notation for the single layer bending energy
\begin{equation}\label{energy-quad}
B[\vy]:=\frac{1}{2}\int_{\Omega}\big|D^2\vy\big|^2
\end{equation}
and the cubic term
\begin{equation}\label{energy-nonlinear}
C[\vy]:=\sum_{i,j=1}^2\int_{\Omega}\partial_{ij}\vy\cdot(\partial_1\vy\times\partial_2\vy)Z_{ij},
\end{equation}
so that
$$
E[\vy] = B[\vy] - C[\vy].
$$
We emphasize that the cubic term $C$ satisfies 
$$
|C[\vy]| \leq \| \vy \|_{H^2(\Omega)} \| \nabla \vy \|_{L^2(\Omega)}  \| \nabla \vy \|_{L^\infty(\Omega)}\|Z\|_{L^{\infty}(\Omega)}
$$
and  $\I[\vy]=I_2$ implies $\|\nabla\vy\|_{L^\infty(\Omega)} \lesssim \|\I[\vy]\|_{L^\infty(\Omega)} \lesssim 1$, whence
$$
|C[\vy]| \lesssim \| \vy \|_{H^2(\Omega)}^2.
$$
Since the discrete deformation $\vy_h$ is piecewise polynomial, our numerical method cannot guarantee that $\vy_h$ satisfies the isometry constraint $\I[\vy_h]=I_2$ everywhere in $\Omega$. 
We choose to enforce a relaxed constraint solely at the barycenter of elements.
This is a chief ingredient of our LDG method and is inspired by Bartels and Palus
for Kirchhoff elements \cite{bartels2020stable}.
\end{subsection}

\begin{subsection}{Previous numerical methods}\label{ss:previous}
There are several finite element methods available for the numerical simulation of bilayers plates
\cite{bartels2017bilayer,bartels2018modeling,bartels2020stable,bonito2020discontinuous}.
In all of them, the isometry constraint $\I[\vy] = I_2$ is linearized at $\vy$
\begin{equation}\label{linearization}
L[\vv;\vy] := \nabla \vv^T \nabla \vy + \nabla \vy^T \nabla \vv = \bz,
\end{equation}
and tangential variations $\vv$ are evolved within a gradient flow that decreases the energy $E[\vy]$ and is
favored for its robustness. 

The gradients of Kirchhoff finite elements are uniquely defined at the mesh vertices, which is where
\eqref{linearization} is imposed in \cite{bartels2018modeling,bartels2017bilayer}. The discrete gradient flow
in \cite{bartels2017bilayer} treats the cubic energy $C[\vy]$ implicitly to get an energy decreasing scheme
but requires the normalization \eqref{normal} of the discrete normal, which renders the algorithm nonlinear.
Discrete energies are shown to $\Gamma$-converge in \cite{bartels2017bilayer}. In contrast, the scheme of
\cite{bartels2018modeling} is linear and much more efficient, but stability and $\Gamma$-convergence are still
open. Recently, Bartels and Palus \cite{bartels2020stable} reformulated the discretization of $C[\vy]$ making
it fully explicit and the ensuing algorithm linear, and were also able to show an energy decreasing property for
the explicit gradient flow with a mild time-step constraint and $\Gamma$-convergence of the discrete energies.

On the other hand, interior penalty discontinuous Galerkin (IPDG) finite element methods are proposed and studied in \cite{bonito2020discontinuous} because they require a lower polynomial degree (2 instead of 3), are easier to find in existing software platforms, are more flexible in imposing boundary conditions as well as the linearized isometry constraint \eqref{linearization}, and are amenable to subdivisions containing curved boundaries which is crucial to deal with creases. The linearized constraint \eqref{linearization} is enforced in average on all elements of the subdivision. 
Furthermore, the cubic energy $C[\vy]$ is treated explicitly at each step of the discrete gradient flow and the ensuing algorithm is linear. However, $\Gamma$-convergence and energy decreasing properties remain open problems.

We note that the bilayer model \eqref{energy-2} reduces to single layer plates endowed with the bending energy $B[\vy]$ for $\vy \in \A$ provided the upper and lower layers are identical, i.e. $Z=0$. We refer to \cite{bartels2013approximation,bonito2019dg} for the design and analysis of Kirchhoff and IPDG methods in this simpler context.

\end{subsection}

\begin{subsection}{LDG-discretization and our contribution}

We propose a local discontinuous Galerkin (LDG) method for the approximation of the minimization problem \eqref{minimization} along the lines of \cite{bonito2020ldg,bonito2020ldg-na}. LDG method was originally introduced in \cite{cockburn1998local}, and further explored in \cite{bassi1997high,brezzi1999discontinuous,brezzi2000discontinuous,di2010discrete,di2011mathematical}. Our discrete energy $E_h[\vy_h]$ is obtained (up to stabilization terms) by simply replacing the Hessian $D^2\vy$ in \eqref{energy-3} by a discrete Hessian $H_h[\vy_h]$ (defined by \eqref{def:discrHess_bi} below), which is constructed and analyzed in \cite{bonito2020ldg,bonito2020ldg-na} in terms of the discontinuous Galerkin solution $\vy_h$. This is conceptually simpler than IPDG methods, which are based on integration by parts and are harder to design for intricate nonlinear systems. In contrast to IPDG, LDG is also stable for any positive stabilization parameters, and exhibits better convergence properties at the expense of a slightly worse sparsity pattern \cite{bonito2020ldg,bonito2020ldg-na}.

Our treatment of the cubic term hinges on the mid-point quadrature. If $\Th$ is a mesh made of shape-regular triangles or
quadrilaterals $\K$ with barycenter $x_T$, let 
\begin{equation}\label{quadrature}
C_h[\vy_h] := \sum_{i,j=1}^2 \sum_{\K\in\Th} |\K| \Big( \overline{H}_h[\vy_h]_{ij} \cdot (\partial_1\vy_h \times \partial_2\vy_h) Z_{ij}\Big)(x_\K)
\end{equation}
where $\overline{H}_h[\vy_h] = \frac{1}{|\K|}\int_T H_h[\vy_h]$ for all $\K\in\Th$ is the piecewise constant
reduced discrete Hessian. 
We also make the simplifying assumption that the spontaneous curvature $Z$ in \eqref{minimization} is piecewise constant over all partitions $\Th$, $h>0$.
Moreover, we control the isometry defect at barycenters, namely given a
parameter $\delta>0$ we impose
\begin{equation}\label{isometry-defect}
  D_h[\vy_h](x_T) := \big| [\nabla\vy_h^T \nabla\vy_h - I_2](x_T) \big| \le \delta
  \quad\forall \, \K\in\Th.
\end{equation}
We enforce the Dirichlet condition upon augmenting the discrete energy $E_h[\vy_h]$ via a Nitsche method. Therefore, we say that discrete functions satisfying \eqref{isometry-defect} belong to the discrete admissible set $\A_{h,\delta}$, the discrete counterpart of {$\A$} in \eqref{admissible}. We prove that {$\A_{h,\delta}$,} is non-empty, and derive convergence of global minimizers $\vy_h$ of $E_h$ within {$\A_{h,\delta}$,} towards global minimizers $\vy$ of \eqref{minimization} {in the spirit of $\Gamma$-convergence.}

{It is worth pointing out that $\Gamma$-convergence does not give error estimates and that, except for \cite{BaBoTs:23} for linear plates with folding, we are not aware of such bounds for nonlinear plates undergoing large deformations. The main obstructions are: the energy is nonconvex; the isometry constraint is nonconvex; there might be multiple solutions; there is no regularity theory beyond the basic energy estimate $\vy \in [H^2(\Omega)]^3$; mapping properties of the linearized Euler-Lagrange equation and isometry constraint, that govern the behavior of perturbations, have to be discovered and most likely will entail additional regularity of $\vy$; no monotonicity argument is available because $\vy$ is vector-valued. However, it is plausible that error estimates are valid for small perturbations of smooth branches of solutions. Proving error estimates is a challenging and important endeavor, but is far beyond the scope of this paper.
}
  
Solving the {\it nonconvex} discrete minimization counterpart of \eqref{minimization} is a highly nontrivial task. We resort to a discrete gradient flow that enforces the linearized isometry constraint \eqref{linearization} at the barycenters {$x_T$ of elements $T$}
\begin{equation}\label{linearization-discrete}
  L[\vv_h;\vy_h](x_\K) := [\nabla \vv_h^T \nabla \vy_h + \nabla \vy_h^T \nabla \vv_h](x_\K)=\bz
  \quad\forall \, \K\in\Th,
\end{equation}
and solve a discrete minimization problem for a tangential variation $\vv_h$ of $\vy_h$, in the sense \eqref{linearization-discrete}, with the cubic term \eqref{quadrature} treated explicitly. The latter is a clever idea, due to Bartels and Palus \cite{bartels2020stable}, that renders the problem linear at each step of the gradient flow; however, our approach is different from \cite{bartels2020stable}. We show that this procedure is energy decreasing, convergent, and preserves the isometry defect \eqref{isometry-defect} provided $\delta$ is
proportional to $h$, which entails a linear relation between the time step $\tau$ of the gradient flow and $h$.
Moreover, we derive a (suboptimal) discrete inf-sup condition for the Lagrange multiplier approach to the linear constraint \eqref{linearization-discrete}, which seems to be the first such result for this type of matrix constraint and is consistent with computations.

The rest of this article is organized as follows. Section \ref{sec:Discretization} is about LDG.
We introduce the (\emph{broken}) finite element spaces in Section \ref{subsec:broken}. We examine the discrete Hessian operator and its reduced counterpart in Subsection \ref{subsec: DH}, together with their boundedness and convergence properties. In Subsections \ref{subsec: D-admissible} and \ref{s:discrete_min}, we define the discrete problem and investigate consistency of the cubic discrete energy $C_h$. The proof of $\Gamma$-convergence of the discrete energy to the exact one is the content of Section \ref{Gamma-bilayer}, and its extension to a bilayer model with piecewise quadratic creases is included in Section \ref{S:creases}. In Section \ref{sec: gf-bilayer}, we introduce the gradient flow scheme used to solve the discrete problem, prove its conditional stability and show how the constraint violation \eqref{isometry-defect} is controlled throughout the flow. Moreover, we derive a suboptimal inf-sup condition for \eqref{linearization-discrete} at each step of the flow. We present several insightful simulations in Section \ref{sec:numerical} to illustrate the performance of LDG, including folding across a curved crease. 
\end{subsection}

\section{LDG Discretization}\label{sec:Discretization}

\begin{subsection}{Subdivisions} \label{subsec:mesh}
From now on, we assume that $\Omega \subset \mathbb R^2$ is a polygonal domain and denote by $\{\Th\}_{h>0}$ a shape-regular sequence of conforming partitions of $\Omega$ made of either triangles or quadrilaterals $\K$ with diameter $h_{\K} := \textrm{diam}(\K)\leq h$.
The set of edges $\Eh:=\Eh^0\cup\Eh^b$ is decomposed into the interior edges $\Eh^0$ and boundary edges $\Eh^{b}$. 
For $e \in \E_h$, we define $h_e:= \textrm{diam}(e)$ and note that $h_e\le h$, and thus
\begin{equation} \label{eqn:he_h}
h^{-1}\le h_e^{-1} \quad \forall e\in\Eh.
\end{equation}
We assume a compatible representation of the Dirichlet boundary $\Gamma^D= \cup \{e: e\in \Eh^{D}\}$,
and let $\Eh^a:=\Eh^0\cup\Eh^{D}$ be the set of {\it active edges} on which jumps and averages will be computed. The union of these edges gives rise to the corresponding skeletons of $\Th$
\begin{equation}\label{E:skeleton}
\Gh^0 := \cup \big\{e: e\in\Eh^0 \big\},
\quad
\Gh^D := \Gamma^D,
\quad
\Gh^a := \Gh^0 \cup \Gh^D.
\end{equation}
We use the notation $(\cdot,\cdot)_{L^2(\Omega)}$ and $(\cdot,\cdot)_{L^2(\Gh^a)}$ to denote the $L^2$ inner products over $\Omega$ and $\Gh^a$, and a similar notation for subsets of  $\Omega$ and $\Gh^a$. We denote by $\h$ a mesh density function, locally equivalent to $h_T$ and $h_e$, and utilize it as a weight in the preceding norms. 
We often write $f \lesssim g$ to indicate that there exists a constant $C$ independent of discretization parameters such that $f \leq C g$. 

\end{subsection}

\begin{subsection}{Broken spaces and operators}\label{subsec:broken}

For an integer $r\geq 0$, we denote by $\mathbb{P}_r$  the space of polynomials of total degree at most $r$ when the subdivision is made of triangles and by $\mathbb{Q}_r$ the space of polynomials of degree at most $r$ in each variable when quadrilaterals are used.
We also use the same notation, $\widehat\K$, to denote either the unit triangle or the unite square depending on the type of subdivision used.
We let $F_\K:\widehat\K \rightarrow \K \in [\mathbb{Q}_1]^2$ be the generic map from the reference element to the physical element. 
It is affine only when the subdivision is made of triangles.

We fix {the polynomial degree} $k \ge 2$. The (\textit{broken}) finite element space $\V_h^k$ to approximate each component of the deformation $\vy$ reads
\begin{equation} \label{def:Vhk_tri}
\V_h^k:=\left\{v_h\in L^2(\Omega): \,\, \restriction{v_h}{\K}\circ F_{\K}\in\mathbb{Q}_k \quad \forall \K \in\Th \right\},
\end{equation}
when the subdivision is made of quadrilaterals, and we replace $\mathbb{Q}_k$ by $\mathbb{P}_k$ if we have triangular elements.
We define the broken gradient $\nabla_h v_h$ of $v_h\in\V_h^k$ to be the elementwise gradient, and use similar notation for other differential operators.
For instance $D_h^2 v_h=\nabla_h\nabla_h v_h$ stands for the broken Hessian, and $\partial_i v_h:= \partial_{i,h} v_h$
denotes the components $i=1,2$ of the broken gradient $\nabla_h v_h$.

We now introduce the jump and average operators. For every $e\in\Eh^0$, fix $\vn_e$ to be one of the two unit normals to $e$ (the choice is arbitrary but does not affect the formulation). For a boundary edge $e\in\Eh^b$, we set $\vn_e = \vn$, the outward unit normal vector to $\partial\Omega$. The jump of $v_h \in \V_h^k$ and $\nabla_h v_h$ across $e \in \E_h^0$ are given by
\begin{equation} \label{def:jump}
\restriction{\jump{v_h}}{e} := v_h^{-}-v_h^+,
\quad
\restriction{\jump{\nabla_h v_h}}{e} := \nabla_h v_h^{-} - \nabla_h v_h^+,
\quad
\end{equation}
where $v_h^{\pm}(\vx):=\lim_{s\rightarrow 0^+}v_h(\vx\pm s\vn_e)$ for $\vx \in e$. The jumps of a vector or matrix valued function are computed componentwise.

In order to incorporate the Dirichlet boundary conditions $\vy = \vvarphi$, $\nabla \vy = \Phi$ on $\Gamma^D$, we resort to a Nitsche's approach which does not impose essential restrictions on the discrete space $[\V_h^k]^3$ but rather modifies the discrete formulation by including boundary jumps defined for $\vv_h \in [\V_h^k]^3$
\begin{equation}\label{E:bd-jumps}
\restriction{[\vv_h]}{e}:=\restriction{[\vv_h]}{e}(\vvarphi) := \vv_h - \vvarphi,
\quad
\restriction{[\nabla_h \vv_h]}{e}  :=\restriction{[\nabla_h \vv_h]}{e}(\Phi):=\nabla_h \vv_h - \Phi,
\end{equation}
for all $e\in\Eh^D$.
However, to simplify the notation, it is convenient to introduce the discrete set $\V_h^k(\vvarphi,\Phi)$
\begin{equation}\label{discrete-set}
\V_h^k (\vvarphi,\Phi) := \Big\{ \vv_h\in [\V_h^k]^3: \
\restriction{[\vv_h]}{e}, \, \restriction{[\nabla_h \vv_h]}{e} \text{ given by \eqref{E:bd-jumps} for all } e\in\Eh^D \Big\},
\end{equation}
which coincide with $[\V_h^k]^3$ but carries the notion of boundary jump \eqref{E:bd-jumps} for its elements.
We define the {\it average} of $v_h \in \V_h^k$ across an edge $e\in \Eh$ as 
\begin{equation} \label{def:avrg}
\restriction{\avrg{v_h}}{e} := 
\left\{\begin{array}{ll}
\frac{1}{2}(v_h^{+}+v_h^{-}) & e\in\Eh^0 \\
v_h^{-} & e\in\Eh^b,
\end{array}\right. 
\end{equation}
and apply \eqref{def:avrg} componentwise to vector and matrix-valued functions.
\end{subsection}

We let $\langle\cdot,\cdot\rangle_{H_h^2(\Omega)}$ be the following mesh-dependent form defined, for any $\vv_h,\vw_h\in \V_h^k(\vvarphi,\Phi)$, by
\begin{equation}\label{def:H2norm}
\begin{aligned}
  \langle\vv_h,\vw_h\rangle_{H_h^2(\Omega)} & :=  (D^2_h\vv_h,D^2_h\vw_h)_{L^2(\Omega)}
  \\& +(\h^{-1}\jump{\nabla_h\vv_h},\jump{\nabla_h\vw_h})_{L^2(\Gh^a)}
  +(\h^{-3}\jump{\vv_h},\jump{\vw_h})_{L^2(\Gh^a)}.
\end{aligned}
\end{equation}
We emphasize that \eqref{def:H2norm} is not bilinear in $\V_h^k(\vvarphi,\Phi)$ because of the presence of $(\vvarphi,\Phi)$ in the boundary jump terms, unless $\vvarphi=\mathbf{0},\Phi=\mathbf{0}$. Moreover, we set
\begin{equation}\label{def:realH2norm}
  \|\vv_h\|_{H_h^2(\Omega)}^2:=\langle\vv_h,\vv_h\rangle_{H_h^2(\Omega)}
  \quad\forall \, \vv_h \in \V_h^k(\vvarphi,\Phi),
\end{equation}
and observe the validity of the following Friedrichs-type inequality
\cite[(2.27)]{bonito2019dg}
\begin{equation}\label{friedrichs}
  \|\vv_h\|_{L^2(\Omega)} + \|\nabla_h\vv_h\|_{L^2(\Omega)} \lesssim \|\vv_h\|_{H_h^2(\Omega)}
  + \|\vvarphi\|_{H^1(\Omega)} + \|\Phi\|_{H^1(\Omega)}
  \quad\forall \, \vv_h \in \V_h^k(\vvarphi,\Phi).
\end{equation}
Once restricted to $\V_h^k(\mathbf{0},\mathbf{0})$, the form $\langle\cdot,\cdot\rangle_{H_h^2(\Omega)}$
turns out to be a scalar product, according to \eqref{friedrichs},
which corresponds to the discrete counterpart of $\langle\cdot,\cdot\rangle_{H^2(\Omega)}$.

\begin{subsection}{Discrete Hessians}\label{subsec: DH}
The central ingredient in the proposed LDG approximation is the reconstructed Hessian $H_h[\vy_h]\in\left[L^2(\Omega)\right]^{3\times 2\times 2}$ defined in \cite{bonito2020ldg,bonito2020ldg-na}. 
Let $l_1,l_2\ge0$ be integers and consider two {\it local lifting operators} $r_e:[L^2(e)]^2\rightarrow[\V_h^{l_1}]^{2\times 2}$ and $b_e:L^2(e)\rightarrow[\V_h^{l_2}]^{2\times 2}$ defined for $e\in\Eh^a$ by
%
\begin{gather} 
r_e(\vphi) \in [\V_h^{l_1}]^{2\times 2}: \,
\int_{\omega_e}r_e(\vphi):\tau_h = \int_e\avrg{\tau_h}\vn_e\cdot\vphi \quad \forall \tau_h\in [\V_h^{l_1}]^{2\times 2}\label{def:lift_re},
\\
b_e(\phi) \in [\V_h^{l_2}]^{2\times 2}: \,
\int_{\omega_e} b_e(\phi):\tau_h = \int_e\avrg{\di \tau_h}\cdot\vn_e\phi \quad \forall \tau_h\in [\V_h^{l_2}]^{2\times 2} \label{def:lift_be} \, ,
\end{gather}
where $\omega_e$ is the union of the two elements of $\Th$ sharing $e\in\Gh^0$ or the element of $\Th$ having $e\in\Gh^b$ as part of its boundary. The definitions extend to $\left[[L^2(e)]^2\right]^3$ and $[L^2(e)]^3$ by component-wise application. The corresponding {\it global lifting operators} are then given by
\begin{equation}\label{E:global-lifting}
\begin{split}
\mathcal{R}_h &:= \sum_{e\in\Eh^a} r_e : [L^2(\Gh^a)]^2 \rightarrow [\V_h^{l_1}]^{2\times 2}, \\
\mathcal{B}_h &:= \sum_{e\in\Eh^a} b_e : L^2(\Gh^a) \rightarrow [\V_h^{l_2}]^{2\times 2}.
\end{split}
\end{equation}

Their purpose is to lift inter-element information to the cells so that once added to the piecewise Hessian $D^2_h$, they constitute a weakly convergent approximation of the exact Hessian (see Lemma \ref{weak-conv}).
In fact, we define the discrete Hessian operator $H_h :\V_h^k(\vvarphi,\Phi)\rightarrow\left[L^2(\Omega)\right]^{3\times 2\times 2}$ by
\begin{equation} \label{def:discrHess_bi}
H_h[\vv_h] :=  D_h^2 \vv_h - \mathcal{R}_h(\jump{\nabla_h\vv_h}) + \mathcal{B}_h(\jump{\vv_h}).
\end{equation}
We point out the implicit dependence on data $(\vvarphi,\Phi)$ and that we will later compute $H_h[\vv_h]$ for $\vv_h\in \V_h^k(\bz,\bz)$, i.e. $\vvarphi = \bz$, $\Phi=\bz$, slightly abusing notation.
Thanks to the relation between the edge and cell diameter \eqref{eqn:he_h}, we have the following a priori upper bounds for lifting operators 
\begin{equation}\label{eqn:hessian_upper}
\|H_h[\vv_h]\|_{L^2(\Omega)}\lesssim ||\vv_h||_{H_h^2(\Omega)}.
\end{equation}
Moreover, we have the following properties of the discrete Hessian $H_h[\vv_h]$. 

\begin{lemma}[weak convergence of $H_h$]\label{weak-conv}
Let $k \geq 2$ and $\vv_h\in\V_h^k(\vvarphi,\Phi)$. If $||\vv_h||_{H_h^2(\Omega)}\lesssim 1$ and $\vv_h\to\vv\in[H^2(\Omega)]^3$ in $[L^2(\Omega)]^3$ as $h\rightarrow 0$, then for any polynomial degree $l_1,l_2\ge0$ we have
\begin{equation} \label{eqn:weakH}
H_h[\vv_h]\rightharpoonup D^2 \vv \quad \mbox{in } \left[L^2(\Omega)\right]^{3\times 2\times 2} \quad \mbox{as } h\rightarrow 0.
\end{equation} 
\end{lemma}
\begin{proof}  
See \cite[Lemma~2.4 and Appendix B]{bonito2020ldg-na}.
\end{proof}

\begin{lemma}[strong convergence of $H_h$]\label{strong-conv}
Let $\vv\in[H^2(\Omega)]^3$ be any function such that $\vv=\vvarphi$ and $\nabla\vv=\Phi$ on $\Gamma^D$. Moreover,  let $\vv_h\in\V^k_h(\vvarphi,\Phi)$ satisfy
\begin{equation}\label{eqn:strong-cond}
  \|D^2\vv_h\|_{L^2(T)}\lesssim\|\vv\|_{H^2(T)}\,\,\forall T\in\Th,
  \quad
  \sum_{T\in\Th}\|\vv_h-\vv\|_{H^2(T)}^2\to0\text{ as }h\to0^+.
\end{equation}
Then for any polynomial degree $l_1,l_2\ge 0$ we have as $h\rightarrow 0^+$
\begin{equation} \label{eqn:H_strong}
H_h[\vv_h]\to D^2\vv \quad \mbox{strongly in } \,\, [L^2(\Omega)]^{3\times 2\times 2}.
\end{equation} 	
\end{lemma}
\begin{proof}
This is a minor modification of \cite[Lemma~2.5 and Appendix B]{bonito2020ldg-na}, which {assumes} that
$\vv_h$ is the Lagrange interpolant of $\vv\in H^2(\Omega)$.
\end{proof}

For later use, we now discuss properties of the \emph{reduced discrete Hessian} defined {as} the  local $L^2$ projection onto the space of piecewise constants, i.e.
\begin{equation}\label{modified-hessian}
  \overline{H}_h[\vv_h]|_T:= \frac{1}{|T|}\int_T H_h[\vv_h]
  \quad\forall \, T\in \Th.
\end{equation}
We start with the stability of $\overline{H}_h[\vy_h]$.
\begin{lemma}[{stability of $\overline{H}_h[\vv_h]$}]\label{L2bound-modified}
For any $\vv_h\in\V_h^k(\vvarphi,\Phi)$, there holds
\begin{equation}\label{e:L2boundreduce}
  \|\overline{H}_h[\vv_h]\|_{L^2(\Omega)}\le c_{stab}\|\vv_h\|_{H_h^2(\Omega)},
\end{equation}
where the constant $c_{stab}$ is independent of $h$.
\end{lemma}
\begin{proof}
This result is a direct consequence of the stability of the reconstructed Hessian \eqref{eqn:hessian_upper} and the local $L^2$ projection \eqref{modified-hessian}.
\end{proof}

The reduced discrete Hessian is also weakly converging.
\begin{lemma}[{weak convergence of $\overline{H}_h[\vv_h]$}]\label{weakconv-modified}
Let $\vv_h\in\V_h^k(\vvarphi,\Phi)$ be a sequence of discrete deformations satisfying  $\|\vv_h\|_{H_h^2(\Omega)}\lesssim 1$ for all $h$ and such that $\vv_h\to\vv$ in $[L^2(\Omega)]^3$ for  some $\vv\in[H^2(\Omega)]^3$. Then, $\overline{H}_h[\vv_h]$ converges weakly to $D^2\vv$ in $[L^2(\Omega)]^3$.
\end{lemma}
\begin{proof}
For any $\phi\in[C^{\infty}_0(\Omega)]^{3\times2\times2}$, we have 
\begin{align*}
\int_{\Omega}\overline{H}_h[\vv_h]:\phi=\sum_{T\in\Th}\int_{T}H_h[\vv_h]: \overline{\phi}=\sum_{T\in\Th}\int_{T}H_h[\vv_h]:\phi+H_h[\vv_h]:(\overline\phi-\phi),
\end{align*}
where $\overline \phi:= \frac{1}{|T|}\int_T \phi$. Lemma \ref{weak-conv} (weak convergence of $H_h[\vy_h]$) implies
\begin{equation*}
\int_{\Omega}H_h[\vv_h]:\phi\to\int_{\Omega}D^2\vv:\phi.
\end{equation*}
On the other hand, the uniform boundedness \eqref{eqn:hessian_upper} and the assumption $\| \vv_h \|_{H_h^2(\Omega)} \lesssim 1$ guarantee that
\begin{align*}
\Big|\sum_{T\in\Th}H_h[\vv_h]:(\overline\phi-\phi)\Big| \lesssim h\|H_h[\vv_h]\|_{L^2(\Omega)}\|\nabla\phi\|_{L^2(\Omega)} \lesssim h \|\nabla\phi\|_{L^2(\Omega)} \to 0
\end{align*}
as $h \to 0^+$.
Combining these two estimates yields the desired result.
\end{proof}

\end{subsection}

\subsection{Discrete admissible set}\label{subsec: D-admissible}

We introduce the discrete counterpart of the admissible set $\A$. Given a parameter $\delta>0$ to be related later to $h$, we recall the discrete isometry defect $D_h[\vy_h]$ from \eqref{isometry-defect} and define the {\it discrete admissible set} $\A_{h,\delta}$ as
\begin{equation}\label{admissible-discrete}
\A_{h,\delta}:= \big\{\vy_h\in\V_h^k(\vvarphi,\Phi):\quad{D_h[\vy_h](x_T)} \le\delta\quad\forall T\in\Th \big\},
\end{equation}
where the polynomial degree is $k\geq 2$ and $x_T$ is the barycenter of $T\in\Th$.
The Dirichlet boundary conditions are hidden within the definition \eqref{discrete-set} of $\V_h^k(\vvarphi,\Phi)$ and imposed in the weak formulation; hence they do not contribute to any essential restriction in $\A_{h,\delta}$. 
The following two lemmas are simple consequences of \eqref{admissible-discrete}.

\begin{lemma}[$\A_{h,\delta}$ is non-empty]\label{Ahdelta}
For all $h>0$ there exists $\vy_h\in\V_h^k(\vvarphi,\Phi)$ such that $D_h[\vy_h](x_T)=0$ for all $T\in\Th$.
\end{lemma}
\begin{proof}
  Let $\vy_h(x):=x$ for $x\in\Omega$. We see that $\vy_h\in[\V_h^k]^3$, and therefore $\vy_h\in\V_h^k(\vvarphi,\Phi)$ because the Dirichlet boundary conditions are not imposed \emph{essentially} in the space $\V_h^k(\vvarphi,\Phi)$ defined in \eqref{discrete-set}. Moreover, $\I[\vy_h](x_T)=I_2$, whence $D_h[\vy_h](x_T)=0$.
\end{proof}

Note that this implies $\A_{h,\delta}$ is non-empty for any $\delta>0$, because $\A_{h,0}\subset\A_{h,\delta}$. 
We postpone until Theorem \ref{Gamma:E_h} the hard question whether $\A_{h,\delta}$ is sufficiently rich to approximate $\A$: for any $\vy\in\A$ there is $\vy_h\in\A_{h,\delta}$ close to $\vy$ in a suitable sense.
The following lemma provides an estimate on the amount of local stretch and shear associated with functions in $\mathbb A_{h,\delta}$. 
\begin{lemma}[pointwise isometry constraint]\label{lem:bi_iso_control}
If $\vy_h\in\A_{h,\delta}$, then for all $T\in\Th$ { and $i=1,2$}
\begin{equation}\label{eqn:bi_iso_control}
  1-\delta\le{|\partial_i\vy_h(x_T)|^2}\le1+\delta,
  \quad
  |\partial_1{\vy_h}(x_T)\cdot\partial_2{\vy_h}(x_T)|\le\delta.
\end{equation}
\end{lemma}
\begin{proof}
From definition \eqref{admissible-discrete}, we deduce that for any $i,j=1,2$
\begin{equation*}
\Big|\partial_i{\vy_h}(x_T)\cdot\partial_j{\vy_h}(x_T)- \delta_{ij}\Big|\le\delta,
\end{equation*}
where $\delta_{ij}$ is the Kronecker delta.
The assertion thus follows.
\end{proof}

The pointwise control of isometry defect in \eqref{admissible-discrete} is inspired by the algorithms based on Kirchhoff finite elements developed in \cite{bartels2017bilayer,bartels2020stable}, where this constraint is imposed at the element vertices. Dealing with element barycenters is novel in the context of DG methods in that previous schemes impose this constraint in average over elements \cite{bonito2020discontinuous,bonito2020ldg-na}.
Having control at barycenters does not imply control of $\nabla_h\vy_h$ anywhere else, and dictates the use of mid-point quadrature for the discretization of the cubic nonlinear energy $C_h$. We discuss this next.

\subsection{Discrete energy}\label{s:discrete_min}

The LDG approximation of the energy $E[.]$ reads
\begin{equation}\label{discrete-energy}
E_h[\vy_h]:=B_h[\vy_h]-C_h[\vy_h]
\end{equation}
where $B_h[.]$ approximates the bending energy \eqref{energy-quad} and $C_h[.]$ approximates the cubic interaction energy in \eqref{energy-nonlinear}. The energy $B_h[\vy_h]$ is defined by
\begin{equation}\label{discrete-energy-0}
  B_h[\vy_h]:=\frac{1}{2}\int_{\Omega}\big|H_h[\vy_{h}] \big|^2 + S_h[\vy_h],
\end{equation}
where 
\begin{equation}\label{stabilization}
S_h[\vy_h]:= \gamma_1\|\h^{-\frac12}[\nabla_h\mathbf{y}_h]\|_{L^2(\Gh^a)}^2+\gamma_0\|\h^{-\frac32}[\mathbf{y}_h]\|_{L^2(\Gh^a)}^2
\end{equation}
is a stabilization term with parameters $\gamma_0,\gamma_1>0$, whereas $C_h[\vy_h]$ is given by \eqref{quadrature}
\begin{equation}\label{quadrature-2}
C_h[\vy_h]:=\sum_{i,j=1}^2\sum_{T\in\Th}|T|\Big( \overline{H}_h[\vy_h]_{ij}\cdot(\partial_1\vy_h\times\partial_2\vy_h)Z_{ij}\Big)(x_T).
\end{equation}
With these notations the discrete minimization problem reads
\begin{equation} \label{discrete_minimization_bilayer}
\min_{\vy_h\in\A_{h,\delta}} E_h[\vy_h].
\end{equation}

We devote the rest of this section to examine the cubic energy \eqref{quadrature-2}.
Combining Lemma~\ref{lem:bi_iso_control} (pointwise isometry constraint) with Lemma~\ref{L2bound-modified} (stability of $\overline{H}_h[\vy_h]$) yields
$$
\big| C_h[\vy_h] \big| \lesssim (1+\delta) \| \vy_h \|_{H_h^2(\Omega)} \| Z \|_{L^2(\Omega)},
$$
whence $\big| C_h[\vy_h] \big|$ is uniformly bounded whenever $\| \vy_h \|_{H_h^2(\Omega)}$ is.
Another crucial aspect of \eqref{quadrature-2} is the convergence of $C_h$ towards the continuous energy $C$ within the basic $H^2$-regularity framework. This requires dealing with the reduced Hessian $\overline{H}_h[\vy_h]$ as we show next.

\begin{lemma}[convergence of cubic energy]\label{l:conv_nl}
Let $Z$ be piecewise constant over $\Th$. Let
$\vy_h\in \A_{h,\delta}$ be a sequence of discrete deformations satisfying
\begin{equation}\label{reg-yh}
  \|\vy_h\|_{H_h^2(\Omega)} \lesssim 1
  \quad\forall \, h>0
\end{equation}
and such that $\vy_h\to\vy$ in $[L^2(\Omega)]^3$, $\nabla_h \vy_h \to \nabla \vy$ in $[L^2(\Omega)]^{3\times 2}$ for  $\vy\in [H^2(\Omega)\cap W^1_\infty(\Omega)]^3$ as $h \to 0^+$.
Then 
\begin{equation}\label{convergence-Ch}
\lim_{h\to0^+} C_h[\vy_h] = C[\vy].
\end{equation}
\end{lemma}
\begin{proof}
For any $\epsilon>0$, it suffices to show that
\begin{equation}\label{e:smooth}
\limsup_{h\to0^+}\big| C[\vy] - C_h[\vy_h] \big| \lesssim \epsilon. 
\end{equation}
We need a regularization argument to deal with the effect of quadrature.
Since $\Omega$ is Lipschitz we can regularize $\vy$, say by convolution, in such a manner that the approximate
deformation $\vy^\epsilon \in [H^3(\Omega)]^3$ satisfies
\begin{equation}\label{regularization}
\| \vy^\epsilon \|_{H^2(\Omega)} + \| \vy^\epsilon \|_{W^1_\infty(\Omega)}\lesssim 1,
\qquad 
\| \vy - \vy^\epsilon \|_{H^2(\Omega)} \lesssim \epsilon;
\end{equation}
we recall the convention that constants hidden in $\lesssim$ are independent of $h$ and $\epsilon$.
We point out that this procedure is simpler than the regularization due to Hornung \cite{hornung2008approximating} in
that $\vy^\epsilon$ need not be an isometry. We first observe that the energies  $C[\vy]$ and $C[\vy^\epsilon]$ can be made arbitrarily close because
\begin{align*}
\Big| C[\vy] &- C[\vy^\epsilon] \Big| \lesssim  \|\vy-\vy^{\epsilon}\|_{H^2(\Omega)}\|\partial_1\vy\|_{L^2(\Omega)}\|\partial_2\vy\|_{L^{\infty}(\Omega)} \| Z\|_{L^\infty(\Omega)}\\
&+\|\vy^{\epsilon}\|_{H^2(\Omega)}\|\vy-\vy^{\epsilon}\|_{H^1(\Omega)}(\|\partial_2\vy\|_{L^{\infty}(\Omega)}+\|\partial_1\vy^{\epsilon}\|_{L^{\infty}(\Omega)}) \| Z\|_{L^\infty(\Omega)} \lesssim \epsilon.
\end{align*}
We next write  $C_h[\vy_h] - C[\vy^\epsilon] = \sum_{i,j=1}^2\sum_{T\in\Th} R_1(T) + R_2(T) + R_3(T)$, where
\begin{align*}
R_1 (T)&:=\int_T(\overline{H}_h[\vy_h]_{ij}-\partial_{ij}\vy^\epsilon)\cdot(\partial_1\vy^\epsilon\times\partial_2\vy^\epsilon)Z_{ij}, \\
R_2 (T)&:=|T|\big[\overline{H}_h[\vy_h]_{ij}\cdot(\partial_1\vy_h\times\partial_2\vy_h-\partial_1\vy^\epsilon\times\partial_2\vy^\epsilon)Z_{ij}\big](x_T), \\
R_3 (T)&:=|T|\big[\overline{H}_h[\vy_h]_{ij}\cdot(\partial_1\vy^\epsilon\times\partial_2\vy^\epsilon)Z_{ij}\big](x_T)-\int_T\overline{H}_h[\vy_h]_{ij}\cdot(\partial_1\vy^\epsilon\times\partial_2\vy^\epsilon)Z_{ij},
\end{align*}
and disregard the non critical dependence on $i,j=1,2$ in the notation.
Lemma \ref{weakconv-modified} (weak convergence of $\overline{H}_h[\vy_h]$) in conjunction with \eqref{regularization} implies that
\begin{equation*} 
\limsup_{h\to0^+}\sum_{i,j=1}^2\sum_{T\in\Th} |R_1(T)|~\lesssim~\epsilon.
\end{equation*}
For $R_2$, we note that
\begin{align*}
\big|\big(\partial_1\vy_h\times\partial_2\vy_h-\partial_1\vy^\epsilon\times\partial_2\vy^\epsilon\big)(x_T)\big|
& \le \big|\nabla\big(\vy_h - \vy^\epsilon\big)(x_T)\big|
\big(|\nabla\vy_h(x_T)| + |\nabla\vy^\epsilon(x_T)| \big).
\end{align*}
By Lemma \ref{lem:bi_iso_control} (pointwise isometry constraint), the fact that $\vy_h\in\A_{h,\delta}$ and \eqref{regularization}, we have the uniform bound $|\nabla\vy_h(x_T)| + |\nabla\vy^\epsilon(x_T)|\lesssim1$ for all $x_T$. If $I_h\nabla\vy^\epsilon$ indicates the standard $\mathbb{P}_1$-Lagrange interpolant of $\nabla\vy^\epsilon$, applying approximating properties of $I_h$ together with an inverse inequality for polynomials, we conclude
\begin{align*}
\big|\nabla\big(\vy_h - \vy^\epsilon\big)(x_T)\big|&\le\big|\big(\nabla\vy_h - I_h\nabla\vy^\epsilon\big)(x_T)\big|+\big|\big(I_h\nabla\vy^\epsilon-\nabla\vy^\epsilon\big)(x_T)\big| \\
&\lesssim h_T^{-1}\|\nabla\vy_h - I_h\nabla\vy^\epsilon\|_{L^2(T)}+h_T\|D^3\vy^\epsilon\|_{L^2(T)}.
\end{align*}
We next add and subtract $\nabla\vy^\epsilon$ in the first term of the right-hand side and apply again an interpolation estimate of $I_h$ to derive
\begin{equation*}\label{eq:est-diff-yhyeps}
|T|^{1/2}\big|\nabla\big(\vy_h - \vy^\epsilon\big)(x_T)\big|\lesssim \|\nabla(\vy_h - \vy^\epsilon)\|_{L^2(T)} + h_T^2 \|D^3 \vy^\epsilon\|_{L^2(T)}.
\end{equation*}
Moreover, since $|T|^{1/2}\big|\overline{H}_h[\vy_h]_{ij}(x_T)\big|=\|\overline{H}_h[\vy_h]_{ij}\|_{L^2(T)}$ because $\overline{H}_h[\vy_h]$ is piecewise constant, we obtain
\[
 |R_2(T)| \lesssim \|\overline{H}_h[\vy_h]_{ij}\|_{L^2(T)}
\big( \|\nabla(\vy_h - \vy)\|_{L^2(T)} + \|\nabla(\vy - \vy^\epsilon)\|_{L^2(T)} + h_T^{2}\|D^3 \vy^\epsilon\|_{L^2(T)} \big),
\]
where the hidden constant is proportional to $\|Z\|_{L^\infty(\Omega)}$.  
After summing over elements, Lemma \ref{L2bound-modified} (stability of $\overline{H}_h[\vy_h]$), together with
the assumption $\nabla_h\vy_h\to\nabla\vy$ in $[L^2(\Omega)]^{3\times2}$, \eqref{reg-yh} and \eqref{regularization}, yields
\[
\limsup_{h\to0^+}\sum_{i,j=1}^2\sum_{T\in\Th} |R_2(T)|\lesssim\epsilon.
\]

It remains to deal with $R_3$ which entails the effect of quadrature. Since $Z$ and $\overline{H}_h[\vy_h]$ are constant in $T$, which is the chief reason for utilizing the reduced discrete Hessian, we can equivalently rewrite $R_3(T)$ as follows:
\[
R_3(T) = \overline{H}_h[\vy_h]_{ij} Z_{ij} \int_T (\vf(x_T) - \vf)
\]
with $\vf = \partial_1\vy^\epsilon\times\partial_2\vy^\epsilon$. The Bramble-Hilbert Lemma, in conjunction with the Sobolev embedding $W^2_1(T)\subset C(\overline{T})$ (cf. \cite[Lemma 4.3.4]{brenner2007mathematical}), implies the existence of a linear polynomial $\vp\in[\mathbb{P}_1(T)]^3$ such that $\|\vf - \vp\|_{L^{\infty}(T)} \lesssim \|D^2\vf\|_{L^1(T)}$. Since the mid-point quadrature is exact for linears, we deduce
\begin{align*}
\Big|\int_T ( \vf(x_T) - \vf)\Big|=\Big|\int_T \big\{ \big(\vf - \vp\big)(x_T) + \big(\vp - \vf\big)\big\} \Big|\le 2|T|\|\vf - \vp\|_{L^{\infty}(T)} \lesssim h_T^2\|D^2\vf\|_{L^1(T)}.
\end{align*}
Moreover, invoking \eqref{regularization},
\[
\|D^2\vf\|_{L^1(T)} \lesssim \|D^3\vy^\epsilon\|_{L^2(T)} \|\nabla\vy^\epsilon\|_{L^2(T)}
+ \|D^2\vy^\epsilon\|_{L^2(T)}^2 \lesssim \|\vy^\epsilon\|_{H^3(T)}.
\]
Inserting this back into $R_3(T)$ and adding we end up with
\[
\limsup_{h\to0^+} \sum_{i,j=1}^2\sum_{T\in\Th} |R_3(T)|\lesssim
\limsup_{h\to0^+} \Big(h \|\overline{H}_h[\vy_h]\|_{L^2(\Omega)}\Big) \|\vy^\epsilon\|_{H^3(\Omega)} = 0,
\]
because of Lemma \ref{L2bound-modified}. Altogether, we arrive at
\[
\limsup_{h\to0^+} \big|C_h[\vy_h] - C[\vy^\epsilon] \big| \lesssim \epsilon
\]
which implies the desired estimate \eqref{e:smooth}.
\end{proof}

It is worth realizing the role of the reduced discrete Hessian $\overline{H}_h[\vy_h]$
in the preceding proof, namely that it factors out
the integral defining $R_3(T)$. If we had used the discrete Hessian $H_h[\vy_h]$ instead, then there would
have been a term of the form $h_T^2 \|D^2 H_h[\vy_h]\|_{L^2(T)}$ that could only be handled via an inverse inequality
within the $H^2$-regularity setting.
This in turn would have gotten rid of the factor $h_T^2$ and the proof of \eqref{convergence-Ch} would have failed.

\section{$\Gamma$-convergence}\label{Gamma-bilayer}

The reduced energy \eqref{energy-3} consists of a bending energy $B[\vy]$ and a cubic term $C[\vy]$, and so does its discrete counterpart \eqref{discrete-energy}, namely $B_h[\vy_h]$ and $C_h[\vy_h]$. Compactness and
$\Gamma$-convergence of the bending energy part, being similar to the single layer model, could be deduced from the results in \cite{bonito2020ldg-na}. For instance, we have that for any $\gamma_0,\gamma_1>0$, there exists a constant $c_{coer}$ such that \cite[(37) and (38)]{bonito2020ldg-na}
\begin{equation} \label{eqn:coercivity_bi}
  c_{coer}^{-1} \|\vy_h\|_{H_h^2(\Omega)}^2 \le B_h[\vy_h] \le c_{cont} \|\vy_h\|_{H_h^2(\Omega)}^2
  \qquad  \forall \vy_h \in \V^k_h(\vvarphi,\Phi),
\end{equation}
and the constant $c_{coer} \to \infty$ if either $\gamma_0$ or $\gamma_1\to0^+$.
In spite of that, \cite{bonito2020ldg-na} enforces the isometry constraint in average and constructs the recovery sequence needed for $\Gamma$-convergence via standard nodal interpolation. 
Therefore, the analysis below incorporates new ideas which do not follow from \cite{bonito2020ldg-na}.

We start with the equicoercivity of energy $E_h$. The difficulty is dealing with $C_h$.
\begin{lemma}[coercivity of total energy]\label{coercivity:E_h_bi}
Let $\delta>0$ and $\vy_h \in \A_{h,\delta}$. 
There exists a constant $\tilde c_{coer}>0$ independent of $\delta$, but depending on the given data $Z$ and $\Th$ only through its shape regularity constant, such that
\begin{equation}\label{e:coercfull}
 (2c_{coer})^{-1}\|\vy_h\|^2_{H_h^2(\Omega)} \leq   E_h[\vy_h] + \tilde c_{coer}(1+\delta)^2.
\end{equation}
\end{lemma}
\begin{proof}
We write $B_h = E_h + C_h$ and employ \eqref{eqn:coercivity_bi} for $B_h$ to obtain
$$
c_{coer}^{-1} \| \vy_h \|_{H^2_h(\Omega)}^2 \leq E_h[\vy_h] + C_h[\vy_h].
$$
It remains to estimate the cubic term $C_h[\vy_h]$. 
Combining Lemma \ref{lem:bi_iso_control} (pointwise isometry constraint) with the Cauchy-Schwarz inequality yields
\begin{align*}
C_h[\vy_h] &\leq \sum_{i,j=1}^2\sum_{T\in\Th}|T|\big|\overline{H}_h[\vy_h]_{ij}\cdot(\partial_1\vy_h\times\partial_2\vy_h)Z_{ij}\big|(x_T) \\
& \le \sum_{i,j=1}^2 \sum_{T\in\Th} |T|^{\frac12} \|\overline{H}_h[\vy_h]_{ij}\|_{L^2(T)} |\partial_1\vy_h(x_T)| \, |\partial_2\vy_h(x_T)|\|Z\|_{L^{\infty}(T)}\\
&\le 2(1+\delta)\|Z\|_{L^{\infty}(\Omega)}|\Omega|^{\frac{1}{2}}\|\overline{H}_h[\vy_h]\|_{L^2(\Omega)}.
\end{align*}
Invoking Lemma \ref{L2bound-modified} (stability of $\overline{H}_h[\vy_h]$) and Young's inequality yields
\begin{equation}\label{eq:coer-eh-ehtilde}
  \frac{1}{2c_{coer}} \| \vy_h \|_{H^2_h(\Omega)}^2 \leq E_h[\vy_h] + 2 c_{coer}c_{stab}^2|\Omega|\|Z\|_{L^{\infty}(\Omega)}^2
  (1+\delta)^2,
\end{equation}
which is the desired estimate \eqref{e:coercfull} with $\tilde c_{coer} = 2 c_{coer}c_{stab}^2|\Omega|\|Z\|_{L^{\infty}(\Omega)}^2$.
\end{proof}

We now prove {a compactness result and} $\Gamma$-convergence of $E_h$ towards $E$, which consists of a $\liminf$ and a $\limsup$ property.

\begin{thm}[{compactness and $\Gamma$-convergence}]\label{Gamma:E_h}
Let $\delta=\delta(h)\rightarrow 0$ as $h\rightarrow 0^+$. Then
\begin{enumerate}[label=(\roman*),itemindent=-0ex]
{\item Compactness: Let $\vy_h\in\A_{h,\delta}$ be a sequence such that $E_h[\vy_h]$ is uniformly bounded in $h$. Then there exists $\vy\in\A$ such that $\vy_h\to\vy$ in $[L^2(\Omega)]^3$ and $\nabla_h\vy_h \to  \nabla\vy$ in $[L^2(\Omega)]^{3\times2}$ for a subsequence (not relabeled).}
\item Lim-inf property: 
  {If $\vy_h\in\A_{h,\delta}$ satisfies $\vy_h\to\vy$ in $[L^2(\Omega)]^3$ and $\nabla_h\vy_h \to  \nabla\vy$ in $[L^2(\Omega)]^{3\times2}$ where $\vy\in[H^2(\Omega)\cap W^1_\infty(\Omega)]^3$, then
$E[\vy]\le\liminf\limits_{h\to0^+}E_h[\vy_h]$.}
%
\item Lim-sup property: For any $\vy\in\A$ there exists $\vy_h\in\A_{h,\delta}$ such that $\vy_h\to\vy$ in $[L^2(\Omega)]^3$ and $E[\vy]\ge\limsup\limits_{h\to0^+}E_h[\vy_h]$.
\end{enumerate}
\end{thm}

\begin{proof}
We prove properties {(i),(ii) and (iii)} separately.

\smallskip\noindent
{
\textbf{(i) Compactness.}
Lemma~\ref{coercivity:E_h_bi} (coercivity of total energy) and \eqref{friedrichs} imply
\[  
\|\vy_h\|_{L^2(\Omega)} + \|\nabla_h\vy_h\|_{L^2(\Omega)} + \|\vy_h\|_{H_h^2(\Omega)} \lesssim 1.
\]
Proceeding as in \cite[Proposition~5.2]{bonito2019dg}, there exists $\vy\in[H^2(\Omega)]^3$ satisfying the Dirichlet boundary conditions in \eqref{admissible} and $\vy_h\to\vy$ in $[L^2(\Omega)]^3$, $\nabla_h\vy_h \to  \nabla\vy$ in $[L^2(\Omega)]^{3\times2}$.

It just remains to prove the isometry constraint $\I[\vy] = I_2$ a.e. in $\Omega$.
To this end, recall that $\I[\vy_h]=\nabla_h\vy_h^T\nabla_h\vy_h$, let $T\in\Th$ and note that
\begin{align*}
\|\I[\vy_h] - I_2 \|_{L^1(T)} & \le \|\I[\vy_h] - \I[\vy_h](x_T) \|_{L^1(T)} + |T| \, \big| \I[\vy_h](x_T) - I_2 \big|
\\
& \lesssim h_T \|D^2_h\vy_h\|_{L^2(T)}\|\nabla_h\vy_h\|_{L^2(T)} + \delta |T|,
\end{align*}
because $\vy_h\in\A_{h,\delta}$ whence $D_h[\vy_h](x_T)=\big| \I[\vy_h](x_T) - I_2 \big|\le\delta$. Adding over
$T$ and employing the uniform boundedness of $\| D_h^2\vy_h \|_{L^2(\Omega)}$ and $\|\nabla_h\vy_h\|_{L^2(\Omega)}$
results in
\[
\|\I[\vy_h] - I_2 \|_{L^1(\Omega)} \lesssim h + \delta \to 0
\quad\textrm{as } h\to0^+.
\]
On the other hand, we see that
\[
\I[\vy_h] - \I[\vy] = \nabla_{h}(\vy_h-\vy)^T\nabla_{h}\vy_h+\nabla\vy^T\nabla_h(\vy_h-\vy)
\]
implies
\begin{align*}
\|\I[\vy_h] - \I[\vy]\|_{L^1(\Omega)} \le\left(\|\nabla\vy\|_{L^2(\Omega)}+\|\nabla_h\vy_h\|_{L^2(\Omega)}\right)\|\nabla_h\vy_h-\nabla\vy\|_{L^2(\Omega)} \to0,
\end{align*}
as $h\to0^+$ because $\|\nabla_h\vy_h\|_{L^2(\Omega)}\lesssim 1$.
This and the triangle inequality lead to $\|\I[\vy] - I_2\|_{L^1(\Omega)}=0$ and consequently
$\I[\vy] = I_2$ a.e. in $\Omega$, as desired.
}
  
\medskip\noindent
\textbf{(ii) lim-inf property.}
In view of Lemma~\ref{weak-conv} (weak convergence of $H_h$), we deduce $H_h[\vy_h]\rightharpoonup D^2 \vy$ in
$\big[L^2(\Omega)\big]^{3\times 2\times 2}$. The lower-semicontinuity of the $L^2$-norm under weak-limits
together {with} the fact that the
stabilization terms in $B_h[\vy_h]$ are positive guarantee that
$$
B[\vy] = \frac12 \int_\Omega |D^2\vy|^2 \leq \liminf \limits_{h\to0^+} B_h[\vy_h].
$$
In addition, Lemma~\ref{l:conv_nl} (convergence of cubic energy) yields $\lim_{h\to0^+} C_h[\vy_h]=C[{\vy}]$,
and altogether gives $E[\vy]\ \le \liminf_{h\to0^+} E_h[\vy_h]$ as asserted.

\medskip\noindent
\textbf{(iii) lim-sup property.}
The difficulty to construct a recovery sequence $\vy_h\in\A_{h,\delta}$ is that the regularity
$\vy\in [H^2(\Omega) \cap W^1_\infty(\Omega)]^3$ is borderline to define pointwise values of $\nabla\vy$
and thus enforce the isometry defect $D_h[\vy_h](x_T)$ at every element barycenter $x_T$. Hence, 
we invoke the regularization procedure of P. Hornung \cite{hornung2008approximating}: given an isometry $\vy \in [H^2(\Omega)]^3$ and $\epsilon>0$, there exists an {\it isometry} $\vy^\epsilon \in [H^3(\Omega)]^3$ such that
\begin{equation}\label{yeps-approx}
\| \vy - \vy^\epsilon \|_{H^2(\Omega)} \lesssim \epsilon, \qquad \| D^2 \vy^\epsilon \|_{L^2(\Omega)} \lesssim
\| D^2 \vy \|_{L^2(\Omega)}.
\end{equation}
As usual, the constants hidden in the symbol $\lesssim$ are independent of $h$ and $\epsilon$.
We now set $\vy_h:=R_h[\vy^\epsilon]$, where the recovery operator $R_h: [H^3(\Omega)]^3 \to [\V^k_h]^3$ is
the following quadratic Taylor expansion about $x_T$ for every $T\in\Th$
\begin{equation}\label{recovery}
  R_h[\vw] (x):=\vw(x_T)+\nabla\vw(x_T)(x-x_T)
  +\frac{1}{2}(x-x_T)^T Q_T[\vw](x-x_T) \quad\forall \, x\in T,
\end{equation}
where $Q_T[\vw]:=\frac{1}{|T|}\int_{T}D^2\vw$. Note that $\nabla \vy_h(x_T)
= \nabla\vy^\epsilon(x_T)$ and $D_h[\vy_h](x_T) = 0$, whence $\vy_h\in\A_{h,0}\subset\A_{h,\delta}$. We next show the two
convergence properties of $\vy_h$ in (ii).

Since $R_h \big|_T$ is invariant over the space $\mathbb{P}_1$ of polynomials of degree $\le1$, we have
\[ 
\vw - R_h[\vw] = (\vw - \vp) - R_h[\vw-\vp]
\quad\forall \, \vp\in[\mathbb{P}_1]^3.
\]
Therefore, combining the stability in $W^1_\infty(T)$ of the linear part of $R_h$ with the
Bramble-Hilbert lemma and the property $\|Q_T[\vw]\|_{L^2(T)} \le |\vw|_{H^2(T)}$, we deduce
\begin{align*}
\|\vw-R_h[\vw]\|_{H^1(T)} &\lesssim h_T \|\nabla(\vw-\vp)\|_{W^1_\infty(T)} + h_T \|Q_T[\vw]\|_{L^2(T)}
\\
& \lesssim h_T^2 \|\vw\|_{H^3(T)} + h_T |\vw|_{H^2(T)} \lesssim  h_T \|\vw\|_{H^3(T)}.
\end{align*}
Notice the presence of the full $H^3$-norm on the right-hand side of the above estimate, which
accounts for possible subdivisions made of quadrilaterals \cite{ciarlet1972interpolation,ern2021finite,bonito2019dg}.
We next square and add over $T\in\Th$ to obtain
\[
\|\vw-R_h[\vw]\|_{L^2(\Omega)} + \|\nabla\vw-\nabla_h R_h[\vw]\|_{L^2(\Omega)} \lesssim h \|\vw\|_{H^3(\Omega)}.
\]
This estimate for $\vw=\vy^\epsilon$, in conjunction with \eqref{yeps-approx}, yields
\[
\|\vy-\vy_h\|_{L^2(\Omega)} + \|\nabla\vy-\nabla_h \vy_h\|_{L^2(\Omega)} \lesssim \epsilon + h \|\vy^\epsilon\|_{H^3(\Omega)},
\]
whence $\|\vy-\vy_h\|_{L^2(\Omega)}\lesssim \epsilon$ provided $h$ is sufficiently small so that
$h \|\vy^\epsilon\|_{H^3(\Omega)} \le \epsilon$. This shows the asserted convergence $\vy_h\to\vy$ in $[L^2(\Omega)]^2$
because $\epsilon$ is arbitrary.

It remains to show the convergence $E_h[\vy_h]\to E[\vy]$ as $h\to0^+$, which in turn implies the
desired lim-sup property. Since $D^2\vy_h=Q_T[\vy^\epsilon]$, we infer that
\[
\|D_h^2\vy_h\|_{L^2(\Omega)}^2 = \sum_{T\in\Th}\|Q_T[\vy^\epsilon]\|_{L^2(T)}^2 \le
\sum_{T\in\Th}\|D^2 \vy^\epsilon\|_{L^2(T)}^2 \lesssim \|D^2 \vy\|_{L^2(\Omega)}^2,
\]
according to \eqref{yeps-approx}. Moreover,
\begin{align*}
\|D_h^2 \vy_h &- D^2\vy\|_{L^2(\Omega)}^2 = \sum_{T\in\Th} \|Q_T[\vy^\epsilon]-D^2\vy\|_{L^2(T)}^2
\\
& \le \sum_{T\in\Th} \|Q_T[\vy^\epsilon]-D^2\vy^\epsilon\|_{L^2(T)}^2 + \|D^2\vy^\epsilon-D^2\vy\|_{L^2(T)}^2
\lesssim h^2 \|\vy^\epsilon\|_{H^3(\Omega)}^2 + \epsilon^2
\end{align*}
shows that $D_h^2\vy_h \to D^2\vy$ and Lemma~\ref{strong-conv} (strong convergence of $H_h$)
gives
\begin{equation*} 
H_h[\vy_h]\to D^2\vy \quad \mbox{strongly in } \,\, [L^2(\Omega)]^{3\times 2\times 2}.
\end{equation*}
An argument similar to \cite[Appendix B and C]{bonito2020ldg-na}, invoking the trace inequality,
yields
\[
S_h[\vy_h]  \lesssim \sum_{T\in\Th}  \|\vy - \vy_h\|_{H^2(T)}^2 \rightarrow 0, \quad \textrm{as} \ h\to0^+
\]
for the stabilization energy $S_h[\vy_h]$ in \eqref{stabilization} and implies convergence
of the bending energy $B_h$ in \eqref{discrete-energy-0}, namely
$
\lim\limits_{h\to0^+} B_h[\vy_h]=B[\vy].
$
Finally, in view of the preceding discussion, we see that the assumptions of
Lemma \ref{l:conv_nl} (convergence of the cubic energy) are valid, whence Lemma \ref{l:conv_nl} 
implies $C_h[\vy_h]\to C[\vy]$ and completes the proof.
\end{proof}

The construction of the recovery sequence in Theorem \ref{Gamma:E_h} {(compactness and $\Gamma$-convergence)}
is closely related to Lemma \ref{l:conv_nl} (convergence
of the cubic energy) and illustrates the crucial interplay between enforcing the isometry defect $D_h[\vy_h]$
at barycenters and the mid-point quadrature rule in the cubic energy $C_h[\vy_h]$.
This, however, limits the accuracy of LDG to that of lowest polynomial degree $k=2$. We leave the design of
an LDG method with formal higher accuracy $k>2$ open. 

\begin{cor}[convergence of global minimizers]\label{C:convergence}
  If $\vy_h{\in\A_{h,\delta}}$ is an almost global minimizer of $E_h$ in the sense that
  \[
   E_h[\vy_h] \le \inf_{\vw_h\in\A_{h,\delta}} E_h[\vw_h] + \sigma
  \]
  where $\sigma, \delta \to 0$ as $h\to0^+$, then $\{\vy_h\}_{h>0}$ is precompact in $[L^2(\Omega)]^3$
  and every cluster point $\vy$ belongs to $\A$ and is a global minimizer of $E$,
  namely $E[\vy]=\inf_{\vw\in\A} E[\vw]$. Moreover, up to a subsequence (not relabeled) the energies
  converge
  \[
   E[\vy] = \lim_{h\to0^+} E_h[\vy_h].
  \]
\end{cor}

We omit the proof of Corollary \ref{C:convergence}, which readily follows from Theorem \ref{Gamma:E_h}
{(compactness and $\Gamma$-convergence),}
and refer instead to \cite{bartels2013approximation,bartels2015numerical,bartels2017bilayer,bonito2019dg}
for details.

\section{Bilayer model with creases}\label{S:creases}

Bartels, Bonito and Hornung have recently developed a reduced single layer model that allows for folding across creases  \cite{BBH2021Folding}. The resulting two-dimensional model hinges on a general hyperelastic material description with appropriate scaling conditions on the energy, and consists of a piecewise nonlinear Kirchhoff plate bending model with a continuity condition at the creases. For a prescribed Lipschitz curve  $\mathcal C$ intersecting the boundary of $\Omega$ transversally, the modified bending energy of \cite{BBH2021Folding} reads
$$
\widetilde B[\vy]:= \frac 1 2 \int_{\Omega \setminus \mathcal C} \big| \II[\vy] \big|^2
= \frac12 \int_{\Omega \setminus \mathcal C} | D^2\vy |^2
$$
for deformations $\vy \in [H^2(\Omega \setminus \mathcal C) \cap W^{1,\infty}(\Omega)]^3$ satisfying the isometry constraint $\I[\vy] = I_2$ along with possible  boundary conditions. Properly designed creases allow for flapping mechanisms
upon actuation at the boundary which are of interest in engineering and medicine \cite{BBH2021Folding}.

In this section we explore a similar modification of the elastic energy \eqref{minimization}
\begin{equation}\label{eq:modified-energy}
\widetilde{E}[\vy]:= \frac 1 2 \int_{\Omega \setminus \mathcal C} \big| \II[\vy] - Z \big|^2,
\end{equation}
but without justification from 3d hyperelasticity. Therefore, we leave open the question
whether this energy is the appropriate $\Gamma$-limit for bilayer materials. 
We also modify the admissible set to be
\[
\widetilde{\A}:=\big\{\mathbf{y}\in[H^2(\Omega\setminus\Ccal)\cap W^1_\infty(\Omega)]^3:\quad\I[\vy]=I_2\ \text{ in }\Omega,\quad \mathbf{y}=\vvarphi, \ \nabla\mathbf{y}=\Phi\text{ on }\Gamma^D\big\}.
\]
Our goal is, instead,
to investigate the relation between \eqref{eq:modified-energy} and its fully discrete counterpart,
and demonstrate computationally the crucial role of spontaneous curvature $Z$ to produce plate folding
without actuation via boundary conditions.

We extend our LDG method to account for creases as in \cite{BBH2021Folding}. We consider iso-parametric
partitions $\mathcal T_h$ made of possibly curved elements, i.e. the mapping $F_T$ used to define the finite element space $\V_h^k$ locally is $[\mathbb Q_2]^2$ instead of $[\mathbb Q_1]^2$ (or $[\mathbb P_2]^2$ instead of $[\mathbb P_1]^2$). We further assume that the crease $\mathcal C$ is exactly matched by $\Th$:
\begin{equation}\label{eq:crease}
  \textit{$\mathcal C$ is made of piecewise quadratic edges $e_1,...,e_J \in \mathcal E_h$.}
\end{equation}
This geometric
assumption is restrictive but instrumental for the theory below. Dealing with more general creases $\mathcal C$, just
interpolated by $\mathcal E_h$, is important and the subject of current research; we refer to
\cite[Section 4.4]{BBH2021Folding} for some discussion.

The distributional derivative of $\vy\in[H^2(\Omega\setminus\Ccal) \cap W^1_\infty(\Omega)]^3$ reads
\[
D^2 \vy = \widetilde{D}^2 \vy + [\nabla\vy]\otimes\vn \, \delta_\Ccal,
\]
where $\widetilde{D}^2 \vy$ stands for the absolutely continuous part of $D^2 \vy$,
or restriction of $D^2 \vy$ to $\Omega\setminus\Ccal$ that happens to be $L^2$,
while $[\nabla\vy]\otimes\vn \, \delta_\Ccal$ is the singular part supported on $\Ccal$ and $\vn$ is a unit normal vector to $\Ccal$.
The first issue to tackle is the construction of a discrete Hessian $\widetilde{H}_h[\vy_h]$ that allows
for folding across $\Ccal$ and mimics $\widetilde{D}^2 \vy$.
As in \cite{BBH2021Folding}, we replace the global lift $\mathcal{R}_h$ in \eqref{E:global-lifting} by
\[
\widetilde{\mathcal{R}}_h := \sum_{e \in \mathcal E_h^a \setminus \{ e_1,...,e_J \}} r_e,
\]
where $\{e_j\}_{j=1}^J$ are defined in \eqref{eq:crease}, and let the {\it modified discrete Hessian} be
\[
\widetilde{H}_h[\vy_h] := D_h^2 \vy_h - \widetilde{\mathcal{R}}_h([\nabla_h\vy_h]) + \mathcal{B}_h([\vy_h]).
\]
We likewise replace \eqref{def:H2norm} by the modified mesh-dependent form
$\langle \cdot,\cdot \rangle_{\widetilde{H}_h^2}$
\begin{align*}
\langle \vv_h,\vw_h \rangle_{\widetilde{H}_h^2} &:=  (D^2_h\vv_h,D^2_h\vw_h)_{L^2(\Omega)}
  \\& +(\h^{-1}\jump{\nabla_h\vv_h},\jump{\nabla_h\vw_h})_{L^2(\Gh^a\setminus\Ccal)}
  +(\h^{-3}\jump{\vv_h},\jump{\vw_h})_{L^2(\Gh^a)}.
\end{align*}
In essence, the ability for the plates to fold freely across $\mathcal C$ is reflected in
the absence of all the contributions related to $[\nabla \vy_h]$ across $\mathcal C$. This is
the key to the following lemma whose proof follows along the lines of \cite[Appendix B]{bonito2020ldg-na}
and is thus omitted.

\begin{lemma}[convergence of $\widetilde{H}_h$]\label{L:convergence-modifiedH}
  Let the crease $\Ccal$ satisfy \eqref{eq:crease}. Then there holds
  \begin{enumerate}[label=(\roman*),itemindent=0ex]
  \item Weak convergence:
    If $k \geq 2$ and $\vv_h\in\V_h^k(\vvarphi,\Phi)$ satisfies $||\vv_h||_{H_h^2(\Omega)}\lesssim 1$ and $\vv_h\to\vv\in[H^2(\Omega\setminus\Ccal)\cap H^1(\Omega)]^3$ in $[L^2(\Omega)]^3$ as $h\rightarrow 0^+$, then we have
    \begin{equation*}
    \widetilde{H}_h[\vv_h]\rightharpoonup \widetilde{D}^2 \vv \quad \mbox{in } \left[L^2(\Omega)\right]^{3\times 2\times 2} \quad \mbox{as } h\rightarrow 0^+.
   \end{equation*}

  \item Strong convergence:
  Let $\vv\in[H^2(\Omega\setminus\Ccal)]^3$ be any function such that $\vv=\vvarphi$ and $\nabla\vv=\Phi$ on $\Gamma^D$. Moreover,  let $\vv_h\in\V^k_h(\vvarphi,\Phi)$ satisfy
\begin{equation*}
  \|D^2\vv_h\|_{L^2(T)}\lesssim\|\vv\|_{H^2(T)}\,\,\forall T\in\Th,
  \quad
  \sum_{T\in\Th}\|\vv_h-\vv\|_{H^2(T)}^2\to0\text{ as }h\to0^+.
\end{equation*}
Then we have as $h\rightarrow 0^+$
\begin{equation*}
\widetilde{H}_h[\vv_h]\to \widetilde{D}^2\vv \quad \mbox{strongly in } \,\, [L^2(\Omega)]^{3\times 2\times 2}.
\end{equation*} 	

  \end{enumerate}
\end{lemma}

We are now ready to introduce the LDG approximation of $\widetilde{E}[\vy]$ in \eqref{eq:modified-energy}, namely
$$
\widetilde{E}_h[\vy_h] := \widetilde{B}_h[\vy_h] + \widetilde{C}_h[\vy_h],
$$
where 
$$
\widetilde{B}_h[\vy_h]:= \frac 1 2 \int_\Omega | \widetilde{H}_h[\vy_h]|^2 +\gamma_1\|\h^{-\frac12}[\nabla_h\mathbf{y}_h]\|_{L^2(\Gh^a \setminus \mathcal C)}^2+\gamma_0\|\h^{-\frac32}[\mathbf{y}_h]\|_{L^2(\Gh^a)}^2,
$$
and
$$
\widetilde{C}_h[\vy_h]:=\sum_{i,j=1}^2\sum_{T\in\Th}|T| \, \overline{H}_h[\vy_h]_{ij}\cdot\big(\partial_1\vy_h\times\partial_2\vy_h\big)(x_T) \, Z_{ij},
$$
with $\overline{H}_h[\vy_h]|_T:= \frac{1}{|T|}\int_T \widetilde{H}_h[\vy_h]$. Lemmas \ref{L2bound-modified} and
\ref{weakconv-modified} are valid for $\overline{H}_h[\vy_h]$, as well as Lemma \ref{l:conv_nl} (convergence
of cubic energy) and Lemma \ref{coercivity:E_h_bi} (coercivity of total energy).

It remains to examine the convergence of the discrete global minimizers towards the continuous global minimizers.
Assume that the crease $\Ccal$ splits $\Omega$ into two disjoint sets $\Omega_1$ and $\Omega_2$.
Since Hornung's regularization procedure \cite{hornung2008approximating} cannot guarantee general Dirichlet
boundary conditions, it is not clear how to regularize in $\Omega_1$ and $\Omega_2$ functions that belong to
$[H^2(\Omega\setminus\Ccal)\cap W^1_\infty(\Omega)]^3$ and yet maintain the location of the crease $\Ccal$,
namely obtain an isometry in $[H^3(\Omega\setminus\Ccal)\cap W^1_\infty(\Omega)]^3$. 
Another obstruction stems from the use of curved elements necessarily for the subdivisions to match the crease. 
When using polynomial mappings from the reference to the physical elements, the resulting finite element functions are not necessarily polynomial in the physical element, thereby ruling out the construction of the recovery sequence proposed to guarantee the $\limsup$ property; see Theorem~\ref{Gamma:E_h}(ii). 

We circumvent these issues by requiring slightly more smoothness on \emph{one} of the global minimizers $\vy$, which in turn allows for a different, more generic, construction of its recovery sequence. 
Because the additional regularity cannot be derived from our $\Gamma$-convergence theory, we assume the existence of a global minimizer $\vy^*\in\widetilde{\A}$ of $\widetilde{E}$ with the
following property
\begin{equation}\label{eq:extra-regularity}
\vy^*|_{\Omega_i}  \in C^1(\overline{\Omega}_i), \qquad i=1,2.
\end{equation}
Note that the above assumption is consistent with practical configurations. 
We also point out that this regularity assumption and the fact that the subdivision matches the crease entail the existence of a modulus of smoothness $\omega$ so that
\begin{equation}\label{e:modulus_smooth}
  | \nabla \vy^*(x) - \nabla \vy^*(z) | \leq \omega( h_T)
  \quad\forall\, x,z\in T, \,\, \forall\, T \in  \mathcal T_h,
\end{equation}
with $\omega(s) \to 0$ as $s\to 0^+$. 

The construction of the recovery sequence for deformations satisfying the additional regularity \eqref{eq:extra-regularity} is then based on a piecewise averaged Taylor polynomial. The latter does not preserve the isometry constraint pointwise but \eqref{eq:extra-regularity} allows for control of the isometry defect.

Before embarking on the proofs, we recall a useful result on the averaged Taylor polynomial \cite{brenner2007mathematical} defined on the reference element $\widehat T$ (see Section~\ref{subsec:broken}). Until the end of this section, we consider the case when the reference element is a square and $\mathbb Q_k$ finite element functions are used.
The case where $\widehat T$ is the unit simplex is somewhat simpler and can be dealt with similarly. Let $\widehat B$ be a ball centered at the barycenter of $\widehat T$ such that its closure is contained in $\widehat T$ and $\widehat \zeta$ be a cut-off function with unit mass supported on the closure of $\widehat B$. For $\widehat w \in L^1(\widehat T)$ let
\begin{equation}\label{averaged-taylor}
Q[\widehat w](\hat x) := \sum_{|\alpha|_{\infty} \leq 2} \int_{\widehat B}\frac{1}{\alpha!} \widehat D^\alpha  \widehat w(\hat z) (\hat x - \hat z)^\alpha \widehat \zeta(\hat z) d \hat z \in \mathbb Q_2,
\end{equation}
be the averaged Taylor polynomial
where $\alpha := (\alpha_1,\alpha_2)$ is a multi-index with non-negative integers $\alpha_1, \alpha_2$ and $|\alpha|_\infty := \max\{\alpha_1,\alpha_2\}$. We recall the following useful properties of $Q$ and refer to \cite{brenner2007mathematical} for additional details: $Q$  preserves $\mathbb Q_2$ on $\widehat{T}$
\begin{equation}\label{e:Q_preserve_q2}
Q[\widehat p] = \widehat p, \qquad \widehat p \in \mathbb Q_2,
\end{equation}
is stable 
\begin{equation}\label{e:Q_stability}
\| Q[\widehat w] \|_{W^k_\infty(\widehat T)} \lesssim \| \widehat w \|_{L^1(\widehat B)}, \qquad \forall k\in \mathbb{N}
\end{equation}
and convergent
\begin{equation}\label{e:Q_approx}
| \widehat w - Q[\widehat w] |_{H^k(\widehat T)} \lesssim \left(\sum_{i=1}^2 \| \partial^3_{\widehat x_i} \widehat w \|_{L^2(\widehat T)}^2\right)^{1/2}, \qquad 0\le k<3.
\end{equation}

We next discuss estimates for isoparametric mappings $F_T:\widehat{T}\to T$ between the reference element $\widehat{T}$ and $T\in\Th$ so that $F_T\in[\mathbb{Q}_2]^2$. They establish relationship between norms on $\widehat T$ and $T$, as well as provide an interpolation estimate in $[\V_h^2]^3$. In fact, for $v \in H^2(T)$ and $\widehat{v}=v \circ F_T \in H^2(\widehat{T})$, there holds
\begin{equation}\label{e:L2_map_back}
\|\widehat{v}\|_{L^2(\widehat T)} \approx  h_T^{-1}\| v \|_{L^2(T)}, \quad \|\widehat{v}\|_{L^{\infty}(\widehat T)} \approx  \| v \|_{L^{\infty}(T)}
\end{equation}
\begin{equation}\label{e:nabla_map_back}
\| \widehat \nabla \widehat{v} \|_{L^2(\widehat T)} \approx  \|\nabla v \|_{L^2(T)},\quad \| \widehat \nabla \widehat{v} \|_{L^{\infty}(\widehat T)} \approx h_T \|\nabla v \|_{L^{\infty}(T)},
\end{equation}
\begin{equation}\label{e:D2_map_back}
\| \widehat D^2 \widehat{v} \|_{L^2(\widehat T)} \lesssim h_T \|D^2 v \|_{L^2(T)} + \| \nabla v \|_{L^2(T)},
\end{equation}
\begin{equation}\label{e:D2_map_back_2}
\| D^2 v \|_{L^2(T)} \lesssim h_T^{-1} \|\widehat{v}\|_{H^2(\widehat T)},
\end{equation}
whence for $m=0,1,2$
\begin{equation}\label{Hm-est}
|v|_{H^m(T)} \lesssim h_T^{1-m} \| \widehat{v} \|_{H^2(\widehat{T})}.
\end{equation}
Moreover if $v \in H^3(T)$, we further obtain
\begin{equation}\label{e:D3_map_back}
\| \widehat D^3 \widehat{v}\|_{L^2(\widehat T)} \lesssim h_T^2 \| D^3 v \|_{L^2(T)} + h_T \|D^2 v \|_{L^2(T)} + \| \nabla v \|_{L^2(T)}.
\end{equation}
Note that the first four estimates are discussed and proved in \cite[Appendix]{bonito2019dg}, while one can extend the proof of \eqref{e:D2_map_back} to show \eqref{e:D3_map_back}. Additionally, as in \cite[Lemma A.4]{bonito2019dg}, the local Lagrange interpolant $I_h w \in [\mathbb V_h^2]^3$ for any $w \in H^3(T)$ satisfies the estimate
\begin{equation}\label{e:approx_Ih}
| w - I_h w |_{H^m(T)} \lesssim h^{3-m} \| v \|_{H^3(T)},
\end{equation}
for $0\le m\le 3$.

The next lemma describes the modified $\limsup$ property which hinges on the averaged Taylor polynomial \eqref{averaged-taylor}; compare with {Theorem~\ref{Gamma:E_h}(iii).}

\begin{lemma}[$\limsup$ property with creases]\label{l:limsup_crease}
Let $\vy^* \in \widetilde\A$ satisfy the regularity assumption \eqref{eq:extra-regularity} and let $\omega$ be the modulus of smoothness in \eqref{e:modulus_smooth}. There is a constant $c$ and $\vy_h^* \in \mathbb A_{h,c\omega(h)}$ such that $\vy_h^*\to \vy^*$ in $[L^2(\Omega)]^3$ and $\lim_{h\to 0^+} \widetilde E_h[\vy_h] = \widetilde E[\vy^*]$.
\end{lemma}
\begin{proof}
As usual, the hat symbol denotes quantities defined on the reference element $\widehat{T}$. 
Let  $\vy_h^*  \in [\mathbb V_h^2]^3$ be defined locally by
$$
\vy_h^*|_T := \widehat \vy_h^*|_T \circ F_T^{-1}, \quad \widehat \vy_h^*|_T := Q[\vy^* \! \circ F_T]
\quad\forall \, T\in \mathcal T_h,
$$
where $Q$ is given in \eqref{averaged-taylor} and is applied component-wise.
Note that by construction we indeed have $\vy_h^*\in [\mathbb V_h^2]^3$.
The rest of the proof consists of 4 steps.

\medskip
\textit{Step (i): isometry defect}.
We claim the intermediate estimates
\begin{equation}\label{rec:intermediate_estim}
\| \nabla \vy_h^*\|_{L^\infty(T)} \lesssim \| \nabla \vy^* \|_{L^\infty(T)}, \quad \| \nabla (\vy^* - \vy_h^*) \|_{L^\infty(T)} \lesssim \omega(h_T) \quad\forall \, T\in \mathcal T_h.
\end{equation}
To show the first estimate, we use \eqref{e:nabla_map_back} and \eqref{e:Q_preserve_q2} to write
$$
\| \nabla \vy_h^*\|_{L^\infty(T)} \lesssim h_T^{-1} \| \widehat \nabla \widehat \vy_h^*\|_{L^\infty(\widehat T)} 
 \lesssim h_T^{-1} \| \widehat \nabla (\widehat \vy_h^*-\vc)\|_{L^\infty(\widehat T)} = h_T^{-1} \| \widehat \nabla Q[\widehat \vy^* \! -\vc]\|_{L^\infty(\widehat T)},
 $$
where $\vc := |T|^{-1} \int_T \vy^* \in \mathbb R^3$ is the average of $\vy^*$.
We then employ the stability \eqref{e:Q_stability} of $Q$, together with \eqref{e:L2_map_back} and Poinc\'are inequality, to deduce
$$
\| \nabla \vy_h^*\|_{L^\infty(T)}  \lesssim h_T^{-1} \| \widehat \vy^*-\vc\|_{L^\infty(\widehat T)} \lesssim  h_T^{-1} \|  \vy^*-\vc \|_{L^\infty(T)} \lesssim  \| \nabla \vy^* \|_{L^\infty(T)},
$$
which is the first estimate in \eqref{rec:intermediate_estim}.

To prove the second estimate in \eqref{rec:intermediate_estim}, we first notice that for any $\vp \in [\mathbb P_1]^3$ we have $\widehat \vp:= \vp \circ F_T \in [\mathbb Q_2]^3$. Since $\widehat{\vp}=Q[\widehat{\vp}]$, according to \eqref{e:Q_preserve_q2}, we proceed as before, but now using $\widehat \vp \in \mathbb Q_2$ instead of the constant $\vc$ along with \eqref{e:nabla_map_back}, to write
\begin{equation}\label{e:Q_conv_W1}
\begin{split}
\| \nabla (\vy^*- \vy_h^*) \|_{L^\infty(T)} &\lesssim h_T^{-1} \left(\| \widehat \nabla (\widehat \vy^*- \widehat \vp) \|_{L^\infty(\widehat T)} + \| \widehat \nabla Q[\widehat \vy^*- \widehat \vp] \|_{L^\infty(\widehat T)} \right) \\
&\lesssim h_T^{-1} \|  \widehat \vy^*- \widehat \vp \|_{W^1_\infty(\widehat T)} \\
&  \lesssim h_T^{-1}  \| \vy^*- \vp \|_{L^\infty(T)} +  \|  \nabla ( \vy^*-  \vp) \|_{L^\infty(T)}.
\end{split}
\end{equation}
We next choose $\vy^*$ to take advantage of the piecewise smoothness \eqref{eq:extra-regularity}, namely
$$
\vp(x) := \vy^*(x_T) + \nabla \vy^*(x_T) (x-x_T).
$$
The property \eqref{e:modulus_smooth} of $\omega$ implies
$$
\| \nabla (\vy^*- \vp) \|_{L^\infty(T)} \leq \omega(h_T)
$$
and combined with $\vy^*(x)-\vy^*(x_T)=\nabla\vy^*(\xi) (x-x_T)$ for some $\xi\in T$, gives
$$
\| \vy^*- \vp \|_{L^\infty(T)} \leq h_T \, \omega (h_T).
$$
Inserting these estimates in \eqref{e:Q_conv_W1} yields the second estimate in \eqref{rec:intermediate_estim}.

The estimate on the isometry defect $\| (\nabla \vy_h^*)^T \nabla \vy_h^*- I_2\|_{L^\infty(T)}$ follows directly from the intermediate estimates \eqref{rec:intermediate_estim} and the assumption $\I[\vy^*]=I_2$
\begin{equation*}
\begin{split}
\| (\nabla \vy_h^*)^T \nabla \vy_h^* &- I_2\|_{L^\infty(T)} = \| (\nabla \vy_h^*)^T \nabla \vy_h^*-( \nabla \vy^*)^T \nabla \vy^*\|_{L^\infty(T)} \\
&\leq
( \| \nabla \vy^*\|_{L^\infty(T)}+ \| \nabla \vy_h^*\|_{L^\infty(T)} ) \| \nabla (\vy_h^*-\vy^*) \|_{L^\infty(T)}
\le c \, \omega(h_T).
\end{split}
\end{equation*}
for a constant $c$ independent of the discretization parameters; hence 
$\vy_h^*\in \mathbb A_{h,c\omega(h)}$.

\medskip
\textit{Step (ii): Broken $H^2-$ Stability}. For $\vp \in [\mathbb P_1]^3$, we set $\widehat \vp := \vp \circ F_T \in [\mathbb Q_2]^3$ to get
$$
\| D^2 \vy_h^*\|_{L^2(T)}  = \| D^2 (\vy_h^*-\vp) \|_{L^2(T)} \lesssim h_T^{-1}  \| \widehat \vy_h^*- \widehat \vp \|_{H^2(\widehat T)},
$$
in view of \eqref{e:D2_map_back_2}.
Thanks to the invariance  \eqref{e:Q_preserve_q2} and stability \eqref{e:Q_stability} of $Q$, we obtain
$$
\| D^2 \vy_h^*\|_{L^2(T)}  \lesssim  h_T^{-1}  \| Q[\widehat \vy^*- \widehat \vp] \|_{H^2(\widehat T)} \lesssim h_T^{-1}  \| \widehat \vy^*- \widehat \vp \|_{L^2(\widehat T)}.
$$
This, together with \eqref{e:L2_map_back} and a standard interpolation estimate on $T$, yields
$$
\| D^2 \vy_h^*\|_{L^2(T)}\lesssim h_T^{-2}  \|  \vy^*-  \vp \|_{L^2(T)} \lesssim \| D^2 \vy^*\|_{L^2(T)},
$$
which is the desired stability estimate. 

\medskip
\textit{Step (iii): $H^2-$Convergence}. We exploit a density argument. For any $\epsilon>0$, there exists {$\vy^\epsilon$ so that $\vy^\epsilon|_{\Omega_i} \in[H^3(\Omega_i)]^3$ and $\| \vy^*- \vy^\epsilon \|_{H^2(\Omega_i)} \leq \epsilon$ for $i=1,2$;}
$\vy^\epsilon$ may not be an isometry {and not be continuous across $\mathcal C$}. 
{Exploiting the fact that $\mathcal C$ is exactly matched by $\Th$, we split} 
\begin{equation}\label{e:rec_split}
\sum_{T\in\Th}\|\vy^*- \vy_h^*\|_{H^2(T)}^2\lesssim {\sum_{i=1}^2\|\vy^*- \vy^\epsilon\|_{H^2(\Omega_i)}^2}   + \sum_{T\in\Th}\|\vy^\epsilon - \vy_h^\epsilon\|_{H^2(T)}^2   + \sum_{T\in\Th}\|\vy_h^\epsilon - \vy_h^*\|_{H^2(T)}^2
\end{equation}
with $\vy_h^\epsilon = \widehat \vy_h^\epsilon \circ F_T^{-1}$ and $\widehat \vy_h^\epsilon := Q[\vy^\epsilon\circ F_T]$, and estimate each of the three terms separately. The first term is obviously bounded by $\epsilon^2$.

For the second term, we let $m=0,1,2$ and combine \eqref{e:Q_approx} with \eqref{Hm-est} to arrive at
\begin{equation*}
| \vy_h^\epsilon - \vy^\epsilon |_{H^m(T)} 
\lesssim h_T^{1-m} \left(\sum_{i=1}^2 \| \partial^3_{\widehat x_i} \widehat \vy^\epsilon \|_{L^2(\widehat T)}^2\right)^{1/2} 
\lesssim h_T^{1-m} \left(\sum_{i=1}^2 \| \partial^3_{\widehat x_i} (\widehat \vy^\epsilon - \widehat{I_h \vy^\epsilon}) \|_{L^2(\widehat T)}^2\right)^{1/2},
\end{equation*}
where $I_h$ is the local Lagrange interpolant onto $[\mathbb V_h^2]^3$ and $\widehat{I_h \vy^\epsilon} := I_h \vy^\epsilon \circ F_T^{-1} \in [\mathbb Q_2]^3$.
As a consequence, using the estimate \eqref{e:D3_map_back} to map back to the physical element $T$, and applying the error estimate \eqref{e:approx_Ih} for $I_h$ to the ensuing terms, we deduce
$$
| \vy_h^\epsilon - \vy^\epsilon |_{H^m(T)} \lesssim h_T^{3-m} \| \vy^\epsilon \|_{H^3(T)},
$$
for $m=0,1,2$. We thus conclude
\begin{equation*}
\sum_{T\in\Th}\|\vy^\epsilon - \vy_h^\epsilon\|_{H^2(T)}^2\lesssim h^2 {\sum_{i=1}^2\| \vy^\epsilon \|_{H^3(\Omega_i)}^2}.
\end{equation*}

It remains to estimate the third term $\| \vy_h^*- \vy_h^\epsilon \|_{H^2(T)}$. To deal with each term $| \vy_h^*- \vy_h^\epsilon |_{H^m(T)}$ for $m=0,1,2$, we let $\vp_{m-1} \in \mathbb P_{m-1}$ to be chosen later with the convention that $\vp_{-1} = \mathbf{0}$, and use the invariance $Q[\vp \circ F_T]=\vp \circ F_T \in [\mathbb Q_2]^3$ for any $\vp \in [\mathbb P_1]^3$. Combining \eqref{Hm-est} with the stability \eqref{e:Q_stability} of $Q$ yields
\begin{equation*}
\begin{split}
| \vy_h^*- \vy_h^\epsilon |_{H^m(T)} &= | \vy_h^*- \vy_h^\epsilon -\vp_{m-1} |_{H^m(T)}  \lesssim h_T^{1-m} \| Q[\widehat \vy^*- \widehat \vy^\epsilon - \vp_{m-1} \circ F_T] \|_{H^2(\widehat T)} \\
&\lesssim 
h_T^{1-m} \| \widehat \vy^*- \widehat \vy^\epsilon -\vp_{m-1}\circ F_T\|_{L_2(\widehat T)} 
\lesssim h_T^{-m} \| \vy^*- \vy^\epsilon -\vp_{m-1}\|_{L^2(T)}\\
& \lesssim \| \vy^*- \vy^\epsilon\|_{H^2(T)},
\end{split}
\end{equation*}
provided $\vp_{m-1}$ is an averaged Taylor polynomial of $\vy^*- \vy^\epsilon$. This in turn implies
\begin{equation*}
\sum_{T\in\Th}\|\vy_h^\epsilon - \vy_h^*\|_{H^2(T)}^2\lesssim{\sum_{i=1}^2\|\vy^*- \vy^\epsilon\|_{H^2(\Omega_i)}^2}\lesssim \epsilon^2.
\end{equation*}
Therefore, gathering the estimates for the three terms in \eqref{e:rec_split} we obtain
$$
\sum_{T\in\Th}\|\vy^*- \vy_h^*\|_{H^2(T)}^2 \lesssim \epsilon^2 + h^2 {\sum_{i=1}^2\| \vy^\epsilon \|_{H^3(\Omega_i)}^2}
{\lesssim} \epsilon^2,
$$
provided {$h\| \vy^\epsilon \|_{H^3(\Omega_i)}
\le\epsilon$} for $h$ sufficiently small and {$i=1,2$}. Since $\epsilon$ is arbitrary, we deduce
$$
\sum_{T\in \mathcal T_h} \| \vy^*- \vy_h^*\|^2_{H^2(T)} \to 0, 
$$ 
and in particular $\vy_h^*\to\vy^*$ in $[L^2(\Omega)]^3$, as $h\to 0^+$.

\medskip
\textit{Step (iv): Convergence of $\widetilde E_h[\vy_h^*]$}. 
Steps (iii) and (iv) show that the conditions in Lemma \ref{L:convergence-modifiedH}(ii) (strong convergence of $\widetilde{H}_h$) are fulfilled, whence $\widetilde{H}_h[\vy_h^*]\to\widetilde D^2\vy^*$ strongly in $[L^2(\Omega)]^{3\times2\times2}$. Consequently, convergence of $\widetilde E_h[\vy_h^*]$ towards $\widetilde E[\vy^*]$ reduces to the argument given in Theorem~\ref{Gamma:E_h} (ii) and is not repeated here.
\end{proof}

The next theorem guarantees convergence of discrete global minimizers towards exact global minimizers,
but it is not a standard $\Gamma$-convergence result because we assume \eqref{eq:extra-regularity} for one global minimizer. Other minimizers may fail to satisfy \eqref{eq:extra-regularity}.

\begin{thm}[convergence of global discrete minimizers with creases]\label{T:conv_crease}
Assume that a global minimizer $\vy^* \in \widetilde{\mathbb A}$ of $\widetilde E$ {satisfies} the additional regularity \eqref{eq:extra-regularity}.
{Let $\delta=\delta(h) \geq  c \, \omega(h)$ with $c$ the constant in Lemma~\ref{l:limsup_crease} and $\omega$ the modulus of smoothness in \eqref{e:modulus_smooth}.}
 If $\vy_h{\in\A_{h,\delta}}$ is an almost global minimizer of $\widetilde E_h$ in the sense that
  \begin{equation}\label{almost-min}
   \widetilde E_h[\vy_h] \le \inf_{\vw_h\in\A_{h,\delta}} \widetilde E_h[\vw_h] + \sigma
  \end{equation}
  where $\sigma, \delta \to 0$ as $h\to0^+$, then $\{\vy_h\}_{h>0}$ is precompact in $[L^2(\Omega)]^3$
  and every cluster point $\vy$ belongs to ${\widetilde\A}$ and is a global minimizer of $\widetilde{E}$,
  namely $\widetilde E[\vy]=\inf_{\vw\in\widetilde\A}\widetilde E[\vw]$. Moreover, up to a subsequence (not relabeled) 
  %
\begin{equation}\label{conv-Et}
 \widetilde  E[\vy] = \lim_{h\to0^+}\widetilde E_h[\vy_h].
\end{equation}
\end{thm}
\begin{proof}
The $\liminf$ property follows along the lines of Theorem~\ref{Gamma:E_h} (i) because it is based on Lemmas \ref{coercivity:E_h_bi} and \ref{l:conv_nl}, which remain valid in this context, as well as Lemma \ref{L:convergence-modifiedH} (i) (weak convergence of $\widetilde{H}_h$) instead of Lemma \ref{weak-conv} (weak convergence of $H_h$) and the weak lower semicontinuity of the $L^2$-norm.
  Therefore, there is $\vy \in \widetilde\A$ such that (up to a subsequence not relabelled) $\vy_h \to \vy$ in $[L^2(\Omega)]^3$ and
\[
  \widetilde E[\vy] \leq \liminf_{h\to 0^+}\widetilde E_h[\vy_h]. 
\]
  
To show that $\vy$ is a global minimizer of $\widetilde E$ we resort to the extra regularity \eqref{eq:extra-regularity} of the global minimizer $\vy^*\in\widetilde{\mathbb{A}}$ of $\widetilde E$. Let $\{ \vy_h^*\}_{h>0} \subset [L^2(\Omega)]^3$ be the sequence provided by Lemma~\ref{l:limsup_crease} ($\limsup$ property with creases), which satisfies
\[
\vy_h^* \in \mathbb A_{h,c\omega(h)},
\quad
\widetilde E_h[\vy_h^*] \to \widetilde E[\vy^*]
\]
as $h\to 0^+$. In view of \eqref{almost-min} and $\vy_h \in \mathbb{A}_{h,\delta(h)}$, we end up with
$$
\widetilde E[\vy] \leq \liminf_{h\to 0^+}\widetilde E_h[\vy_h] \leq \limsup_{h\to 0^+} \big(\widetilde E_h[\vy_h^*]+\sigma \big) =\widetilde E[\vy^*] = \inf_{\vw\in\widetilde{A}} \widetilde E[\vw].
$$
Therefore, $\vy$ is indeed a global minimizer of $\widetilde E$ and \eqref{conv-Et} is valid.
\end{proof}

\section{Discrete gradient flow}\label{sec: gf-bilayer}

Solving the minimization problem \eqref{discrete_minimization_bilayer} is a nontrivial
task because it entails enforcing the {\it nonconvex constraint} $D_h[\vy_h](x_T)\le\delta$ at element barycenters $x_T$.
We now develop a discrete gradient flow with respect to the $H^2_h$ metric \eqref{def:H2norm} that linearizes
the isometry constraint according to \eqref{linearization-discrete}. We refer to
\cite{bartels2013approximation,bartels2018modeling,bartels2017bilayer,bonito2020ldg,bonito2020ldg-na,bonito2019dg,bonito2020discontinuous} and especially to S. Bartels and Ch. Palus \cite{bartels2020stable} for similar gradient flows.

We start recalling the notion of linearized isometry constraint for $\vv_h,\vy_h\in[\V_h^k]^3$
\begin{equation}\label{linear-discrete}
  L[\vv_h;{\vy_h}](x_T) = \big[ \nabla\vv_h^T \nabla{\vy_h} + \nabla{\vy_h}^T \nabla \vv_h  \big](x_T)
  \quad\forall \, T\in\Th
\end{equation}
and defining a tangent space associated with the isometry constraint for any ${\vy_h} \in \A_{h,\delta}$
\begin{equation}\label{linearized-cons}
\mathcal{F}_{h}({\vy_h}):=\Big\{\vv_h\in\V^k_h(\bz,\bz): \quad L[\vv_h;{\vy_h}](x_T)=0 \quad\forall\, T\in\Th\Big\}.
\end{equation}
Given $\vy^0_h\in\A_{h,0}$ (i.e, $\vy^0_h$ satisfies the isometry constraint $\I[\vy^0_h](x_T)=I_2$ at each barycenter $x_T$), the discrete gradient flow consists of seeking recursively $\delta\vy^{n+1}_h:= \vy_n^{n+1}-\vy_h^n \in\mathcal{F}_{h}(\vy^n_h)$ such that
\begin{equation}\label{gf_bilayer}
\frac{1}{\tau}\big\langle\delta\vy^{n+1}_h,\vv_h\big\rangle_{H_h^2(\Omega)}+a_h(\delta\vy^{n+1}_h,\vv_h)=-a_h(\vy^{n}_h,\vv_h)+\ell[\vy^n_h](\vv_h) \quad\forall \, \vv_h\in\mathcal{F}_{h}(\vy^n_h).
\end{equation}
Here $\tau>0$ is a pseudo time step and $a_h$ is the bilinear form corresponding to the variational derivative of the bending energy $B_h[\vy_h]$ defined in \eqref{discrete-energy-0}, i.e.
\begin{equation}\label{e:a_h}
\begin{split}
  a_h({\vy_h},\vv_h) &:=\int_{\Omega}H_h[{\vy_h}]:H_h[\vv_h]
  \\ &
  +\gamma_1(\h^{-1}[\nabla_h{\vy_h}],[\nabla_h\vv_h])_{L^2(\Gh^a)}+\gamma_0(\h^{-3}[{\vy_h}],[\vv_h])_{L^2(\Gh^a)}.
\end{split}
\end{equation}
The linear form $\ell[\vy^n_h](\vv_h)$ on $\vv_h$ is the first variation of the cubic energy $C_h[\vy^n_h]$,
defined in \eqref{quadrature-2}, along the direction of the test function $\vv_h$ and is given by 
\begin{align*}
\ell[\vy^n_h](\vv_h):=&\sum_{i,j=1}^2\sum_{T\in\Th}|T| \, \overline{H}_h[\vv_h]_{ij}\cdot(\partial_1\vy_h^n\times\partial_2\vy_h^n)(x_T) \, Z_{ij} \\
&+\sum_{i,j=1}^2\sum_{T\in\Th}|T| \, \overline{H}_h[\vy_h^n]_{ij}\cdot(\partial_1\vv_h\times\partial_2\vy_h^n)(x_T) \, Z_{ij} \\
& +\sum_{i,j=1}^2\sum_{T\in\Th}|T| \, \overline{H}_h[\vy_h^n]_{ij}\cdot(\partial_1\vy_h^n\times\partial_2\vv_h)(x_T) \, Z_{ij};
\end{align*}
recall that both $\overline{H}_h[\vv_h]$ and $Z$ are piecewise constant on $\Th$.
The explicit treatment of $\vy^n_h$ in $\ell[\vy^n_h](\vv_h)$ is similar to the scheme proposed and analyzed
by S. Bartels and Ch. Palus \cite{bartels2020stable}.
{We note that \eqref{gf_bilayer} is the discrete Euler-Lagrange equation for the augmented energy $E_h[\vy_h^{n+1}]+\frac{1}{2\tau}\|\vy_h^{n+1}-\vy_h^{n}\|_{H^2_h(\Omega)}^2$, except that the nonlinear terms corresponding to the cubic energy $C_h[\vy_h^{n+1}]$ are linearized by $\ell[\vy^n_h](\vv_h)$. We refer to the nonlinear continuous Euler-Lagrange equation \eqref{eq:energy-E-L} for a comparison.}

For later use, we note that Lemma~\ref{L2bound-modified} (stability of $\overline{H}_h[\vv_h]$) yields
\begin{equation*}
\begin{split}
  |\ell[{\vy_h}](\vv_h)| \lesssim  \sqrt{1+\delta} \| Z\|_{L^\infty(\Omega)}
  \Big(\| \nabla_h\vv_h \|_{L^2(\Omega)}\| {\vy_h}\|_{H^2_h(\Omega)}+
\| \nabla_h{\vy_h} \|_{L^2(\Omega)}\| \vv_h\|_{H^2_h(\Omega)} \Big),
\end{split}
\end{equation*}
provided ${\vy_h} \in \A_{h,\delta}$ because $|\nabla{\vy_h}(x_T)| \le {\sqrt{2(1+\delta)}}$
from Lemma~\ref{lem:bi_iso_control} (pointwise isometry {constraint}) and the inverse inequality
  $|\nabla \vv_h(x_T)| \lesssim h_T^{-1} \|\nabla\vv_h\|_{L^2(T)}$ applies.
In addition, we rewrite the Friedrichs inequality \eqref{friedrichs} as follows
\begin{equation}\label{e:friederich}
\| \nabla_h {\vy_h} \|_{L^2(\Omega)} \lesssim \|  {\vy_h} \|_{H^2_h(\Omega)} + C_{\vvarphi,\Phi}, \qquad \forall \vw_h \in \V_h^k(\vvarphi,\Phi),
\end{equation}
where $C_{\vvarphi,\Phi} = \|\vvarphi\|_{H^1(\Omega)} + \|\Phi\|_{H^1(\Omega)}$. From these estimates we deduce the existence of a constant $c_{nl}$ such that for ${\vy_h} \in \A_{h,\delta}$ and $\vv_h \in \V^k_h(\bz,\bz)$ we have
\begin{equation}\label{e:contlinnonlin}
\begin{split}
\big|\ell[{\vy_h}](\vv_h)\big|  \leq  c_{nl}  \sqrt{1+\delta} \| Z\|_{L^\infty(\Omega)}(\| {\vy_h} \|_{H^2_h(\Omega)}\!+ C_{\vvarphi,\Phi}) \| \vv_h\|_{H^2_h(\Omega)}.
\end{split}
\end{equation}

\subsection{Energy stability and admissibility}

We discuss in this section the energy reduction property of the gradient flow and, although the isometry constraint
$D_h[\vy_h^n](x_T)=0$ is relaxed and linearized in the iterative scheme, the deviation of $D_h[\vy_h^n](x_T)$
from $0$ is controlled by a parameter $\delta>0$ provided $\tau$ is sufficiently small.
These results  rely on a discrete inverse inequality on finite dimensional subsets.

\begin{lemma}[discrete Sobolev inequality]\label{lem:discrete-inverse}
Let $W_h \subset  \Pi_{T\in\Th}H^1(T)$ be a  finite element space subordinated to the partition $\Th$.
For any $w_h \in W_h$ there holds
\begin{equation}\label{eq:discrete-sobolev-ineq}
  \|w_h\|_{L^{\infty}(\Omega)}\lesssim \big(1+|\log h_{\min}|\big)^{\frac12}
  \Big(\|w_h\|_{L^2(\Omega)} +\|\nabla_h w_h\|_{L^2(\Omega)}
  +\|\h^{-\frac12}\jump{w_h}\|_{L^2(\Gh^0)}\Big),
\end{equation}
where  $h_{\min}:=\min_{T\in\Th}h_T$.
\end{lemma}
\begin{proof}
%
We denote by $\Pi_h: \Pi_{T\in\Th}H^1(T)\to\V_h^k\cap H^1(\Omega)$ the smoothing operator from \cite{bonito2010quasi, bonito2019dg} and recall that it satisfies
\begin{equation*}\label{eqn:smooth_interp_H1}
\|\nabla \Pi_h w_h\|_{L^2(\Omega)} + \|\h^{-1}(w_h-\Pi_h w_h)\|_{L^2(\Omega)} \lesssim  \|\nabla_h w_h\|_{L^2(\Omega)}
 +\|\h^{-\frac{1}{2}}\jump{w_h}\|_{L^2(\Gh^0)},
\end{equation*}
and
\begin{equation*}\label{eqn:smooth_interp_L2}
\|\Pi_hw_h\|_{L^2(\Omega)}\lesssim\|w_h\|_{L^2(\Omega)}.
\end{equation*}
Therefore, combining the triangle and inverse inequalities {for $w_h\in W_h$} implies
\begin{align*}
\|w_h\|_{L^{\infty}(\Omega)} &\lesssim \|w_h-\Pi_hw_h\|_{L^{\infty}(\Omega)}+\|\Pi_hw_h\|_{L^{\infty}(\Omega)} \\
&\lesssim \|\h^{-1}(w_h-\Pi_hw_h)\|_{L^2(\Omega)}+\big(1+|\log h_{\min}|\big)^{\frac12} \|\Pi_hw_h\|_{H^1(\Omega)},
\end{align*}
in view of the following discrete Sobolev inequality in 2d \cite{bramble1986construction,brenner2007mathematical}
\begin{equation*}
\|\Pi_hw_h\|_{L^{\infty}(\Omega)}\lesssim \big(1+|\log h_{\min}|\big)^{\frac12}\|\Pi_hw_h\|_{H^1(\Omega)}.
\end{equation*}
This leads to the assertion upon applying the preceding estimates for $\Pi_h$.
\end{proof}
 
We are now in a position to prove the main result of this section, namely that the gradient flow is
energy decreasing and controls the isometry defect.

\begin{thm}[properties of gradient flow]\label{thm:gf}
    Let $\{ \vy^0_h\}_{h>0} \subset \A_{h,0}$ satisfy $E_h[\vy^0_h]\le c_0$ with $c_0$ a constant
    independent of $h$
    and let all subdivisions $\Th$ be such that $|\log h_{\min}| \geq 1$. Let $N$ be the number of
    iterations of the gradient flow and $\tau$ be its pseudo-time step.
    There exists a constant $\alpha_1= \alpha_1(\vvarphi,\Phi,Z)>0$ independent of $h$ and $N$
    such that if $\tau \le (2\alpha_1|\log h_{\min}|)^{-1}$, then the energy $E_h[\vy^N_h]$ satisfies
\begin{equation}\label{stability}
E_h[\vy^N_h]+\frac{1}{2\tau}\sum_{n=0}^{N-1}\|\delta\vy^{n+1}_h\|_{H_h^2(\Omega)}^2\le E_h[\vy^0_h].
\end{equation}
In addition, there are  constants $\alpha_2=\alpha_2(\vvarphi,\Phi,Z) > 0$
and $\alpha_3 > 0$, both independent of $h$ and $N$, such that the isometry defect $D_h[\vy_h^{N}]$ satisfies
\begin{equation}\label{control-defect-bi}
{D_h[\vy_h^{N}](x_T)} \le\alpha_3\tau|\log h_{\min}|\big(E_h[\vy_h^0]+\alpha_2\big) \quad\forall \, T\in\Th.
\end{equation} 
\end{thm}
\begin{proof}
We proceed by induction. We first note that estimates \eqref{stability} and \eqref{control-defect-bi}
hold trivially for $N=0$ and $\vy_h^0 \in \A_{h,0}$.
Therefore, we assume that \eqref{stability} and \eqref{control-defect-bi} are valid for $N=1,...,M$ with positive constants $\alpha_1,\alpha_2,\alpha_3$ to be specified below and prove the validity of the same estimates for $N=M+1$ with the same constants $\alpha_1,\alpha_2,\alpha_3$.

We split the proof into four steps with the following roadmap. After deriving an intermediate estimate in Step (i), we prove \eqref{stability} in Step (ii) and \eqref{control-defect-bi} in Step (iii) under suitable restrictions on $\alpha_i$, $i=1,2,3$. In Step (iv), we show that it is always possible to find values of these parameters satisfying the desired restrictions. In this proof, the generic constants hidden in the symbol ``$\lesssim$'' are not only independent of $h$ but also of $\tau$, $M$ and $\alpha_i$, $i=1,2,3$.

\medskip
\textit{Step (i): intermediate estimate}.
We take $\vv_h=\delta\vy^{M+1}_h\in\mathcal{F}_h(\vy_h)$ in \eqref{gf_bilayer} for $n=M$, use the elementary relation
$2 b (b - a) = (b-a)^2 + b^2 - a^2$ and discard the positive term $(b-a)^2 = a_h(\delta\vy^{M+1}_h,\delta\vy^{M+1}_h)$ to write
\begin{equation}\label{e:first}
  \|\delta\vy^{M+1}_h\|_{H_h^2(\Omega)}^2+\frac{\tau}{2}a_h\big(\vy^{M+1}_h,\vy^{M+1}_h\big)
  -\frac{\tau}{2}a_h\big(\vy^{M}_h,\vy^{M}_h\big)
\leq \tau  \ell\big[\vy_h^M\big](\delta \vy_h^{M+1}).
\end{equation}
Using \eqref{control-defect-bi} with $N=M$ (induction assumption), together with the restriction on $\tau$ and the uniform bound $E_h[\vy_h^0]\le c_0$, we obtain
\begin{equation}\label{eq:yk-C1}
\big|D_h[\vy_h^M](x_T)\big|\le \frac{\alpha_3}{2\alpha_1}\big(c_0+\alpha_2\big)= \delta.
\end{equation}
To simplify the expressions below, we let $\alpha=(\alpha_i)_{i=1}^3$ and  $c_\alpha^2 = 1+\delta$,  whence
\begin{equation}\label{eq:c-alpha}
c_\alpha :=\sqrt{\frac{\alpha_3}{2\alpha_1}\big(c_0+\alpha_2\big)+1}.
\end{equation}
Estimate \eqref{eq:yk-C1} shows that $\vy_h^M \in \A_{h,\delta}$ with $\delta = c_\alpha^2-1$
which in turn implies
\begin{equation*} 
|\partial_i \vy_h^M(x_T) | \leq c_\alpha, \qquad i=1,2, \quad T \in \Th,
\end{equation*}
according to Lemma~\ref{lem:bi_iso_control} (pointwise isometry constraint).
Substituting into \eqref{e:contlinnonlin} yields
\begin{equation*}
\big|\ell\big[\vy_h^M\big](\delta\vy_h^{M+1}) \big|  \leq  c_{nl} c_\alpha \| Z\|_{L^\infty(\Omega)}   \| \delta\vy_h^{M+1} \|_{H^2_h(\Omega)} \big(\|  \vy_h^M\|_{H^2_h(\Omega)} + C_{\vvarphi,\Phi} \big).
\end{equation*}
Inserting this back into \eqref{e:first} and using Young's inequality to absorb the term $\|\delta\vy^{k+1}_h\|_{H_h^2(\Omega)}^2$ in the left hand side of \eqref{e:first}, gives the estimate
\begin{equation} \label{e:second}
\begin{split}
\frac 1 2 \|\delta\vy^{M+1}_h\|_{H_h^2(\Omega)}^2 &+\frac{\tau}{2}a_h\big(\vy^{M+1}_h,\vy^{M+1}_h\big)\\
& \leq \frac{\tau}{2}a_h\big(\vy^{M}_h,\vy^{M}_h\big) +  \tau^2 c_{nl}^2 c_\alpha^2 \| Z\|_{L^\infty(\Omega)}^2 \big(\|  \vy_h^M\|^2_{H^2_h(\Omega)} + C_{\vvarphi,\Phi}^2 \big).
\end{split}
\end{equation}

We next improve upon \eqref{e:second} by deriving a {uniform} bound for the right-hand side.
According to \eqref{eq:yk-C1}, the isometry defect is controlled by $\delta =c_\alpha^2-1$.
Moreover, \eqref{stability} for $N=M$ (induction assumption) implies that $E_h[\vy^M_h] \leq E_h[\vy_h^0]\le c_0 $. Hence, the coercivity estimate \eqref{e:coercfull} reads
\begin{equation}\label{eqn:est-hypo-yk-2}
(2c_{coer})^{-1}\|\vy^M_h\|^2_{H_h^2(\Omega)} \leq c_0 + \tilde c_{coer}c_\alpha^4.
\end{equation}
Estimate \eqref{eqn:coercivity_bi} can be rewritten in terms of the bilinear form $a_h$ as follows
\begin{equation}\label{eq:a_h-conti}
c_{coer}^{-1} \|\vv_h\|^2_{H_h^2(\Omega)} \le \frac12 a_h (\vv_h,\vv_h) \le c_{cont} \|\vv_h\|^2_{H_h^2(\Omega)} \qquad \forall \vv_h \in \V_h^k(\vvarphi,\Phi).
\end{equation}
 Since $|\log h_{\min}|\geq 1$, $\tau$ satisfies $\tau \leq \frac{1}{2\alpha_1} \leq 1$ provided $\alpha_1\geq \frac 12$, which is our first restriction on $\alpha_1$. Combining this with \eqref{eqn:est-hypo-yk-2} and \eqref{eq:a_h-conti} with $\vv_h=\vy^M_h$, and replacing back into \eqref{e:second}, gives the desired intermediate estimate
\begin{equation}\label{intermediate-2}
\frac{1}{2\tau}\|\delta\vy^{M+1}_h\|_{H_h^2(\Omega)}^2+\frac{1}{2}a_h(\vy^{M+1}_h,\vy^{M+1}_h)\le \psi_1(c_\alpha)
\end{equation}
where $\psi_1(c_\alpha) \geq 0$ is a positive increasing function of its argument $c_\alpha$ but
whose specific expression
is irrelevant except that it is independent of $h$, $M$ and depends on $\alpha=(\alpha_i)_{i=1}^3$ only through
the variable $c_\alpha$ rather than separately on each $\alpha_i$.

\medskip
\textit{Step (ii): proof of \eqref{stability} for $N=M+1$.}
In view of \eqref{e:first}, and telescopic cancellation, this requires dealing with the cubic term
$\ell[\vy_h^M](\delta \vy_h^{M+1})$. Using the identity 
\begin{equation*}
\begin{split}
&a^{M+1}b^{M+1}c^{M+1}-a^Mb^Mc^M\\
&\qquad =(a^{M+1}-a^M)b^{M+1}c^{M+1}+a^M(b^{M+1}-b^M)c^{M+1}+a^Mb^M(c^{M+1}-c^M),
\end{split}
\end{equation*}
we deduce
\begin{equation*}
\begin{split}
(a^{M+1} -a^M)b^{M}c^{M} &+a^M(b^{M+1}-b^M)c^{M}+a^Mb^M(c^{M+1}-c^M)\\
&=a^{M+1}b^{M+1}c^{M+1}-a^{M}b^{M}c^{M}-(a^{M+1}-a^M)(b^{M+1}-b^M)c^{M+1}\\
&-(a^{M+1}-a^M)b^{M}(c^{M+1}-c^{M})-a^M(b^{M+1}-b^M)(c^{M+1}-c^M),
\end{split}
\end{equation*}
and rewrite $\ell[\vy_h^M](\delta \vy_h^{M+1})$ as follows:
\begin{equation*}\label{eq:cubic-rewrite}
\begin{split}
 \ell[\vy_h^M](\delta \vy_h^{M+1})&=\sum_{i,j=1}^2 \sum_{T\in \Th} |T| \, \overline{H}_h[\vy^{M+1}_h]_{ij}\cdot(\partial_1\vy_h^{M+1}\times\partial_2\vy_h^{M+1})(x_T) \, Z_{ij} \\
& - \sum_{i,j=1}^2 \sum_{T\in \Th}|T| \, \overline{H}_h[\vy^{M}_h]_{ij}\cdot(\partial_1\vy_h^M\times\partial_2\vy_h^M)(x_T)\, Z_{ij} \\
& - \sum_{i,j=1}^2 \sum_{T\in \Th}|T| \, \overline{H}_h[\delta\vy^{M+1}_h]_{ij}\cdot(\partial_1\delta \vy_h^{M+1}\times\partial_2\vy_h^{M+1})(x_T) \, Z_{ij}\\
& - \sum_{i,j=1}^2 \sum_{T\in \Th}|T| \, \overline{H}_h[\delta\vy^{M+1}_h]_{ij}\cdot(\partial_1\vy_h^{M}\times\partial_2\delta \vy_h^{M+1})(x_T) \, Z_{ij}\\ 
& - \sum_{i,j=1}^2 \sum_{T\in \Th}|T| \, \overline{H}_h[\vy^{M}_h]_{ij}\cdot(\partial_1\delta\vy^{M+1}_h\times \partial_2\delta\vy^{M+1}_h)(x_T) \, Z_{ij}.
\end{split}
\end{equation*} 
We note that the first two terms are exactly the cubic energies $C_h[\vy^{M+1}_h]$ and $C_h[\vy_h^M]$ and together with the bending energies $B_h[\vy^{M+1}_h]=\frac12 a_h\big(\vy^{M+1}_h,\vy^{M+1}_h\big)$ and $B_h[\vy_h^M]=\frac12 a_h\big(\vy^{M}_h,\vy^{M}_h\big)$ in \eqref{e:first} give rise to the
full energies $E_h[\vy^{M+1}_h]$ and $E_h[\vy_h^M]$ in \eqref{stability}. In contrast, the last three terms must be estimated and absorbed into the remaining term $\|\delta\vy^{M+1}_h\|_{H_h^2(\Omega)}^2$ in \eqref{stability}. To this end,
we combine the Friedrichs inequality \eqref{e:friederich} for ${\vy_h}=\vy_h^{M+1}, \vy_h^{M}\in\V_h^k(\vvarphi,\Phi)$
and ${\vy_h}=\delta\vy_h^{M+1}\in\V_h^k(\bz,\bz)$,
and Lemma~\ref{L2bound-modified} (stability of $\overline{H}_h[\vv_h]$), to obtain
\begin{equation*}
\begin{split}
 &\|\delta\vy^{M+1}_h\|_{H_h^2(\Omega)}^2+\tau E_h[\vy^{M+1}_h]-\tau E_h[\vy^{M}_h]\\
&\quad \lesssim \tau  \|\delta\vy^{M+1}_h\|_{H^2_h(\Omega)} \|\nabla \delta\vy_h^{M+1}\|_{L^{\infty}(\Omega)}\Big(\|\vy_h^{M+1}\|_{H^2_h(\Omega)}+\|\vy_h^{M}\|_{H^2_h(\Omega)}+C_{\vvarphi,\Phi}\Big),
\end{split}
\end{equation*}
where the symbol $\lesssim$ hides $\| Z\|_{L^\infty(\Omega)}$.
To estimate the $L^\infty$-norm on the right-hand side, we resort to Lemma~\ref{lem:discrete-inverse} (discrete Sobolev inequality)
\begin{align*}
\|\delta\vy^{M+1}_h\|_{H_h^2(\Omega)}^2 &+\tau E_h[\vy^{M+1}_h]-\tau E_h[\vy^{M}_h]\\
& {\le c_Z} \tau |\log h_{\min}| \|\delta\vy^{M+1}_h\|^2_{H^2_h(\Omega)} \Big(\|\vy_h^{M+1}\|_{H^2_h(\Omega)}+\|\vy_h^{M}\|_{H^2_h(\Omega)}+C_{\vvarphi,\Phi} \Big),
\end{align*}
because $| \log h_{\min}| \geq 1$, {with constant $c_Z$ depending on $\| Z\|_{L^\infty(\Omega)}$
and the constants hidden in \eqref{e:friederich} and \eqref{eq:discrete-sobolev-ineq}.}
Moreover, the coercivity estimate \eqref{eqn:coercivity_bi} of $B_h$, written now as
\[
\|\vy_h^{M+1}\|_{H^2_h(\Omega)}^2 \le c_{coer} B_h[\vy_h^{M+1}]
= \frac{c_{coer}}{2} a_h \big(\vy_h^{M+1},\vy_h^{M+1} \big) \le c_{coer} \psi_1(c_\alpha)
\]
according to \eqref{intermediate-2}, together with \eqref{eqn:est-hypo-yk-2}  guarantees that
\begin{align*}\label{eqn:coercivity_bi_Mp1} 
\|\vy_h^{M+1}\|_{H_h^2(\Omega)}+\|\vy_h^{M}\|_{H_h^2(\Omega)}+C_{\vvarphi,\Phi}
\leq \psi_2(c_\alpha),
\end{align*}
where $\psi_2(c_\alpha)$ is a positive increasing function of the argument $c_\alpha$ which is
independent of $h, M$, and the individual parameters $(\alpha_i)_{i=1}^3$.
Substituting back yields
\begin{equation*}
\|\delta\vy^{M+1}_h\|_{H_h^2(\Omega)}^2+\tau E_h[\vy^{M+1}_h]-\tau E_h[\vy^{M}_h] \, {\le c_Z} \tau |\log h_{\min}| \psi_2(c_\alpha) \|\delta\vy^{M+1}_h\|_{H_h^2(\Omega)}^2.
\end{equation*}
Consequently, since $\tau \le (2\alpha_1|\log h_{\min}|)^{-1}$, it remains to choose
$(\alpha_i)_{i=1}^3$ so that
\begin{equation*} 
\psi_2(c_\alpha) \le \alpha_1{{c_Z}^{-1}},
\end{equation*}
to derive the desired estimate \eqref{stability} for $N=M+1$.
The validity of this estimate will be justified in Step (iv).

\medskip
\textit{Step (iii): proof of \eqref{control-defect-bi} for $N=M+1$.}
Since $\delta\vy_h^{n+1}\in\mathcal{F}_{h}(\vy_h^{n})$, expanding $\I[\vy_h^{n+1}](x_T)=[(\nabla\vy_h^{n+1})^T\nabla\vy_h^{n+1}](x_T)$ and using the definition \eqref{linearized-cons} of $\mathcal{F}_{h}(\vy_h^{n})$ yields
\begin{equation}\label{eq:bi-gf-ortho}
{
D_h[\vy_h^{n+1}](x_T) \le D_h[\vy_h^n](x_T) + \big|\I[\delta\vy_h^{n+1}](x_T)\big| \quad \forall \, T\in \Th.
}
\end{equation}
Applying Lemma~\ref{lem:discrete-inverse} (discrete Sobolev inequality), followed
by the discrete Friedrichs inequality \eqref{e:friederich} to estimate
$\|\nabla_h\delta\vy_h^{n+1}\|_{L^2(\Omega)}$, implies
\begin{equation}\label{eq:bi-est-inc}
\big| \I[\delta\vy_h^{n+1}](x_T) \big| \le\|\nabla_h\delta\vy_h^{n+1}\|_{L^{\infty}(T)}^2
{\le c_F} |\log h_{\min}| \|\delta\vy_h^{n+1}\|_{H_h^2(\Omega)}^2
\end{equation}
{where $c_F>0$ is a constant that combines constants hidden in \eqref{eq:discrete-sobolev-ineq} and \eqref{e:friederich},}
because $|\log(h_{\min})|\geq 1$ and $\delta\vy_h^{n+1} \in \V_h^k(\bz,\bz)$.
Summing over $0 \le n \le M$, and using telescopic cancellation along with $\vy_h^0 \in \A_{h,0}$, yields
\begin{equation}\label{e:iso_tmp}
{D_h[\vy_h^{M+1}](x_T) \le c_F} |\log h_{\min}|\sum_{n=0}^{M}\|\delta\vy_h^{n+1}\|_{H_h^2(\Omega)}^2.
\end{equation}
Exploiting the energy decay \eqref{stability}, proved for $N=M+1$ in Step (ii), gives
\begin{equation}\label{e:increment_tmp2}
\sum_{n=0}^{M}\|\delta\vy_h^{n+1}\|_{H_h^2(\Omega)}^2\le 2\tau \big(E_h[\vy^0_h]-E_h[\vy^{M+1}_h] \big).
\end{equation}
We now need a lower bound for the energy $E_h[\vy^{M+1}_h]$, which is a consequence of
\eqref{e:coercfull} provided $\vy_h^{M+1} \in \A_{h,\epsilon}$ for some $\epsilon>0$.
To this end, we resort again to \eqref{eq:bi-gf-ortho}. We first bound the second term on the
right-hand side upon combining the intermediate estimate \eqref{intermediate-2}
for $\|\delta \vy_h^{M+1}\|_{H^2_h(\Omega)}$ with \eqref{eq:bi-est-inc}
\begin{equation*}\label{eq:est-yMp1}
\big|\I[\delta\vy_h^{M+1}](x_T)\big|\leq 2\tau{c_F}|\log h_{\min}| \psi_1(c_\alpha)\le {c_F}\alpha_1^{-1} \psi_1(c_\alpha),
\end{equation*}
because $\tau \le (2\alpha_1|\log h_{\min}|)^{-1}$.
Using this bound in \eqref{eq:bi-gf-ortho}, along with the fact that $\vy_h^M \in \A_{h,\delta}$ for $\delta= c_\alpha^2-1$ according to the induction assumption \eqref{eq:yk-C1}, implies
\begin{equation*}
\big|D_h[\vy_h^{M+1}]\big](x_T)\big|\le c_\alpha^2 -1 + {c_F}\alpha_1^{-1} \psi_1(c_\alpha) =: \epsilon,
\end{equation*}
whence $\vy_h^{M+1}\in\A_{h,\epsilon}$. Inserting $E_h[\vy_h^{M+1}]\ge - \tilde{c}_{coer}(1+\epsilon)^2$
from \eqref{e:coercfull} into \eqref{e:increment_tmp2} gives
\begin{equation*}\label{e:increment_tmp}
\sum_{n=0}^{M}\|\delta\vy_h^{n+1}\|_{H_h^2(\Omega)}^2\le 2\tau \Big(E_h[\vy^0_h] + \tilde{c}_{coer}\big(c_\alpha^2 + {c_F}\alpha_1^{-1} \psi_1(c_\alpha) \big)^2\Big).
\end{equation*}
Returning to \eqref{e:iso_tmp}, we arrive at
$$
{D_h[\vy_h^{M+1}](x_T) \leq 2c_F} \tau |\log h_{\min}| \,
\Big(E_h[\vy^0_h] + \tilde{c}_{coer}\big(c_\alpha^2 + {c_F}\alpha_1^{-1} \psi_1(c_\alpha)\big)^2 \Big),
$$
{and we emphasize that $c_F$} is a constant independent of $h$, $M$ and $(\alpha_i)_{i=1}^3$.
The desired control on the isometry defect \eqref{control-defect-bi} is thus guaranteed provided 
\begin{equation*} 
  \alpha_3\ge {2c_F},\qquad
  \alpha_2 \geq \tilde{c}_{coer}\big(c_\alpha^2 + {c_F}\alpha_1^{-1} \psi_1(c_\alpha)\big)^2.
\end{equation*}

\medskip
\textit{Step (iv): choice of parameters.} $\alpha = (\alpha_i)_{i=1}^3$ must satisfy
\[
\alpha_1 \geq \frac 1 2,
\quad
\psi_2(c_\alpha) \leq \alpha_1{{c_Z}^{-1}},
\quad
\alpha_3\ge {2c_F},
\quad
\alpha_2 \geq \tilde{c}_{coer}\big(c_\alpha^2 + {c_F}\alpha_1^{-1} \psi_1(c_\alpha)\big)^2 =:
\psi_3(\alpha_1,c_\alpha),
\]
where $c_\alpha$ is defined in \eqref{eq:c-alpha} and the functions $\psi_1,\psi_2$ are positive
and increasing in their arguments. One admissible set of parameters is
\[
\alpha_3 = {2c_F},
\quad
\alpha_2 = \alpha_1^{\frac12},
\]
with $\alpha_1\ge\frac 12$ sufficiently large. In fact, we note that as $\alpha_1 \to \infty$
\[
c_\alpha \downarrow 1,
\quad
\psi_1(c_\alpha) \downarrow \psi_1(1) \ge 0,
\quad
\psi_2(c_\alpha) \downarrow \psi_2(1) \ge 0,
\quad
\psi_3(\alpha_1,c_\alpha) \downarrow \tilde{c}_{coer},
\]
and the condition $\alpha_1 \ge \max\big\{\frac12,{c_Z}\psi_2(c_\alpha),\psi_3(\alpha_1,c_\alpha)^2\big\}$ admits
a solution. This completes the induction argument.
\end{proof}

It is worth realizing that the $\ell^\infty$-control of the isometry defect \eqref{control-defect-bi} implies that $\vy_h\in\mathbb{A}_{h,\delta}$ provided $\tau$ is so small that
\[  
\alpha_3\tau |\log h_{\min}| \big(E_h[\vy_h^0] + \alpha_2 \big) \le \delta,
\]
where $\mathbb{A}_{h,\delta}$ is defined in \eqref{admissible-discrete}.
This property is novel in the context of DG approximations \cite{bonito2020ldg,bonito2020ldg-na,bonito2020discontinuous,bonito2019dg,ntogkas2018non}, but is inspired by a similar one at element vertices shown by S. Bartels and Ch. Palus for Kirchhoff elements \cite{bartels2020stable}. It is responsible for the explicit treatment of the cubic term in \eqref{gf_bilayer}, which in turn converts \eqref{gf_bilayer} into a {\it linear system} to solve for $\delta\vy_h^{n+1}$. The fact that $H^2(\Omega)$ does not embed in $W^1_\infty(\Omega)$ in two dimensions, but is borderline instead, explains the critical nature of the estimates \eqref{stability} and \eqref{control-defect-bi}. The discrete $H^2$-metric of the gradient flow \eqref{gf_bilayer}, combined with Lemma \ref{lem:discrete-inverse} (discrete Sobolev inequality), makes it possible to exploit this borderline structure discretely at the expense of a log term. No weaker metric for the gradient flow than $H^2$ would allow for $\ell^\infty$-control of the isometry defect.

\subsection{Lagrange multipliers for the isometry constraint}\label{sec:inf-sup_bi}

We enforce tangential variations $\delta\vy^{n+1}_h\in\mathcal{F}_{h}(\vy^n_h)$ at each step of the gradient flow using  Lagrange multipliers within the space of symmetric piecewise constant tensors
\begin{equation*}
\Lambda_h:=\left\{\lambda_h:\Omega\to\mathbb{R}^{2\times2}: \,\, \lambda_h^T=\lambda_h, \,\, \lambda_h\in\big[\V_h^0\big]^{2\times 2}\right\}.
\end{equation*}
 For any $\vw_h \in\V^k_h(\vvarphi,\Phi)$, we define the bilinear form $b_h(\vw_h ; \cdot,\cdot)$ on $\V^k_h(\mathbf{0},\mathbf{0})\times\Lambda_h$ as
\begin{equation}\label{def:bh-bi}
b_h(\vw_h;\vv_h,\vmu_h):=\sum_{T\in\Th}|T| \, L[\vv_h;\vw_h](x_T):\vmu_h,
\end{equation}
where the linearized isometry constraint $L[\vv_h;\vw_h]$ is given in \eqref{linear-discrete}.
Note that $b_h$ is continuous with a continuity constant uniform in $h$
\begin{equation}\label{eq:bh-continuity}
|b_h(\vw_h;\vv_h,\vmu_h)| \lesssim \|\vw_h\|_{H^2_h(\Omega)}\|\vv_h\|_{H^2_h(\Omega)}\|\mu_h\|_{L^2(\Omega)},
\end{equation}
thanks to the inverse inequality $|L[\vv_h;\vw_h](x_T)| \lesssim h_T^{-1} \|L[\vv_h;\vw_h]\|_{L^2(T)}$ and the discrete Sobolev inequality $\| \nabla_h \vw_h \|_{L^4(\Omega)} \lesssim \| \vw_h \|_{H^2_h(\Omega)}$ valid for all $\vw_h \in [\mathbb V_h^k]^3$, see \cite[(6.9)]{bonito2020ldg-na}.
We also observe that $b_h(\vw_h; \vv_h,\vmu_h)=0$ for all $\vmu_h\in\Lambda_h$ implies $\vv_h\in\mathcal{F}_h(\vw_h)$ according to \eqref{linearized-cons}. Therefore, in each step of the gradient flow augmented with the linearized metric constraint, we seek $(\delta\vy_h^{n+1},\vla_h^{n+1})\in\V^k_h(\mathbf{0},\mathbf{0})\times \Lambda_h$ such that
\begin{equation}\label{gf:system-bi}
\begin{split}
  \tau^{-1}(\delta\vy_h^{n+1},\vv_h)_{H_h^2(\Omega)} \!& +\! a_h(\delta\vy_h^{n+1},\vv_h) \!+\! b_h(\vy^n_h;\vv_h,\vla_h^{n+1}) \\
  &+b_h(\vy^n_h;\delta\vy_h^{n+1},\vmu_h) \!=\! \ell[\vy_h^n](\vv_h) \!-\! a_h(\vy_h^n,\vv_h),
\end{split}
\end{equation}
for all $(\vv_h,\vmu_h)\in\V^k_h(\mathbf{0},\mathbf{0})\times \Lambda_h$. 

The proposed strategy is summarized in Algorithm \ref{algo:main_GF_bi}. \looseness=-1
\RestyleAlgo{boxruled}
\begin{algorithm}[htbp]
	\SetAlgoLined
	Given a pseudo-time step $\tau>0$ and a target tolerance $tol$\;
	Choose an initial guess $\vy_h^0\in\A_{h,0}$\;
	\While{$\tau^{-1}\big|E_h[\vy_h^{n+1}]-E_h[\vy_h^{n}]\big|>$tol}{
		\textbf{Solve} \eqref{gf:system-bi} for $(\delta\vy_h^{n+1},\vla_h^{n+1})\in\V^k_h(\mathbf{0},\mathbf{0})\times \Lambda_h$\;
		\textbf{Update} $\vy_h^{n+1} = \vy_h^{n}+\delta\vy_h^{n+1}$\;
	}
	\caption{(discrete-$H^2$ gradient flow with Lagrange multipliers)} \label{algo:main_GF_bi}
\end{algorithm}

It is worth pointing out that utilizing Lagrange multipliers is ubiquitous to enforce linearized metric constraints \cite{bartels2013approximation,bartels2018modeling,bartels2017bilayer,bartels2020stable,bonito2020ldg,bonito2020ldg-na,bonito2019dg,bonito2020discontinuous}.
In particular, the system \eqref{gf:system-bi} is solved using the \emph{Schur complement approach}, whose performance depends on  the inf-sup stability of $b_h$; see e.g. \cite{ntogkas2018non,berrone2019optimal} and refer to Section~\ref{s:implem} for additional details on the practical implementation.
Unfortunately, there are no results available in the literature guaranteeing a uniform inf-sup not even for the continuous problem.
In this section, we make a first step towards a better understanding of the situation in that we derive a sub-optimal estimate of the {discrete} inf-sup constant. 
We start with a linear algebra lemma.

\begin{lemma}[solvability of a matrix equation]\label{lem:solvability-matrix}
Given a $2\times2$ symmetric matrix $C$ and a full-rank  $3\times2$ matrix $B$, there exists a $3\times2$ matrix $A$ that solves the equation
\begin{equation}\label{matrix-eqn}
(A^TB+B^TA):C=|C|^2
\end{equation}
and satisfies $|A|\le\frac{|C|}{2\sigma_2(B)}$, where $|\cdot|$ denotes the Frobenius norm of matrices and
$\sigma_2(B)>0$ is the smallest singular value of $B$.
\end{lemma}
\begin{proof}
Using the cyclic properties of the trace operator yields
\begin{align*}
A^TB:C=\tr(B^TAC)=\tr(CA^TB)=\tr(A^TBC)=B^TA:C=A:BC,
\end{align*}
whence \eqref{matrix-eqn} is equivalent to
\begin{equation*}
A:BC=\frac{1}{2}|C|^2.
\end{equation*}
Let $B=U\Sigma V^T$ be the singular value decomposition of $B$, where $U\in\mathbb{R}^{3\times3}$
and $V=\mathbb{R}^{2\times2}$ are
orthogonal matrices, and $\Sigma=[\sigma_1(B),0;0,\sigma_2(B);0,0]\in\mathbb{R}^{3\times2}$ carries the singular values
$\sigma_1(B)\ge\sigma_2(B)\ge0$ of $B$. Since $B$ is full-rank, we deduce that $\sigma_2(B)>0$ and
\begin{align*}
|BC|^2=|U\Sigma V^TC|^2=|\Sigma\overline{C}|^2\ge\sigma_2(B)^2|\overline{C}|^2=\sigma_2(B)^2|C|^2,
\end{align*}
where $\overline{C}=V^TC$ and thus $|\overline{C}|=|C|$. We can now assume that $C\ne0$,
for otherwise $A=0$ solves \eqref{matrix-eqn}. We then realize that $|BC|>0$ and
\begin{equation*}
A=\frac{(BC)|C|^2}{2|BC|^2}
\end{equation*}
is clearly a solution to \eqref{matrix-eqn} as well as
\begin{equation*}
|A|=\frac{|C|^2}{2|BC|}\le\frac{|C|}{2\sigma_2(B)},
\end{equation*} 
which is the desired estimate $|A|\lesssim|C|$.
\end{proof}

The following sub-optimal estimate of the discrete inf-sup constant is a consequence of the previous lemma.
Since only the gradient of $\vv_h\in\V_h^k(\bz,\bz)$ appears in \eqref{def:bh-bi}, but the underlying norm of
$\V_h^k(\bz,\bz)$ is the discrete $H^2$-norm, it seems natural to consider a negative Sobolev norm of order $-1$ for
the space of Lagrange multipliers $\Lambda_h$. However, the fact that $\nabla_h\vv_h$ is
discontinuous makes it problematic to pair it with a distribution in a negative Sobolev space of order $-1$. 
This leads to the embedding of $\Lambda_h$ into $[L^2(\Omega)]^{2\times2}$, which is somehow responsible for
suboptimality.
  
\begin{thm}[discrete inf-sup constant]\label{inf-sup-bi}
For any $n\ge 0$ and $\vy_h^n\in\A_{h,\delta}$, there exists a constant $\beta$ independent of $n$ and $h$ such that
$\beta_h=\beta h_{\min}>0$ satisfies
\begin{equation} \label{eqn:inf_sup-bi}
		\inf_{\vmu_h\in\Lambda_h}\sup_{\vv_h\in\V_h^k(\bz,\bz)}\frac{b_h(\vy_h^n;\vv_h,\vmu_h)}{\|\vv_h\|_{H_h^2(\Omega)}\|\vmu_h\|_{L^2(\Omega)}} \ge \beta_h.
\end{equation}
\end{thm}
\begin{proof}
We proceed in two steps: we first construct a suitable $\vv_h$ and next show \eqref{eqn:inf_sup-bi}.

\smallskip
{\it Step 1: Construction of $\vv_h$.}
Given $\vmu_{h}\in\Lambda_h$, let $\vmu_{h,T}=\vmu_{h}\vert_{T}$ be the constant symmetric $2\times2$ restriction
of $\vmu_{h}$ to any element $T\in\Th$.
Thanks to Lemma \ref{lem:solvability-matrix} (solvability of a matrix equation),
there exists a $3\times2$ constant matrix $A_T$ such that
\begin{equation*}
L[A_T;\vy_h^n](x_T) : \vmu_{h,T} =
(A_T^T\nabla\vy_h^n(x_T)+\nabla\vy_h^n(x_T)^TA_T) : \vmu_{h,T} =| \vmu_{h,T}|^2,  
\end{equation*}
and $|A_T|\le\frac{|\vmu_{h,T}|}{2\sigma_{\min}(\nabla\vy_h^n(x_T))}$. Let $\lambda_{\min}:\mathbb{M}^{2\times2}\to\mathbb{R}$ be the smallest eigenvalue function defined over the space of symmetric matrices $\mathbb{M}^{2\times2}$ into $\mathbb{R}$, which turns out to be continuous with respect to any norm.
In particular, because $\vy_h^n\in\A_{h,\delta}$ we have
\begin{equation*}
  D_h[\vy_h^n](x_T) = 
  \big|\I[\vy_h^n](x_T)-I_2\big|\le\delta.
\end{equation*}
and there is a constant $c$ independent of $h$ and $n$ so that
$\big|\lambda_{\min}\big(\I[\vy_h^n](x_T)-I_2\big)\big|\le c\delta$, or
\begin{equation*}
\lambda_{\min}\big(\I[\nabla\vy_h^n](x_T)\big) \ge 1-c\delta,
\end{equation*}
Consequently, for $\delta$ sufficiently small we deduce
\[
\sigma_{\min}\big((\nabla\vy_h^n(x_T)\big) = \big(\lambda_{\min}\big(\I[\vy_h^n](x_T)\big)\big)^{\frac12}
\ge \big( 1-c\delta \big)^{\frac12},
\]
is bounded away from $0$ and we have $|A_T|\le\frac{|\mu_{h,T}|}{2(1-c\delta)^{\frac12}}$.
We finally define $\vv_h(\vx)\big|_T:=A_T(\vx-\vx_T)$ on each $T\in\Th$,
where $\vx_T$ is the barycenter of $T$, and observe that $\vv_h\in[\V^k_h]^3$
for $k\ge2$ and $\nabla\vv_h\big|_T=A_T$.

\medskip
{\it Step 2: Discrete inf-sup property.} We first compute
\begin{align*}
b_h(\vy_h^n;\vv_h,\vmu_h) = \sum_{T\in\Th}|T| \, L[\vv_h;\vy_h^n](x_T) : \vmu_{h,T}
=\sum_{T\in\Th}|\vmu_{h,T}|^2|T|=\|\vmu_h\|^2_{L^2(\Omega)}.
\end{align*} 
Since $D_h^2\vv_h=0$ for $\vv_h$ piecewise linear, combining a trace inequality with the Poincar\'e inequality on each element $T$ gives
\begin{align*}
\|\vv_h\|_{H^2_h(\Omega)}^2 &= \sum_{e\in\E_h^a}\|\h^{-3/2}\jump{\vv_h}\|_{L^2(e)}^2+\|\h^{-1/2}\jump{\nabla\vv_h}\|_{L^2(e)}^2\\
&\lesssim \sum_{T\in\Th}\h^{-4}\|\vv_h\|_{L^2(T)}^2+\h^{-2}\|\nabla\vv_h\|_{L^2(T)}^2 \lesssim
\sum_{T\in\Th}\h^{-2}\|\nabla\vv_h\|_{L^2(T)}^2
\end{align*}
due to the the fact that
$\vv_h\in\V^k_h(\bz,\bz)$ has vanishing mean value on $T$. Therefore,
\begin{equation}\label{eq:inf-sup-est-H2}
\begin{aligned}
\|\vv_h\|_{H^2_h(\Omega)}^2 
  \lesssim \sum_{T\in\Th}\h^{-2}\int_T|A_T|^2 \lesssim h_{\min}^{-2}(1-c\delta)^{-1}\|\vmu_h\|^2_{L^2(\Omega)}.
\end{aligned}
\end{equation}
In summary, we have shown that for every $\vmu_{h} \in \Lambda_h$,
there exists $\vv_h \in \V^k_h(\bz,\bz)$ such that
$b_h(\vy_h^n;\vv_h,\vmu_h) = \|\vmu_h\|^2_{L^2(\Omega)}$ and
$\|\vv_h\|_{H^2_h(\Omega)} \lesssim h_{\min}^{-1} \|\vmu_h\|_{L^2(\Omega)}$.
This is the desired inf-sup condition in disguised. 
\end{proof}

\section{Numerical experiments}\label{sec:numerical} 

In this section we present several numerical experiments, some motivated by computations \cite{bartels2017bilayer,bartels2018modeling,bartels2020stable,bonito2020discontinuous} and other by lab experiments \cite{alben2011edge,janbaz2016programming,menges2012material,krieg2016hygroskin,achimmenges,menges2015performative,reichert2015meteorosensitive,simpson2010capture}. We carry out simulations with several {\it spontaneous curvature} matrices $Z$ and both \emph{Dirichlet and free boundary conditions}, so as to capture a variety of insightful configurations exhibiting large bending deformations. We consider the effect of different aspect ratios of rectangular domains. We also explore properties of a novel model inspired by \cite{BBH2021Folding}, which allows folding across curved creases (bilayer origami). Our numerical simulations illustrate the computational performance of our algorithm.     

\subsection{Implementation}\label{s:implem}
We start with a few comments on the implementation of the gradient flow \eqref{gf:system-bi} and Algorithm \ref{algo:main_GF_bi}. 

\medskip\noindent
{\it Saddle-point structure.}  
We resort to a \emph{Schur complement method} to solve the discrete problem \eqref{gf:system-bi}. We refer to \cite{bonito2020ldg} for full implementation details of a similar linear algebra structure, but emphasize here how Theorem \ref{inf-sup-bi} (inf-sup stability) guarantees solvability and affects the solver efficiency in the spirit of \cite[Lemma 3.1]{berrone2019optimal}. 

To make explicit the Schur complement matrix and deduce its condition number, we denote by $\{\vvarphi_i\}_{i=1}^N$ a basis for $\V^k_h(\mathbf{0},\mathbf{0})$ and by $\{\vpsi_i\}_{i=1}^M$ an orthonormal basis for $\Lambda_h$.
The matrix representations of the bilinear forms $A_h(\cdot,\cdot):=\tau^{-1} (\cdot,\cdot)_{H^2_h(\Omega)} + a_h(\cdot,\cdot)$ and $b_h(\vy_h^n;\cdot,\cdot)$ used to define the gradient flow \eqref{gf:system-bi} are thus given by 
$$
  \vA:= \big( A_h(\vvarphi_j,\vvarphi_i)\big)_{i,j=1}^N, \qquad
  \vB_n:= \big( b_h(\vy_h^n;\vvarphi_j,\vpsi_i) \big)_{i=1,j=1}^{M,N}.
$$
With this notation, the Schur complement matrix reads $\vS_n:=\vB_n\vA^{-1}\vB_n^T$ and satisfies
\begin{equation*}
\begin{aligned}
(\vS_n\vm,\vm)&=(\vA^{-1/2}\vB_n^T\vm,\vA^{-1/2}\vB_n^T\vm)=\sup_{\vw\in\mathbb{R}^N}\left(\frac{(\vw,\vA^{-1/2}\vB_n^T\vm)}{\|\vw\|_2}\right)^2 \\
&=\sup_{\vv\in\mathbb{R}^N}\left(\frac{(\vv,\vB_n^T\vm)}{\|\vA^{1/2}\vv\|_2}\right)^2=\sup_{\vv_h \in \V^k_h(\mathbf{0},\mathbf{0})}\frac{b_h(\vy_h^n;\vv_h,\vmu_h)^2}{A_h(\vv_h,\vv_h)},
\end{aligned}
\end{equation*}
where  $\vmu_h := \sum_{i=1}^M m_i \vpsi_i \in \Lambda_h, \vv_h:=\sum_{j=1}^N v_j\vvarphi_j\in\V^k_h(\mathbf{0},\mathbf{0})$ with $\vm=(m_i)_{i=1}^M, \vv=(v_j)_{j=1}^N$, and $\|\cdot\|_2 = (\cdot,\cdot)^{1/2}$ is the Euclidean norm in $\mathbb R^N$.
On one hand, the continuity \eqref{eq:bh-continuity} of $b_h$ and the coercivity estimate \eqref{eqn:coercivity_bi} for $B_h[\vv_h]=\frac12a_h(\vv_h,\vv_h)$ yield
$$
(\vS_n\vm,\vm) \lesssim \tau\| \vy_h^n \|_{H^2_h(\Omega)}^2 \| \vmu_h \|_{L^2(\Omega)}^2 \lesssim \tau\| \vmu_h \|_{L^2(\Omega)}^2 = \tau\| \vm \|_2^2,
$$
because $\|\vy_h^n\|_{H^2_h(\Omega)}\lesssim 1$ in view of the energy stability \eqref{stability} satisfied by $\vy_h^n$ and the coercivity of total energy \eqref{e:coercfull}.
On the other hand, the inf-sup stability \eqref{eqn:inf_sup-bi} and the continuity estimate \eqref{eqn:coercivity_bi} for $B_h[\vv_h]=\frac12a_h(\vv_h,\vv_h)$ imply
$$
\tau h_{\min}^{2} \| \vm \|_2^2 = \tau h_{\min}^{2} \| \vmu_h \|^2_{L^2(\Omega)} \lesssim (\vS_n\vm,\vm).
$$
Combining these two inequalities yields an estimate for the condition number of $\vS_n$
\begin{equation}\label{eq:cond-num-est}
\kappa(\vS_n):=\max_{\vm\in\mathbb{R}^M}\frac{(\vS_n\vm,\vm)}{\|\vm\|_2^2}\left(\min_{\vm\in\mathbb{R}^M}\frac{(\vS_n\vm,\vm)}{\|\vm\|_2^2}\right)^{-1} \lesssim h_{\min}^{-2}.
\end{equation}

Estimate \eqref{eq:cond-num-est} shows that the saddle-point system is invertible but ill-conditioned.
We use a conjugate gradient (CG) iterative solver for the numerical experiments below. Classical convergence theory for CG asserts that the number of iterations to achieve a desired accuracy is of order $\sqrt{\kappa(\vS_n)}$ \cite[Theorem 3.1.1]{greenbaum1997iterative}.
Our numerical experiments reveal that the number of iterations needed in the CG solver roughly behaves like $h_{\min}^{-1}$, which is consistent with \eqref{eq:cond-num-est}. 

We emphasize that solving the linear system \eqref{gf:system-bi} by the Schur complement method for several steps of the gradient flow remains the bottle neck in terms of computing time.
We leave the design of suitable preconditioners open.

\medskip\noindent
{\it Assembly.}
Since the scalar product $\langle\cdot,\cdot\rangle_{H_h^2(\Omega)}$ and bilinear form $a_h$ do not change in the course of the gradient flow, we assemble them once for all before the main loop. In contrast, we assemble the bilinear form $b_h(\vy_h^n;\cdot,\cdot)$ and right hand side $\ell[\vy_h^n](.)$ at each step of the loop as they depend on the previous iterate $\vy^n_h$. Computing the discrete Hessian $H_h[\vy_h]$ is the most expensive part in the assembly process, as it requires solving the linear systems \eqref{def:lift_re} and \eqref{def:lift_be} for lifting operators. In order to save computing time, we find the discrete Hessian of each basis function at the beginning of the simulation and store its values for later use; this pre-processing drastically decreases the assembly time.   

\medskip\noindent
{\it Software and data.}
We implement our LDG method within the software platform \textrm{deal.ii} \cite{bangerth2007} and visualize
the outcome with \textrm{paraview} \cite{ayachit2015}. 
For all the simulations, we fix the polynomial degree $k$ of the deformation $\vy_h$ and the two liftings $l_1,l_2$ of the discrete Hessian $H_h[\vy_h]$, as well as the stabilization parameters $\gamma_1,\gamma_2$  to be
\[
k=l_1=l_2=2,
\quad
\gamma_0=\gamma_1=1.
\]
Recall that LDG is stable for any positive choice of parameters $\gamma_1,\gamma_2$, which contrasts with IPDG that requires  $\gamma_1,\gamma_2$ large for stability purposes \cite{bonito2019dg,bonito2020discontinuous}.

In the following numerical simulations, we consider both clamped Dirichlet ($\Gamma_D \ne \emptyset$) and free boundary conditions ($\Gamma_D= \emptyset$). For the latter, the discrete equation \eqref{gf_bilayer} is no longer well-defined. To fix the system kernel, we add an $L^2$-term to the metric $(\cdot,\cdot)_{H^2_h(\Omega)}$, while all other implementation aspects are similar to the case $\Gamma_D\ne\emptyset$. We refer to \cite{bonito2020ldg} for implementation details of free boundary conditions.

In either situation, a natural choice of initial deformation is that of a flat plate
\[
\vy_h^0(x_1,x_2)=(x_1,x_2,0) \quad\forall \, (x_1,x_2)\in\Omega,
\]
and satisfies clamped boundary conditions and the isometry constraint $\I[\vy_h^0]=I_2$ everywhere. This is much simpler than prestrained plates \cite{bonito2020ldg,bonito2020ldg-na}, which require preprocessing of both
boundary condition and metric constraint to construct suitable initial deformations for LDG to start.

\subsection{Clamped plate: Isotropic curvature}\label{sec:numerical-clamped}
We consider a rectangular plate $\Omega=(-5,5)\times(-2,2)$, clamped on the side $\{-5\}\times[-2,2]$, with isotropic spontaneous curvature $Z=I_2$. The deformation with minimal energy corresponds to a cylinder of radius $1$ and energy $20$ \cite{schmidt2007minimal}. This is confirmed by our simulations in Fig.~\ref{fig:bi-cylinder}, which displays iterations of the discrete gradient flow with {$1024$ elements} ($30720$ dofs), $\tau=5\times10^{-3}$ and $tol=10^{-4}$ in Algorithm \ref{algo:main_GF_bi}, as in \cite{bartels2017bilayer}.

We notice that surface self-intersecting develop during the relaxation dynamics of Algorithm \ref{algo:main_GF_bi};
this is similar to \cite{bartels2020stable} but different from \cite{bonito2020discontinuous}. Moreover, it takes fewer iterations for Algorithm \ref{algo:main_GF_bi} to reach the cylindrical equilibrium configuration, with the same or even smaller time step $\tau$, than FEMs in \cite{bartels2017bilayer,bartels2020stable,bonito2020discontinuous}.

Moreover, we keep the time step $\tau=5\times10^{-3}$ fixed and consider two quasi-uniform meshes with $256$ and $1024$ elements. We obtain bending energies $E_h=16.8627$ and $E_h=17.8038$ ($16\%$ and $11\%$ relative error) respectively, which exhibit smaller errors than the corresponding ones $E_h=15.961$ and $E_h=16.544$ with the Kirchhoff FEM of \cite{bartels2017bilayer} for the same mesh-size and time step. In addition, the energy error compares favorably with the new Kirchhoff FEM in \cite{bartels2020stable}, which computes with $Z=2.5I_2$ and produces a $36\%$ relative error even with a finer mesh of $5120$ triangular elements.         

\begin{figure}[htbp]
	\begin{center}
		\includegraphics[width=8.5cm]{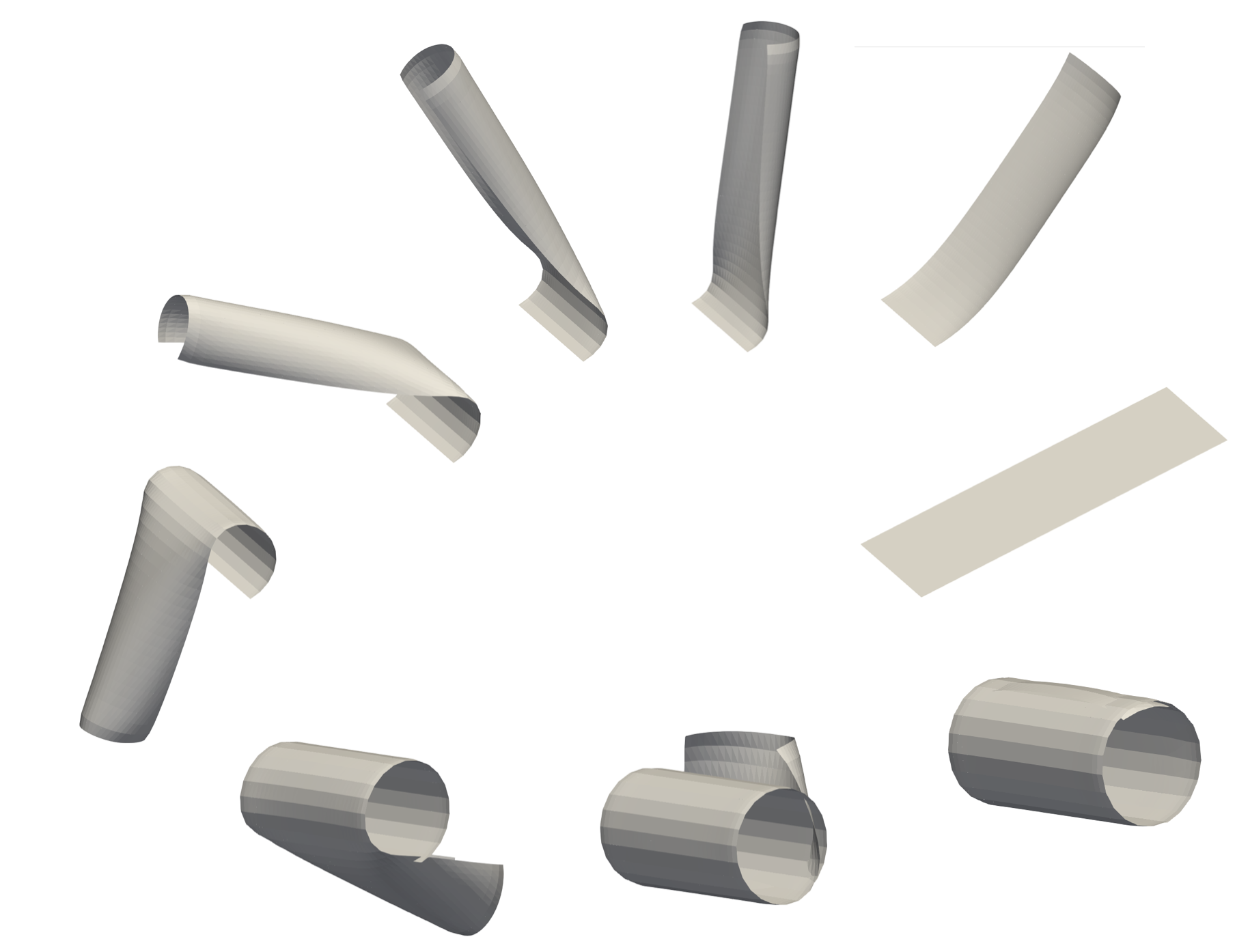}		
	\end{center}
        \caption[Bilayer plates: cylinder]{\small Isotropic curvature: Relaxation dynamics of Algorithm \ref{algo:main_GF_bi} towards the cylinder equilibrium shape of a clamped rectangular plate with the isotropic spontaneous curvature $Z = I$. The bilayer plate is depicted at times $0, 50, 1000, 9000, 18000, 36050, 48100, 56050, 72100$ of the gradient flow (counter-clockwise).}
        \label{fig:bi-cylinder}
\end{figure}

\subsection{Free plate: Anisotropic curvature}\label{S:anisotropy}
%
We now explore a cigar-type configuration motivated by experiments \cite{janbaz2016programming} and computations
\cite{bartels2017bilayer,bonito2020discontinuous}. The plate is again the rectangle $\Omega=(-5,5)\times(-2,2)$,
but now we impose no boundary condition (free boundary) along with the anisotropic spontaneous curvature
\begin{equation}\label{eq:anisotropy}
Z=
\begin{bmatrix}
~3 & -2 \\
-2 & ~3
\end{bmatrix}.
\end{equation}
We observe that the eigenpairs of $Z$ are $(1, [1,1]^T)$ and $(5,[1,-1]^T)$.
We thus expect that the plate deforms at $-45$ degrees with respect to the Cartesian axes in a symmetric way and eventually reaches a cigar-like configuration, as in \cite{bonito2020discontinuous}. We confirm this in Fig.~\ref{fig:bi-cigar}, that displays computations with $1024$ elements ($30720$ dofs) and $\tau=5\times10^{-3}$. The final energy is $E_h=46.3898$. Remarkably, Algorithm \ref{algo:main_GF_bi} takes fewer iterations to reach the equilibrium configuration than \cite{bonito2020discontinuous}.

\begin{figure}[htbp]
	\begin{center}
		\includegraphics[width=9cm]{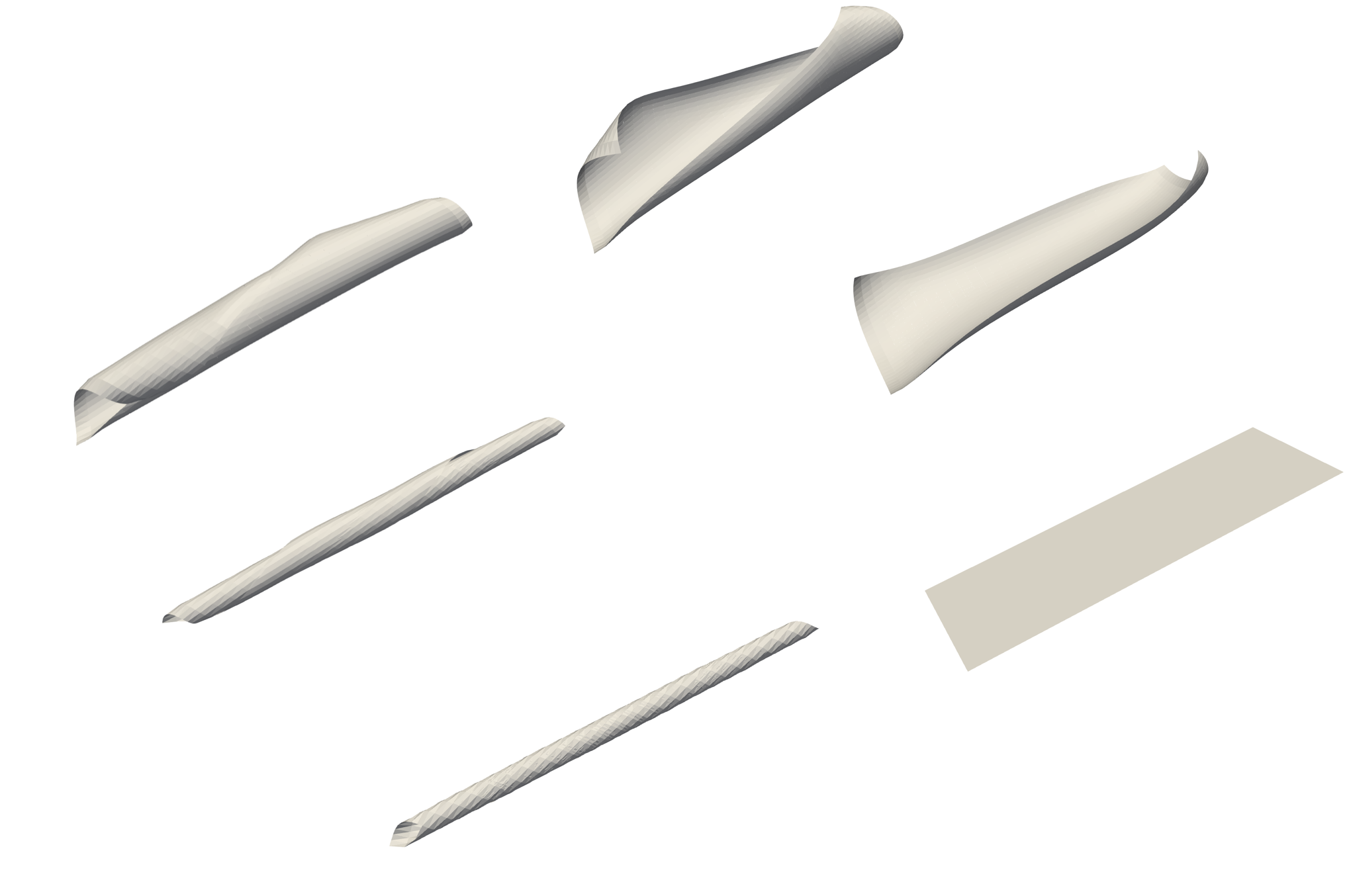}		
	\end{center}
        \caption[Bilayer plates: cigar]{\small Anisotropic curvature: Relaxation dynamics of Algorithm \ref{algo:main_GF_bi} towards the cigar-type equilibrium of a free rectangular plate with the anisotropic spontaneous curvature of \eqref{eq:anisotropy}. The bilayer plate is depicted at times $0, 50, 200, 1000, 10000, 30000$ of the gradient (counter-clockwise).}
        \label{fig:bi-cigar}
\end{figure}

\subsection{Free plate: Helix shape}\label{S:helix}
%
\begin{figure}[htbp]
  \begin{center}
		\includegraphics[width=8.5cm]{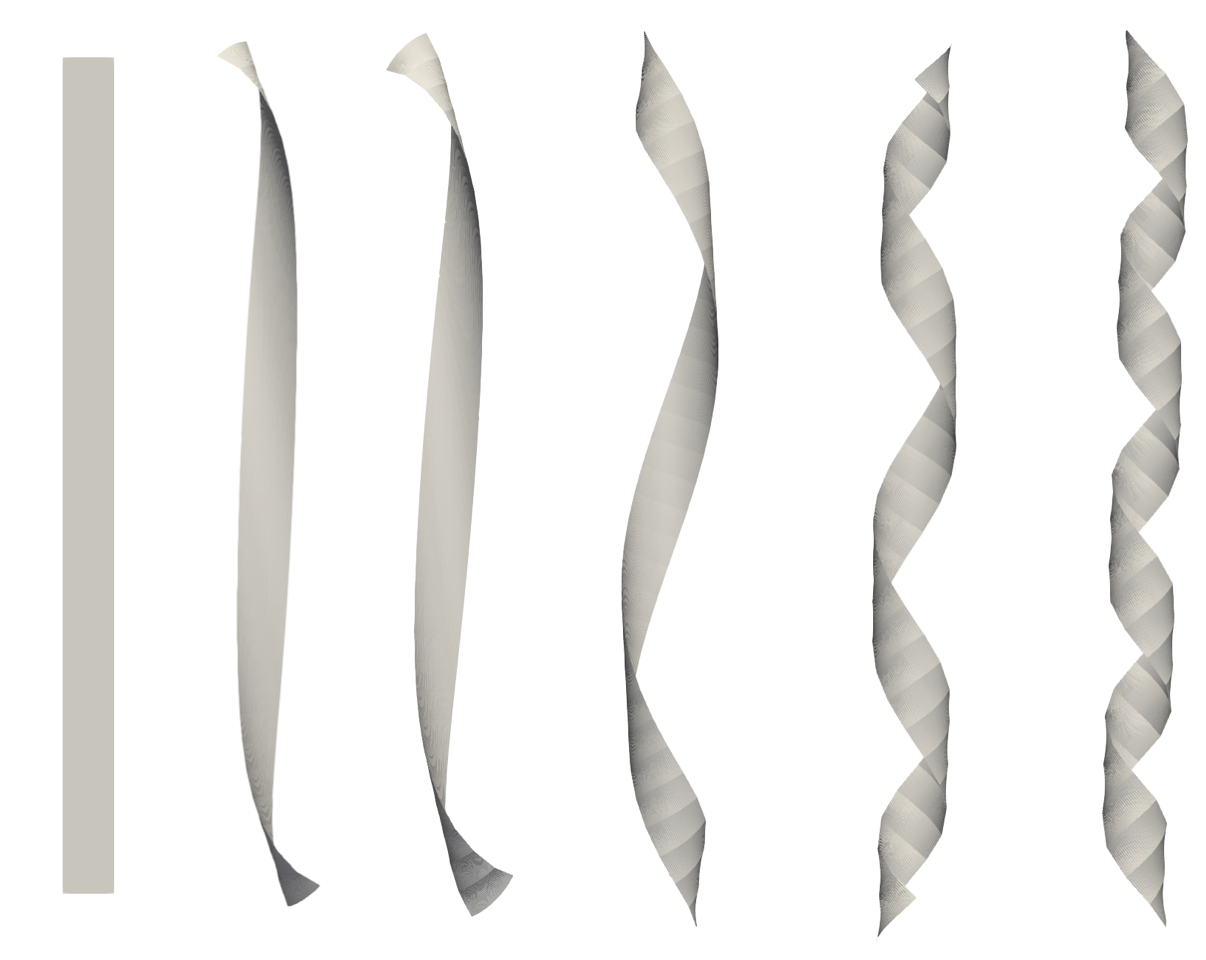}		
  \end{center}
  \caption[Bilayer plates: helix]{\small Anisotropic indefinite curvature: Relaxation dynamics of Algorithm \ref{algo:main_GF_bi}  with spontaneous curvature \eqref{eq:helix} towards a
  DNA-like equilibrium configuration of a free rectangular strip with large aspect ratio. The bilayer plate is depicted at times $0, 100, 200, 1000, 4000, 12600$ of the gradient flow (left to right).}
        \label{fig:bi-helix}
\end{figure}

We present a helix-type shape motivated by a DNA-like configuration \cite{simpson2010capture}. We consider a high aspect ratio plate $\Omega=(-8,8)\times(-0.5,0.5)$, with free boundary condition and anisotropic spontaneous curvature
\begin{equation}\label{eq:helix}
Z=
\begin{bmatrix}
1 & -3/2 \\
-3/2 & 1
\end{bmatrix}.
\end{equation}
We point out that the eigenpairs of $Z$ are $(-\frac12, [1,1]^T)$ and $(\frac52, [1,-1]^T)$, which again correspond to principal directions that form an angle of $45$ degrees with the coordinate axes. This, together with eigenvalues of opposite sign and high aspect ratio, leads to a deformation that resembles the twisting of DNA molecules, as in \cite{bonito2020discontinuous}. We display several snapshots of the relaxation dynamics of Algorithm \ref{algo:main_GF_bi} in {Fig}.~\ref{fig:bi-helix}. 
The simulation is carried out with $1024$ elements and $\tau=10^{-2}$, and yields a final energy $E_h=3.2507$.
Moreover, it again takes fewer iterations for LDG to reach the equilibrium configuration than the DG method of \cite{bonito2020discontinuous}.

\subsection{Climate responsive architectures}\label{S:climate}

Bilayer devices can be used to control the temperature or moisture inside a room.
The HygroSkin project \cite{menges2012material,menges2015performative,reichert2015meteorosensitive,krieg2016hygroskin,achimmenges} exploits this technology by designing visually appealing humidity responsive apertures to a pavilion.  
Heat and moisture are thus dynamically controlled without any high-tech equipment owing to the dominant orientation of fibers in plywood.

To simulate this device with our bilayer model, we consider an equilateral triangle with side length 1 and vertices $(0,0)$, $(1,0)$ and $(\frac 1 2, \frac{\sqrt{3}}{2})$.
The actual climate responsive device consists of 6 of these triangular shapes suitably rotated and arranged together as to form a flat regular hexagon, with the exterior edge of each triangle clamped; we refer to {Fig}.~\ref{f:climate_init}. To mimic the effect of different relative humidity values, we choose several anisotropic spontaneous curvatures
\begin{equation}\label{e:sc_open}
Z=\begin{pmatrix}
0 & 0\\
0 & \alpha
\end{pmatrix}, \qquad \textrm{with } \alpha = 0,1,2,3,4,5,
\end{equation}
for the triangle with exterior edge parallel to the $x$-axis and suitably rotated for the other triangles.
This matrix favors bending along the $y$-axis exclusively.

\begin{figure}[ht!]
\centerline{\includegraphics[height=3cm]{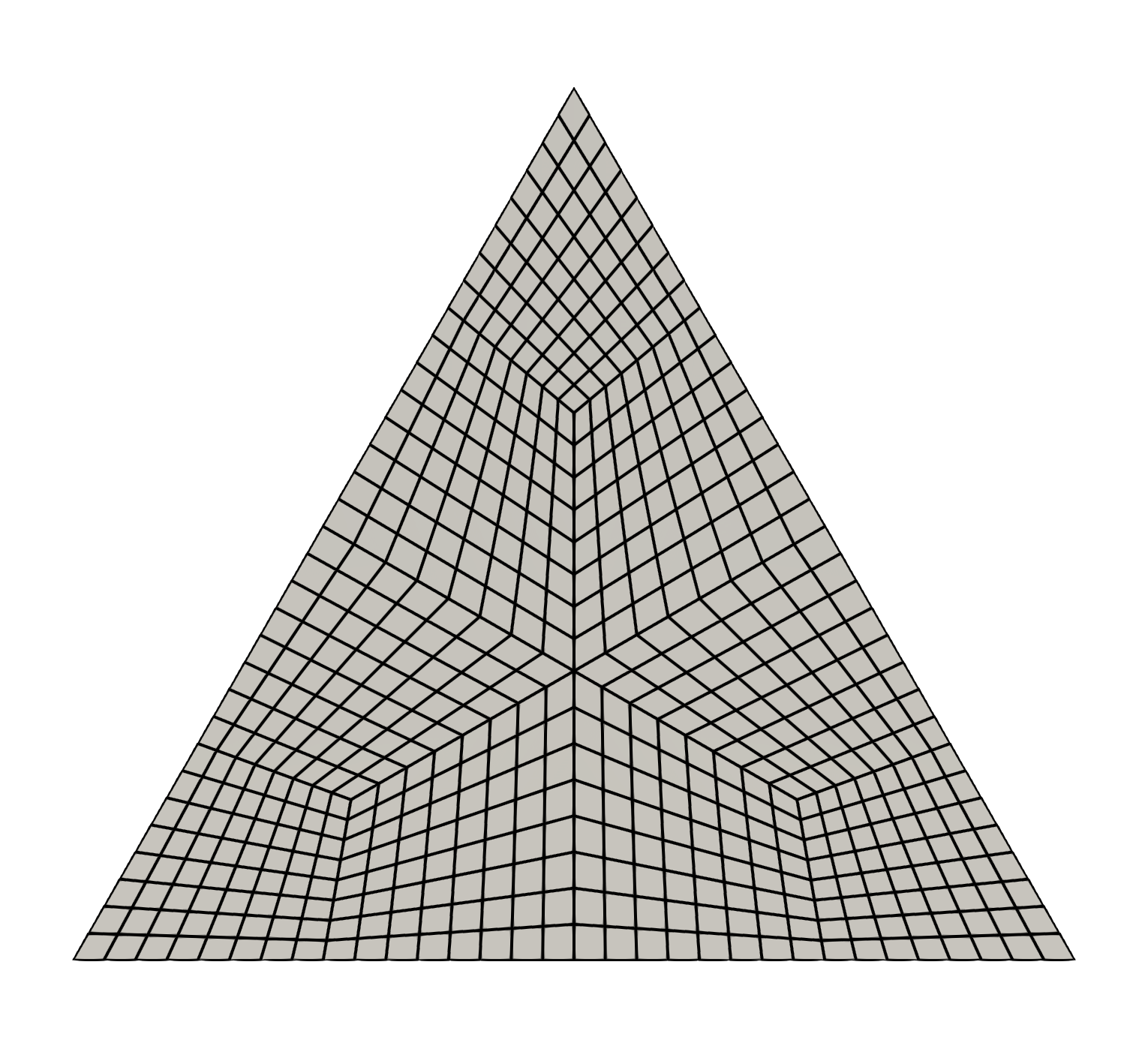}
\includegraphics[height=4cm]{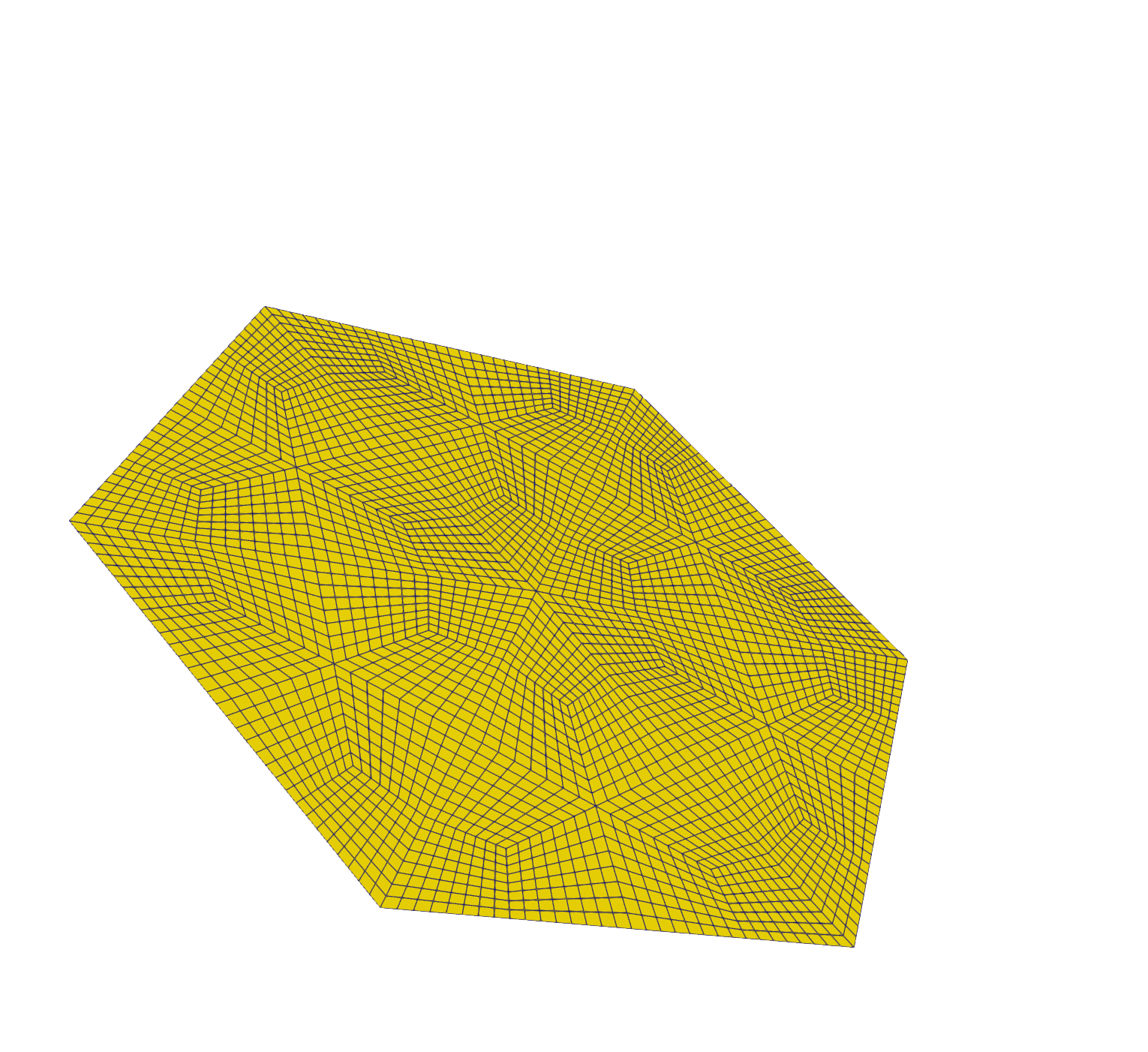}
}
\caption{\small Climate responsive device. The undeformed plate is made of 6 equilateral triangles that together form
a regular hexagon (right). Finite element partition of each triangle into trapezoids (left).}\label{f:climate_init}
\end{figure}

Upon actuation, the climate device automatically opens as depicted in {Fig}.~\ref{f:climate_open}.
The matching of the computed (left) and actual (right) equilibrium shapes in {Fig}.~\ref{f:climate_open}
is quite remarkable for a model with just one parameter $\alpha$ within $Z$.
We run this simulation with time step $\tau=1$ and stopping tolerance $tol=10^{-4}$.

\begin{figure}[ht!]
\begin{tabular}{c|c}
\includegraphics[width=0.45\textwidth]{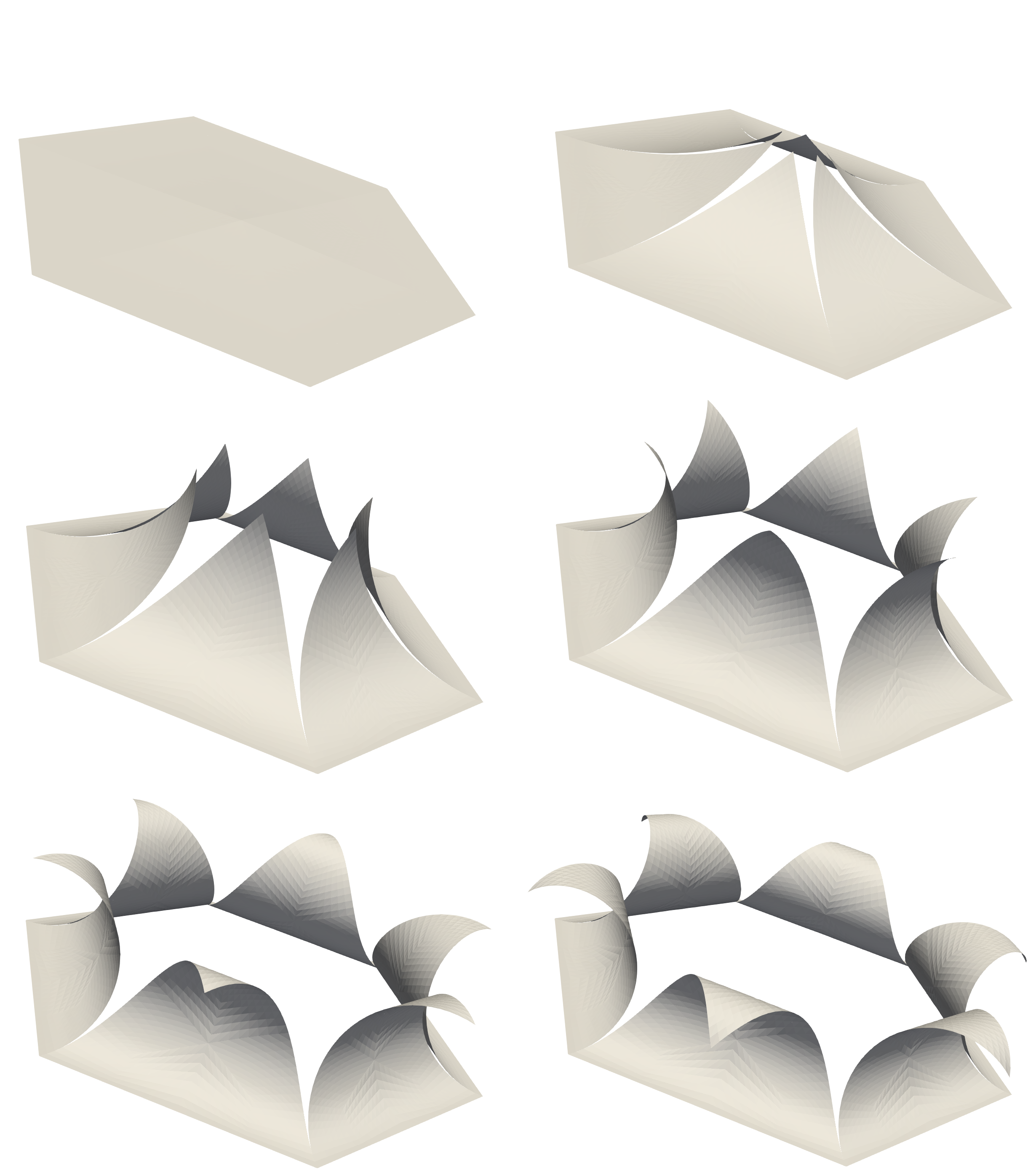} &
\includegraphics[width=0.35\textwidth]{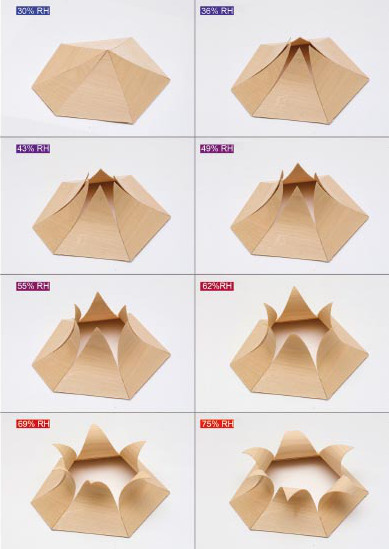}
\end{tabular}
\caption{\small Climate responsive device. (Left) Approximate deformation for different values of spontaneous curvature \eqref{e:sc_open} with parameter $\alpha = 0,1,2,3,4,5$. (from left to right and top to bottom) (Right) Experimental deformations of a device made of plywood. Picture taken from \cite{achimmenges} (courtesy of Prof. Achim Menges); see also \cite{menges2012material}. The matching is remarkable.} \label{f:climate_open}
\end{figure} 
  
\subsection{Folding Model: Bilayer Origami} \label{S:origami}

We finally explore computationally the combined effect of spontaneous curvature, as driving mechanism,
and folding across a preassigned crease. The corresponding bilayer model and LDG method are
discussed in Section \ref{S:creases}.
We consider below the setting from \cite[Section 5.2]{BBH2021Folding} and refer to \cite{bonito2022numerical} for additional numerical simulations.

\begin{figure}[h!]
\begin{tabular}{r|c|l}
\includegraphics[height=3.2cm]{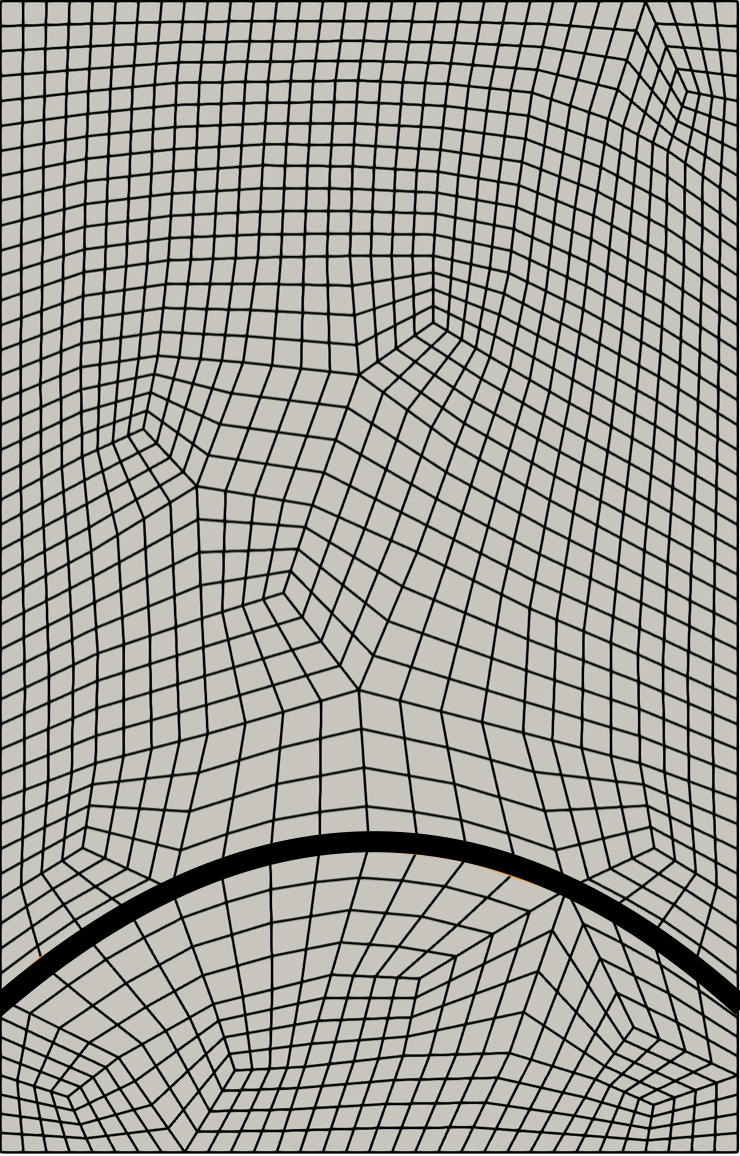}
&
\includegraphics[height=3.5cm]{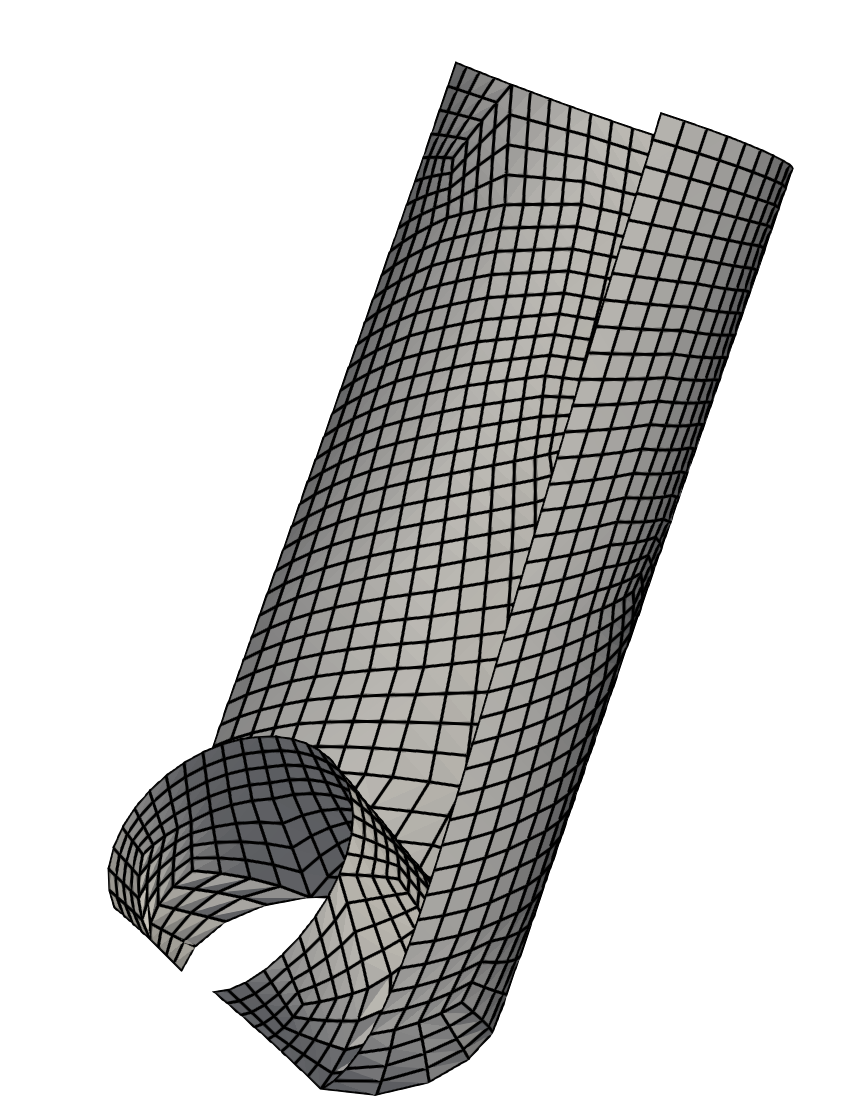}
\includegraphics[height=3.5cm]{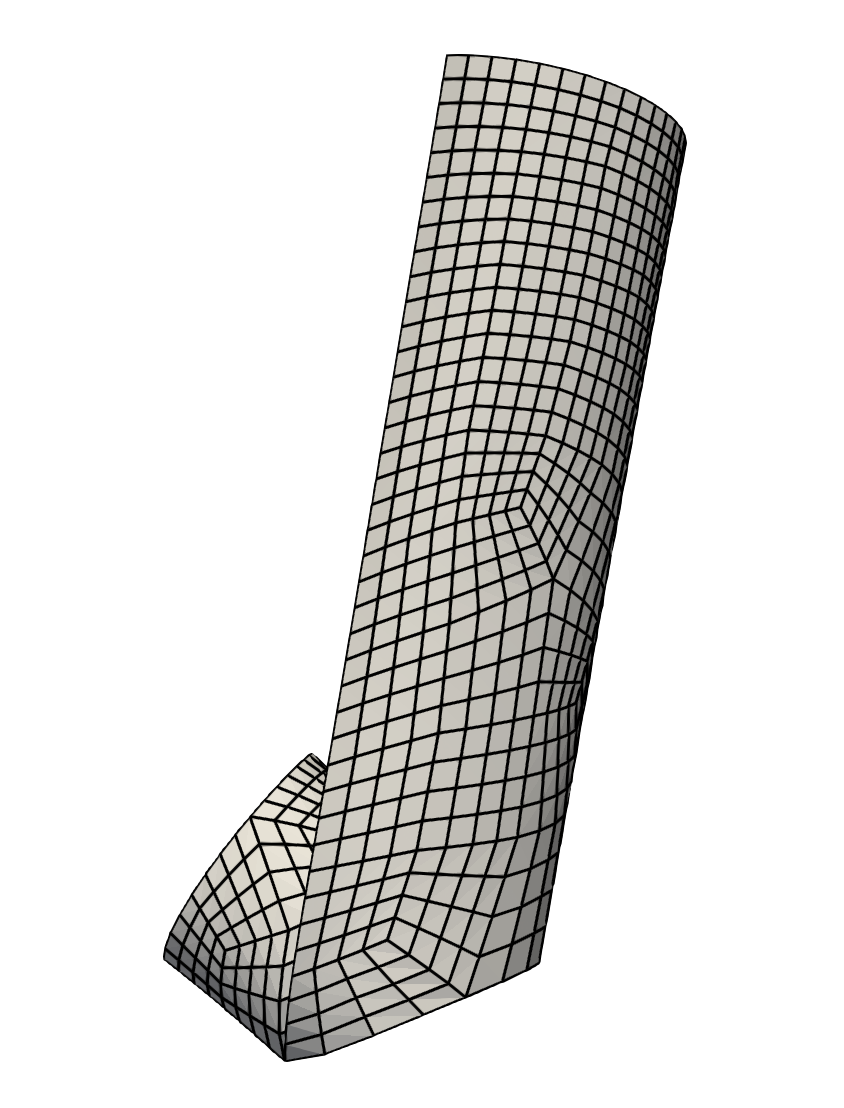}
\includegraphics[height=3.5cm]{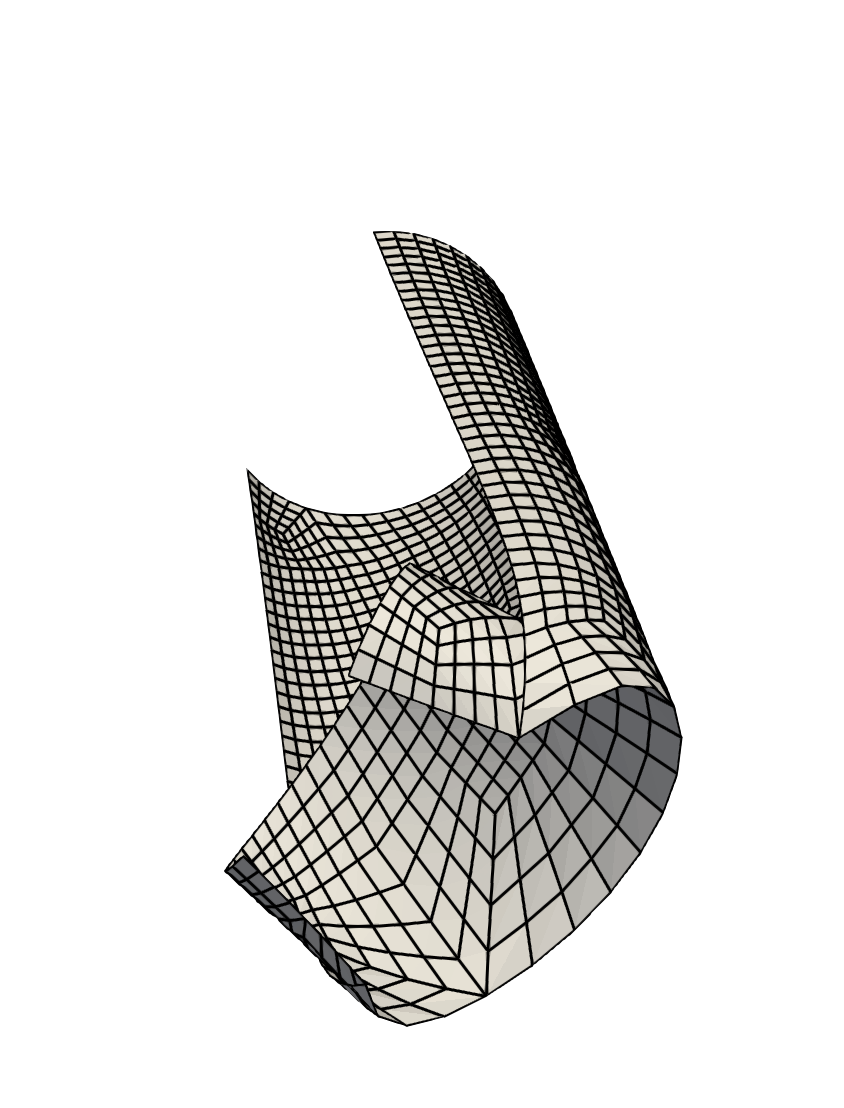}
&
\includegraphics[height=3.2cm]{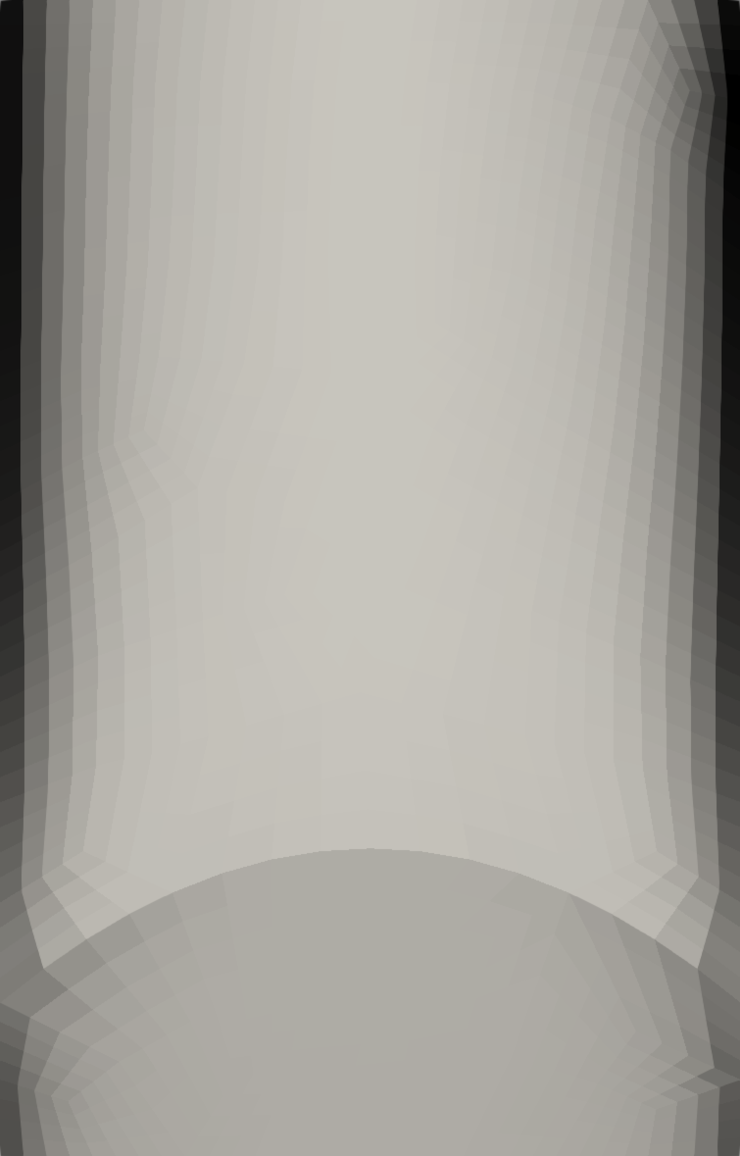}
\end{tabular}
\caption{\small {Bilayer origami: Flapping mechanism generated by folding of a bilayer plate across a crease. (Left) Conforming subdivision of $\Omega$ with quadratic crease. (Middle) Different perspectives of the resulting very large deformation. (Right) Isometry defect $D_h[\vy_h](x)$ ranging from $3.0 \times 10^{-5}$ (white) to $3.9 \times 10^{-2}$ (black).} }\label{f:folding}
\end{figure}

The computational domain is a rectangle  $\Omega=(0,9.6) \times (0,15)$ and the folding crease is a
quadratic curve $\Ccal$ passing through the points $(0,2)$, $(9.6,2)$, and $(4.8,6)$, which can be
exactly represented by the isoparametric mesh $\Th$; see Fig.~\ref{f:folding}.
In order to generate a configuration similar to the flapping mechanism in \cite{BBH2021Folding},
which is obtained by compression of the lateral boundary, we set the spontaneous curvatures
\[
Z = \begin{pmatrix} 0 & 0 \\ 0 & {\frac 1 2}\end{pmatrix},
\quad
Z = \begin{pmatrix} 0 & 0 \\ 0 & -\frac 12 \end{pmatrix},
\]
below the folding arc and above of it, respectively, and do not impose any boundary condition. The resulting equilibrium shape is displayed in Fig.~\ref{f:folding} {along with the isometry defect $D_h[\vy_h](x)$ at equilibrium; it ranges from $3.0 \times 10^{-5}$ to $3.9 \times 10^{-2}$.} We point out the crucial role played by the sign of principal curvatures $\lambda = {\frac12}, -\frac12$ corresponding to the same coordinate eigendirection: bending of the lower and upper plates occurs in opposite directions which gives rise to folding across the crease and yields a rather large compatible deformation.

\section{Conclusions}
In this article, we present a new LDG method for large bending isometric deformations of bilayer plates. We summarize our contributions in this section.  

\smallskip\noindent
1. \emph{LDG discretization.} It consists of replacing the Hessian $D^2\vy$ by a reconstructed Hessian $H_h[\vy_h]$ in the bending energy $B_h[\vy_h]$, and by a reduced (piecewise constant) discrete Hessian $\overline H_h[\vy_h]$ in the cubic energy $C_h[\vy_h]$, which encodes the interaction with spontaneous curvature. We use the mid-point quadrature to integrate $C_h[\vy_h]$.

\smallskip\noindent
2. \emph{{Relaxed} isometry constraint.} This allows for a slight violation of the isometry constraint $\I[\vy_h]=I_2$ while providing control of the $\ell^\infty$-norm of the isometry defect {$D_h[\vy_h]=\big|\I[\vy_h]-I_2\big|$} at element barycenters. This turns out to be a significant improvement over previous DG methods that enforce such defect as sum of averages over elements \cite{bonito2020ldg,bonito2020ldg-na,bonito2019dg,bonito2020discontinuous}.

\smallskip\noindent
3. \emph{$\Gamma$-convergence.} The key novelty of the $\Gamma$-convergence of discrete energies is the construction of the recovery sequence of any admissible deformation $\vy \in [H^2(\Omega) \cap W^1_\infty(\Omega)]^3$. It hinges on a quadratic Taylor expansion at element barycenters of a suitable regularization of $\vy$, and exploits that both the reduced quadrature of $C_h[\vy_h]$ and the isometry defect {$D_h[\vy_h]$} are imposed at element barycenters.

\smallskip\noindent
4. \emph{Bilayer model with foldings.} We extend the LDG method to deal with a piecewise quadratic crease $\mathcal{C}$  and prove its convergence. The construction of a recovery sequence for one absolute minimizer $\vy^* \in [H^2(\Omega\backslash\mathcal{C}) \cap W^1_\infty(\Omega)]^3$ requires the slightly stronger assumption that $\vy^*$ is $C^1$ in each subdomain created by $\mathcal{C}$.

\smallskip\noindent
5. \emph{Fully linear solver.} We design a semi-implicit discrete gradient flow that treats $B_h[\vy_h]$ implicitly and $C_h[\vy_h]$ explicitly. This leads to linear problems at each step. The scheme retains the key property of being energy diminishing and controls the isometry defect provided the fictitious time step $\tau$ satisfies a mild constraint.

\smallskip\noindent
6. \emph{Sub-optimal discrete inf-sup.} As is customary in the literature \cite{bartels2013approximation,bartels2018modeling,bartels2017bilayer,bartels2020stable,bonito2020ldg,bonito2020ldg-na,bonito2019dg,bonito2020discontinuous}, we rely on Lagrange multipliers to enforce the linearized isometry constraint. We prove a sub-optimal inf-sup condition for the resulting saddle-point system, which seems to be the first such result for these type of problems.

\smallskip\noindent
7. \emph{Simulations.} We present several insightful numerical experiments with large isometric deformations, including
a climate responsive device and the folding of a plate across a quadratic crease that yields a bilayer origami as equilibrium shape. {We also document the size of the isometry defect $D_h[\vy_h]$ for the latter.}

\bibliographystyle{acm}
\bibliography{ref}

\end{document}